\documentclass{article}

\usepackage{arxiv}

\usepackage{hyperref}

\usepackage{amsmath,amsfonts,amsthm}
\usepackage{array}
\usepackage{textcomp}
\usepackage{stfloats}
\usepackage{url}
\usepackage{verbatim}
\usepackage{graphicx}
\def\BibTeX{{\rm B\kern-.05em{\sc i\kern-.025em b}\kern-.08em
    T\kern-.1667em\lower.7ex\hbox{E}\kern-.125emX}}
\usepackage{balance,xcolor,hyperref}
\usepackage{algorithm, algpseudocode}
\usepackage{subcaption}
\usepackage{caption}
\usepackage{wrapfig}

\newtheorem{theorem}{Theorem}
\newtheorem{definition}{Definition}
\newtheorem{example}{Example}
\newtheorem{lemma}{Lemma}
\newtheorem{assumption}{Assumption}
\newtheorem{corollary}{Corollary}

\definecolor{cblue}{RGB}{16,78,139}
\definecolor{cred}{RGB}{139,37,0}
\definecolor{cgreen}{RGB}{0,139,0} 
\definecolor{babyblueeyes}{rgb}{0.63, 0.79, 0.95}
\definecolor{amethyst}{rgb}{0.6, 0.4, 0.8}

\usepackage{pifont}
\usepackage{upgreek}
\usepackage{diagbox}

\newcommand{\sqbox}[1]{\textcolor{#1}{\ding{110}}}%
\makeatletter
\newcommand{\thickhline}{%
    \noalign {\ifnum 0=`}\fi \hrule height 1pt
    \futurelet \reserved@a \@xhline
}
\newcolumntype{"}{@{\hskip\tabcolsep\vrule width 1pt\hskip\tabcolsep}}
\makeatother

\DeclareMathOperator{\Ima}{Im}
\newcommand{\eqdef}{\overset{def}{=}}
\DeclareMathOperator{\rank}{rank}
\DeclareMathOperator{\Tr}{Tr}

\newif\ifTR
\TRfalse

\begin{document}

\title{Reducing the dimensionality of data using tempered distributions}

\author{
Rustem Takhanov \\
 School of Sciences and Humanities\\
  Nazarbayev University\\
  53 Kabanbay Batyr Ave, Astana city \\
  \texttt{rustem.takhanov@nu.edu.kz}
}
\maketitle

\begin{abstract}
We reformulate unsupervised dimension reduction problem (UDR) in the language of tempered distributions, i.e. as a problem of approximating an empirical probability density function by another tempered distribution, supported in a $k$-dimensional subspace. 
We show that this task is connected with another classical problem of data science --- the sufficient dimension reduction problem (SDR). In fact, an algorithm for the first problem induces an algorithm for the second and vice versa. 

In order to reduce an optimization problem over distributions to an optimization problem over ordinary functions we introduce a nonnegative penalty function  that ``forces'' the support of the model distribution to be $k$-dimensional.
Then we present an algorithm for the minimization of the penalized objective, based on the infinite-dimensional low-rank optimization, which we call the alternating scheme.
Also, we design an efficient approximate algorithm for a special case of the problem, where the distance between the empirical distribution and the model distribution is measured by Maximum Mean Discrepancy defined by a  Mercer kernel of a certain type.
We test our methods on four examples (three UDR and one SDR) using synthetic data and standard datasets. 
\end{abstract}


\section{Introduction}
{\em Linear dimension reduction} (LDR) is a family of problems in data science that includes principal component analysis, factor analysis, linear multidimensional scaling, Fisher’s linear discriminant analysis, canonical
correlations analysis, sufficient dimension reduction (SDR), maximum autocorrelation factors, slow feature analysis and more. In unsupervised dimension reduction (UDR) we are given a finite number of points in ${\mathbb R}^n$ (sampled according to some unknown distribution) and the goal is to find a ``low-dimensional'' manifold (e.g. an affine or a linear subspace) that approximates ``the support'' of the distribution. UDR, historically, was approached by linear methods and, therefore, has developed into a set of standard tools in data science.
Though non-linear dimensionality reduction (a.k.a. the manifold learning) techniques gained a wide popularity in modern research, the potential of linear methods is far from being exhausted. For high-dimensional datasets, due to the phenomenon of concentration of measure~\cite{Johnson1986}, LDR often can give us an interpretable and low-dimensional representation of data. The linearity of a projection operator is a restrictive property that allows avoiding the overfitting in the dimension reduction (which is a key problem for the manifold learning techniques). 

The LDR study field currently achieved a saturation level at which unifying frameworks for the problem become of special interest. One of such frameworks, that covers many cases of LDR, frames LDR as the optimization task over matrix manifolds such as the Stiefel manifold~\cite{cunningham}. 
Elements of the Stiefel manifold $V_k({\mathbb R}^n)$ are orthogonal $k$-frames $O\in {\mathbb R}^{n\times k}$ whose column space   is the $k$-dimensional space onto which a dataset is projected. Different loss functions on $V_k({\mathbb R}^n)$ define different versions of LDR. Table 1 of~\cite{cunningham} lists fourteen common LDR techniques (such as principal component analysis, multi-dimensional scaling,  linear discriminant analysis etc), nine of which are formulated over Stiefel manifolds. Such a general treatment allows to approach all LDR problems by a single algorithm, i.e. by an adaptation of the gradient descent method to Stiefel manifolds~\cite{AbsilMahonySepulchre}. This adaptation consists of a series of  projected gradient steps where a common gradient descent is followed by a projection onto a Stiefel manifold, which is equivalent to the computation of a singular value decomposition of a current point.  
Note that the Stiefel manifold is a non-convex set and even minimizing a convex function on such a manifold is an NP-hard task, in general. Although, in applications, the projected gradient descent method demonstrated a relatively fast convergence to good quality solutions.


The paper's main contribution is a development of a novel view of LDR.
First, an argument over which we search in an optimization task is not a $k$-frame, but a tempered distribution (which is a generalization of a probabilistic distribution) that is concentrated on a $k$-dimensional linear subspace of ${\mathbb R}^n$. Thus, an argument has a more complex structure, it includes not just a $k$-dimensional subspace, but also a distribution on that subspace. 
The justification of our optimization framework uses the theory of generalized functions, or tempered distributions~\cite{Schwartz,Soboleff}. 
An important generalized function that cannot be represented as an ordinary function is the Dirac delta function, denoted $\delta$, and $\delta^n$ denotes its $n$-dimensional version.

This more general formulation allows us to analyze new types of objectives for LDR. In Section~\ref{examples} we list four examples of such objectives that, to our knowledge, have not been considered in the LDR field so far. A notable specifics of such objectives is that, even for a fixed $k$-dimensional subspace $\mathcal{L}$, finding  an optimal distribution supported in $\mathcal{L}$ is a non-trivial optimization task. In other words, our problems can not be simply reduced to the previous formalisms based on the Stiefel manifold, or the Grassmannian~\cite{WangQiong,Grassmannian}.

Let us briefly describe an optimization problem that motivates the new formalism.
Any dataset $\{{\mathbf x}_i\}_{i=1}^N \subseteq {\mathbb R}^n$ naturally corresponds to the distribution 
\begin{equation}
p_{{\rm emp}}({\mathbf x}) = \frac{1}{N} \sum_{i=1}^N \delta^n (\bold{x} - \bold{x}_i)
\end{equation}
which, with some abuse of terminology, can be called the empirical probability density function.
Based on that, UDR can be understood as a task whose goal is to approximate
$
p_{{\rm emp}}({\mathbf x})
$ by $q({\mathbf x})$, where $q({\mathbf x})$ is a distribution whose density is supported in a $k$-dimensional linear subspace $\mathcal{L}\subseteq {\mathbb R}^n$. Note that a function whose density is supported in some low-dimensional subset of ${\mathbb R}^n$ is not an ordinary function. An exact definition of a set of such distributions, denoted by $\mathcal{G}_k$, is given in Section~\ref{classes}.
To formulate an optimization task we additionally need a loss $D(p_{{\rm emp}},q)$ that measures the distance between the ground truth $p_{{\rm emp}}$ and a distribution $q$, that we search for. Thus, in our approach, the UDR problem is defined as
\begin{equation}\label{DRP}
I\left(q\right) = D\left(p_{{\rm emp}},q\right) \rightarrow\min_{q\in \mathcal{G}_k}
\end{equation}
under the condition that $q({\mathbf x})$ has a $k$-dimensional support. In most of our statements we do not consider any specific loss functions, though in our basic examples we deal with the Maximum Mean Discrepancy distance or the Wasserstein distance.

{\bf The UDR and SDR.} Within our formalism the sufficient dimension reduction problem is tightly connected with the UDR problem. In the SDR, given supervised data, the goal is to find the so called effective subspace, defined by its orthogonal basis (or, a $k$-frame) $\{{\mathbf w}_1, \cdots, {\mathbf w}_k\}\subseteq {\mathbb R}^n$ , such that the regression function can be searched in the form $g({\mathbf w}^T_1{\mathbf x}, \cdots, {\mathbf w}^T_k{\mathbf x})$. In literature, these functions are known under different
names, e.g. functions with low effective dimensionality~\cite{WangZiyu}, functions with active subspaces~\cite{ConstantinePaul}
and multi-ridge functions~\cite{Fornasier,TYAGI2014389}.
In~\cite{Meihong} it was shown that a method originally developed for the SDR can be turned into a UDR method, i.e. applied to unsupervised data, by simply setting an output to be equal to an input. 
In such methods for the SDR problem as the Sliced Inverse Regression~\cite{Li91}, the Principal Hessian Direction~\cite{Li2}, the Sliced Average Variance Estimation~\cite{SaveCook}, an effective subspace is recovered from the Singular Value Decomposition applied to a certain matrix that is constructed from a training set in a straightforward way. Other methods, such as the Principal Fitted Components~\cite{Forzani}, the Likelihood Acquired
Direction~\cite{CookLiliana}, the Kernel Dimensionality Reduction~\cite{Fukumizu}, are based on analytic expressions measuring the affinity of a $k$-dimensional subspace to the effective subspace. The second type of methods reduce the SDR problem to an optimization problem over the Stiefel manifold, or the Grassmanian. For other methods we refer to a tutorial on SDR methods~\cite{Ghojogh}. Again, an important aspect of all these methods is that, given a fixed effective subspace, the regression function that predicts an output variable has a relatively straightforward structure and is not optimized by any additional supervised learning procedure. The key novelty that our framework brings to the SDR is that we suggest to search for an effective subspace and a regression function in a joint manner.

The key observation of our analysis, stated in Theorem~\ref{g-dual-f}, is that a class of functions of the form $g({\mathbf w}^T_1{\mathbf x}, \cdots, {\mathbf w}^T_k{\mathbf x})$ can be characterized as functions whose Fourier transform is supported in the corresponding effective subspace. In other words, functions with an effective dimensionality $k$ are dual to $\mathcal{G}_k$ under the Fourier transform. Three examples of UDR problems that  we give in Section~\ref{examples} are cast as~\eqref{DRP}, whereas in the fourth example we formulate SDR as an optimization task with the search space dual to that of UDR (to distinguish our formulation from a general SDR problem we call it an SDR with optimized regression function). Thus, all four examples can be studied within our optimization framework.

Besides the problem setup we also suggest a general algorithm that tackles it. 
The basic idea of that algorithm, which we call the alternating scheme, instead of optimizing over $\mathcal{G}_k$, to optimize over ordinary functions with a penalty added to an objective that forces the ordinary function's support to be low-dimensional.

{\bf The penalty based reformulation.}  
The starting point of our approach is to reduce the task~\eqref{DRP} to the minimization of $I\left(q\right)+\lambda R(q)$ over ordinary functions $q$. We define the penalty function $R(q)$ in such a way that forcing $R(q)$ to be small is equivalent to forcing ``the support'' of $q$ to be $k$-dimensional. 
Our definition of $R$ is based on using a positive definite kernel $M: {\mathbb R}^n\times {\mathbb R}^n\to {\mathbb C}$. 

First we note that $M$ defines a billinear form on pairs of (possibly, generalized) functions by $\langle f|M|g\rangle = \int_{{\mathbb R}^n\times {\mathbb R}^n} f({\mathbf x})^\ast M({\mathbf x}, {\mathbf y}) g({\mathbf y})d{\mathbf x} d{\mathbf y}$. On a properly defined space of (generalized)  functions, the billinear form $\langle \cdot| M |\cdot\rangle$ is the hermitian inner product, using which one can define distances and other geometrical notions on that space. Note that if $f$ and $g$ are probability density functions and $M$ is real-valued, the corresponding distance function, i.e. ${\rm dist}_M(f,g)=(\langle f|M|f\rangle-\langle g|M|g\rangle-2\langle f|M|g\rangle)^{1/2}$ coincides with the maximum mean discrepancy metric~\cite{MMD}.
We define $R(q)$ as 
\begin{equation}
R(q) = \sum_{i=k+1}^n \lambda_i(M_q)
\end{equation}
where 
$M_{q} = {\rm Re}\begin{bmatrix}
\langle x_i q({\mathbf x})| M | x_j q({\mathbf x})\rangle
\end{bmatrix}_{i,j=\overline{1,n}}$
and $\lambda_1(M_q)\geq \lambda_2(M_q)\geq \cdots $ are ordered eigenvalues of the matrix $M_q$. The sum of all but first $k$ eigenvalues of a positive semidefinite matrix $A$ is a well-known penalty function, denoted by $\|A\|_{n-k}$ and called a Ky Fan $n-k$-antinorm. Applications of the Ky Fan $n-k$-antinorm to low-rank optimization problems can be found in~\cite{TruncatedNN,OhKweon,LiuJin,HONG2016216} and its properties are studied in~\cite{HIAI20131568}.

Thus, we reduce the task~\eqref{DRP} to
\begin{equation}\label{DRP2}
I\left(q\right) +\lambda \|M_q\|_{n-k} \rightarrow\min_{q}
\end{equation}
over ordinary functions. An analysis that we make in Subsection~\ref{generaltheory} of Section~\ref{HowOpt} (based on theory of tempered distributions) shows that if the kernel $M$ is chosen from a class of so called proper kernels and the solution of~\eqref{DRP} satisfies certain regularity conditions, the solution of the task~\eqref{DRP2} for $\lambda\to +\infty$ will approach the solution of~\eqref{DRP}.

{\bf The alternating scheme.} The task~\eqref{DRP2} can be understood as an infinite dimensional low-rank optimization task in which the penalty term forces the matrix $M_q$ to be of rank $k$. In Section~\ref{mod-alt-right} we prove that $M_q = S_qS_q^\dag$ where $S_q$ is a linear operator between a suitable space $\mathcal{H}$ and ${\mathbb R}^n$ that itself depends on $q$ linearly, and this automatically gives us that $R(q) = \min_S\|S_q-S\|^2_\ast$ where the minimum is taken over all operators between $\mathcal{H}$ and ${\mathbb R}^n$ of rank $k$ and $\|\cdot\|_\ast$ is a suitable norm on the space of bounded linear operators from $\mathcal{H}$ to ${\mathbb R}^n$.

Then, a natural idea to solve the task~\eqref{DRP2} is to present it as the joint minimum $\min\limits_{q}\min\limits_{S: {\rm rank\,}S=k} I\left(q\right) +\lambda \|S_q-S\|^2_\ast$ and to minimize the objective over $q$ and over $S$ of rank $k$ in an alternating fashion, i.e.
\begin{equation}\label{DRP3}
\begin{split}
q_{l+1} = \arg\min_{q} I\left(q\right) +\lambda \|S_q-S_{l}\|^2_\ast, \\
S_{l+1} = \arg\min_{S: {\rm rank\,}S=k} \|S_{q_{l+1}}-S\|^2_\ast.
\end{split}
\end{equation}
This algorithm, called the alternating scheme, is suitable for a practical implementation due to the fact that the second step of it, i.e. the optimization over $S$ of rank $k$, is solvable analytically. In fact, $S_{l+1}$ is the Singular Value Decomposition of $S_{q_{l+1}}$ truncated at $k$-th term.
In~\ref{mod-alt-right2} we give an algorithm whose every step is equivalent to a corresponding step of the alternating algorithm, but it operates on Fourier transforms of functions rather than on functions of initial coordinates. Numerical specifications of the alternating scheme for different special cases of UDR/SDR problems are given in~\ref{numerical-mmd}, ~\ref{numerical-hm}, ~\ref{numerical-wd} and~\ref{numerical-sdr}. In Section~\ref{SDR-exp} we describe results of our experiments with the alternating scheme that we conducted for various synthetic and practical datasets. As a result we conclude that the alternating scheme is a practical algorithm that can be applied to datasets of moderate size. For the SDR tasks its performance is comparable with classical algorithms.

{\bf An approximate algorithm for a special case.} For a special case of the task~\eqref{DRP}, where the distance function is the Maximum Mean Discrepancy with the kernel of the form $K({\mathbf x}, {\mathbf y}) = ({\mathbf x}\cdot {\mathbf y})H({\mathbf x}, {\mathbf y}) $ and $H$  is itself a Mercer kernel,
we develop an approximate algorithm that can be applied to large datasets. In Section~\ref{Approximate-MMD-PCA} we demonstrate that a solution with provable approximation ratios is given by the following simple procedure: given a dataset $\{{\mathbf x}_i\}_{i=1}^N$ we build a data matrix $X = [{\mathbf x}_1, \cdots, {\mathbf x}_N]$, a Gram matrix $G = [H({\mathbf x}_i, {\mathbf x}_j)]$, and output first $k$ principal components of the matrix $XGX^T$. This algorithm is tested on Yale B dataset for the shadow/black removal and SBMnet datasets for the background modeling. In both applications, our approximate algorithm showed a performance comparable to the performance of other low-rank approximation algorithms.

{\bf The structure of the paper} is as follows. In Section~\ref{prel} we give some notations and define standard notions from functional analysis that we use throughout the paper.
In Section~\ref{classes} we formally define the search space in Problem~\ref{DRP}, denoted $\mathcal{G}_k$, and an image of $\mathcal{G}_k$ under the Fourier transform, denoted $\mathcal{F}_k$.  In Section~\ref{examples} we formulate some UDR/SDR problems as optimization tasks over $\mathcal{G}_k$/$\mathcal{F}_k$.  
Instead of searching directly in a set of generalized functions, $\mathcal{G}_k$,  
in Section~\ref{HowOpt} we describe how we substitute an ordinary function for a distribution in the optimization task at the expence of adding a new penalty term to its objective, $\lambda R(f)$. Using a kernel $M({\mathbf x}, {\mathbf y})$, Theorem~\ref{rank-k} characterizes generalized $g\in\mathcal{G}_k$ as such $g$ for which the matrix of properly defined integrals $M_g = {\rm Re}\begin{bmatrix}
\iint_{{\mathbb R}^n\times {\mathbb R}^n} x_iy_j g({\mathbf x})^\ast M({\mathbf x}, {\mathbf y}) g({\mathbf y}) d {\mathbf x} d {\mathbf y}
\end{bmatrix}_{i,j=\overline{1,n}}$ is of rank $k$.
In Section~\ref{mod-alt-right} we suggest a method for solving $\min_{\phi} I(\phi)+\lambda R(\phi)$ which we call {\em the alternating scheme}. In Section~\ref{Approximate-MMD-PCA} we describe a simple approximate algorithm for the task~\eqref{DRP} in a special subcase of the Maximum Mean Discrepancy distance and prove some theoretical guarantees on the approximation ratio of this algorithm.
 Section~\ref{SDR-exp} is dedicated to experiments with the alternating scheme on synthetic and real world data and with the approximate algorithm on the shadow/black removal and the background modeling applications. Proofs of all theorems and lemmas are given after their formulations, or can be found in the appendix.  
\subsection{Related work}
As was already mentioned, another unifying framework for LDR tasks is suggested by~\cite{cunningham} in which the basic search space is the Stiefel manifold $V_k({\mathbb R}^n)$.
The main advantage of the Stiefel manifold over ${\mathcal G}_k$ is that its elements are finite-dimensional. Because a distribution from ${\mathcal G}_k$ is an infinite-dimensional object, an optimization over ${\mathcal G}_k$ requires additional constructions to turn it into a finite-dimensional task. Both an optimization over ${\mathcal G}_k$ and over $V_k({\mathbb R}^n)$ is typically hard: for a final point, at best one can guarantee that it is a local extremum. Promising aspects of ${\mathcal G}_k$ are: a) ${\mathcal G}_k$ allows to formulate a new class of objectives naturally on it, b) local extrema on ${\mathcal G}_k$ substantially differ from local extrema on $V_k({\mathbb R}^n)$, because a local search over ${\mathcal G}_k$ uses more degrees of freedom.

There is plenty of literature on the SDR problem some of which was already mentioned. In \cite{archaic} the Fourier transform was applied for estimating the effective subspace in SDR, implicitly using an analog of Theorem~\ref{g-dual-f}.
The closest to ours is a recent approach of~\cite{KAPLA2022107390}, where an effective subspace was computed in a two step process. First, given supervised data, a regression function was trained in the form of a neural network (with a general architecture), then the obtained regression function was approximated by another neural network with a bottleneck architecture (by which a low effective dimensionality is guaranteed by construction). Like in this approach, we train a regression function as a neural network, though we search for it and an effective subspace jointly. In our approach, it is a regularization term $R(f)$ that forces the neural network to have a low effective dimensionality.

Using Ky Fan $k$-antinorm as a regularizer for the matrix completion problem has been suggested by~\cite{TruncatedNN} and further developed in~\cite{OhKweon,LiuJin,HONG2016216}. Unlike this chain of works, we formulate an infinite-dimensional task and our regularizer $R(f)=\Vert M_f\Vert _{n-k}$ is a sum of smallest $n-k$ squared singular values of the infinite-dimensional operator $S_f$ where $S_f$ depends on $f$ linearly and $M_f = S_f S_f^\dag$. Thus, our algorithms are substantially different from algorithms designed within the latter approach. The idea of alternating two basic stages, convex optimization and SVD, is ubiquitous in low-rank optimization, see e.g.~\cite{JMLRmazumder10a,JMLRhastie15a}.

\section{Preliminaries and notations}\label{prel}
Throughout this paper we use standard terminology and notation from functional analysis. 
For details one can address the textbook on the theory of distributions~\cite{friedlander1998introduction}. The Schwartz space, denoted by $\mathcal{S}({\mathbb R}^n)$, is a space of infinitely differentiable functions $f: {\mathbb R}^n\rightarrow {\mathbb C}$ such that $\forall \alpha, \beta \in{\mathbb N}^n, \sup_{{\mathbf x}\in\mathbb{R}^n} $ $\vert {\mathbf x}^\alpha D^\beta f({\mathbf x}) \vert <\infty $, and equipped with standard topology.
Its dual space is denoted by $\mathcal{S'}({\mathbb R}^n)$ and is equipped with weak topology. 
For a tempered distribution $T\in \mathcal{S'}({\mathbb R}^n)$ and $\phi\in \mathcal{S}({\mathbb R}^n)$, $\langle T, \phi\rangle$ denotes $T(\phi)$. Thus, for a sequence $\{f_s\}\subseteq \mathcal{S'}({\mathbb R}^n)$ and $f\in \mathcal{S'}({\mathbb R}^n)$, $\lim_{s\rightarrow \infty}f_s = f$ (or, $f_s\rightarrow^\ast f$) means that $\lim_{s\rightarrow \infty}\langle f_s, \phi\rangle = \langle f, \phi \rangle$ for any $\phi\in \mathcal{S}({\mathbb R}^n)$.  For a sequence $\{f_s\}^\infty_{s=1}\subseteq {\mathcal S}' ({\mathbb R}^n)$, $\mathop{\rm Lim}\limits_{s\rightarrow \infty} f_s$ denotes a set of points $f\in {\mathcal S}' ({\mathbb R}^n)$, such that there exists a growing sequence $\{s_i\}\subseteq {\mathbb N}$ and $\lim_{i\rightarrow \infty} f_{s_i} = f$. 
The Fourier and inverse Fourier transforms are denoted by $\mathcal{F}, \mathcal{F}^{-1}: \mathcal{S'}({\mathbb R}^n)\rightarrow \mathcal{S'}({\mathbb R}^n)$. For brevity, we denote $\mathcal{F}[f]$ by $\hat f$.  
If all required conditions are satisfied, an integrable $f: {\mathbb R}^n \rightarrow {\mathbb C}$ (or, a Borel measure $\mu$ on ${\mathbb R}^n$) is used as the tempered distribution $T_f$ (or, $T_\mu$) where $\langle T_{f},\phi\rangle = \int_{{\mathbb R}^n} f({\mathbf x})\phi({\mathbf x}) d {\mathbf x}$ (or, $\langle T_{\mu}, \phi\rangle = \int_{{\mathbb R}^n} \phi({\mathbf x}) d \mu$). For $\Omega\subseteq \mathcal{S}({\mathbb R}^n)$, $\overline{\Omega}$ denotes the sequential closure of $\Omega$ with respect to standard topology of $\mathcal{S}({\mathbb R}^n)$. For $\Omega\subseteq \mathcal{S'}({\mathbb R}^n)$, $\overline{\Omega}^\ast$ denotes the sequential closure of $\Omega$ with respect to weak topology of $\mathcal{S'}({\mathbb R}^n)$. 
For $\psi\in \mathcal{S}({\mathbb R}^n), T\in \mathcal{S'}({\mathbb R}^n)$, the convolution is defined as a tempered distribution $\psi\ast T$ such that $\langle \psi\ast T, \phi\rangle = \langle T,\tilde{\psi}\ast \phi\rangle$ where $\tilde{\psi} ({\mathbf x}) = \psi(-{\mathbf x})$. If $T\in \mathcal{S'}({\mathbb R}^n)$ and a function $\psi$ is such that $\psi({\mathbf x})\phi({\mathbf x})\in \mathcal{S}({\mathbb R}^n)$ whenever $\phi({\mathbf x})\in \mathcal{S}({\mathbb R}^n)$, then the multiplication $\psi T$ is defined by $\langle \psi T, \phi\rangle = \langle T, \psi \phi\rangle$. 
Given a measure $\mu$, by $L_{2, \mu}({\mathbb R}^n)$ we denote the complex $L_2$-space with the inner product $\langle u, v \rangle_{L_{2,\mu}} = \int u ({\mathbf x})^\ast v({\mathbf x}) d \mu$. The induced norm is then $\Vert {\mathbf u}\Vert _{L_{2, \mu}} = \sqrt{\langle {\mathbf u},{\mathbf u} \rangle_{L_{2, \mu}} }$. If $d\mu = p({\mathbf x})d {\mathbf x}$, then $L_{2, \mu}$ is denoted by $L_{2, p}$.
A set of infinitely differentiable functions in ${\mathbb R}^n$ is denoted by $C^\infty ({\mathbb R}^n)$. A set of infinitely differentiable functions with compact support in ${\mathbb R}^n$ is denoted by $C_c^\infty ({\mathbb R}^n)$.  
If $T$ is a topological space, then a subset $S\subseteq T$ is said to be dense in $T$ if the sequential closure of $S$ is equal to $T$. 
For a square matrix $A$, ${\rm Tr}(A)$ denotes its trace and for an arbitrary matrix, $\Vert A\Vert _F \eqdef \sqrt{{\rm Tr}(A^T A)}$.
The identity matrix of size $n$ is denoted by $I_n$. The notation $f\propto g$ means $f=cg$ where $c$ is some universal constant.

\section{Basic function classes}\label{classes}
To formalize distributions supported in a $k$-dimensional subspace, we need a number of standard definitions~\cite{cmouhot}. For $\phi_1\in {\mathcal S}({\mathbb R}^{k})$ and $\phi_2\in {\mathcal S}({\mathbb R}^{n-k})$, their tensor product is the function $\phi_1\otimes \phi_2\in {\mathcal S}({\mathbb R}^{n})$ such that $(\phi_1\otimes \phi_2) ({\mathbf x}, {\mathbf y}) = \phi_1({\mathbf x}) \phi_2({\mathbf y})$. The span of $\{\phi_1\otimes \phi_2\vert  \phi_1\in {\mathcal S}({\mathbb R}^{k}), \phi_2\in {\mathcal S}({\mathbb R}^{n-k})\}$, denoted by ${\mathcal S}({\mathbb R}^{k})\otimes {\mathcal S}({\mathbb R}^{n-k})$, is called the tensor product of ${\mathcal S}({\mathbb R}^{k})$ and ${\mathcal S}({\mathbb R}^{n-k})$. For $g_1\in {\mathcal S}'({\mathbb R}^{k})$ and $g_2\in {\mathcal S}'({\mathbb R}^{n-k})$, their tensor product is defined by the following rule: $\langle g_1\otimes g_2, \phi_1\otimes \phi_2\rangle = \langle g_1, \phi_1\rangle \langle g_2, \phi_2\rangle$ for any $\phi_1\in {\mathcal S}({\mathbb R}^{k}), \phi_2\in {\mathcal S}({\mathbb R}^{n-k})$. Since $\overline{{\mathcal S}({\mathbb R}^{k})\otimes {\mathcal S}({\mathbb R}^{n-k})} = {\mathcal S}({\mathbb R}^{n})$, there is only one distribution $g_1\otimes g_2\in {\mathcal S}'({\mathbb R}^{n})$ that satisfies the identity.

An example of a generalized function, whose density is concentrated in a $k$-dimensional subspace, is any distribution that can be represented as
$
g\otimes \delta^{n-k} \eqdef g\otimes \underbrace{\delta \otimes \cdots \otimes \delta}_{\text{$n-k$ times}}
$
where $g\in {\mathcal S}'({\mathbb R}^{k})$. If $g = T_f$, where $f: {\mathbb R}^{k}\rightarrow {\mathbb R}$ is an ordinary function, then $g\otimes \delta^{n-k}$ can be understood as a generalized function whose density is concentrated in a subspace $\{{\mathbf x}\in {\mathbb R}^n\vert  x_i=0, i>k\}$ and equals $f({\mathbf x}_{1:k})$. It can be shown that the distribution acts on $\phi\in {\mathcal S}({\mathbb R}^{n})$ in the following way:
\begin{equation}
\langle T_f\otimes \delta^{n-k}, \phi \rangle = \int_{{\mathbb R}^{k}} f({\mathbf x}_{1:k}) \phi({\mathbf x}_{1:k}, {\mathbf 0}_{n-k}) d {\mathbf x}_{1:k}
\end{equation}
Now to generalize the latter definition to any $k$-dimensional subspace we have to introduce a change of variables in tempered distributions. 

Let $g\in {\mathcal S}'({\mathbb R}^{n})$ and $U\in {\mathbb R}^{n\times n}$ be an orthogonal matrix, i.e. $U^TU=I_n$. Then, $g_U \in {\mathcal S}'({\mathbb R}^{n})$ is defined by the rule: $\langle g_U, \phi\rangle  = \langle  g, \psi\rangle$ where $\psi({\mathbf x}) = \phi(U^T {\mathbf x})$. If $g = T_f$, the latter definition gives $g_U = T_{f'}$ where $f'({\mathbf x})= f(U{\mathbf x})$. Now, we define classes of tempered distributions:
\begin{equation}
\mathcal{G}'_k = \{(f\otimes \delta^{n-k})_U \vert  f\in {\mathcal S}'({\mathbb R}^{k}), U\in {\mathcal O}(n)\},
\end{equation}
\begin{equation}
\mathcal{G}_k = \left\{(T_f\otimes \delta^{n-k})_U \vert  f\in {\mathcal S}({\mathbb R}^{k}), U\in {\mathcal O}(n)\right\},
\end{equation}
and
\begin{equation}
\begin{split}
\mathcal{F}_k = \{T_r\mid   r({\mathbf x}) = f(U {\mathbf x}), f\in {\mathcal S}({\mathbb R}^{k}), 
U\in {\mathbb R}^{k\times n}, \rank(U) = k\} 
\end{split}
\end{equation}
where ${\mathcal O}(n) = \{U\in {\mathbb R}^{n\times n} \mid    U^TU=I_n\}$.
The first two classes are related as:
\begin{theorem}\label{easy}
$\mathcal{G}'_k = \overline{\mathcal{G}_k}^\ast$.
\end{theorem}
The last two classes are isomorphic under the Fourier transform. 
\begin{theorem}\label{g-dual-f}
$\mathcal{F}[\mathcal{G}_k] = \mathcal{F}_k$ and $\mathcal{F}^{-1}[\mathcal{F}_k] = \mathcal{G}_k$.
\end{theorem}
\begin{proof}
Let us prove first that if $g = T_f\otimes \delta^{n-k}$, then 
$\mathcal{F}[g] = T_r$,
where $r({\mathbf x}) = \hat f ({\mathbf x}_{1:k}), {\mathbf x}\in {\mathbb R}^n$. For that we have to prove that $\langle \mathcal{F}[g], \phi\rangle = \langle T_r, \phi\rangle$ for any $
\phi\in {\mathcal S}({\mathbb R}^{n})$. Indeed,
\begin{equation}
\begin{split}
\langle \mathcal{F}[g], \phi\rangle = \langle g, \mathcal{F} [\phi]\rangle = 
\langle T_f\otimes \delta^{n-k}, \int_{{\mathbb R}^n}\phi({\mathbf y})e^{-\mathrm{i}{\mathbf x}^T {\mathbf y}}d {\mathbf y}\rangle  = \\
\langle T_f, \int_{{\mathbb R}^n}\phi({\mathbf y})e^{-\mathrm{i}{\mathbf x}_{1:k}^T {\mathbf y}_{1:k}}d {\mathbf y}\rangle = 
\int_{{\mathbb R}^{n+k}}f({\mathbf x}_{1:k})\phi({\mathbf y})e^{-\mathrm{i}{\mathbf x}_{1:k}^T {\mathbf y}_{1:k}}d {\mathbf y} d {\mathbf x}_{1:k} =\\
\int_{{\mathbb R}^n}\hat f({\mathbf y}_{1:k})\phi({\mathbf y})d {\mathbf y} = \langle T_{r}, \phi\rangle .
\end{split}
\end{equation}
Let us calculate the image of $\mathcal{G}_k$ under the Fourier transform. It is easy to see that for any $g\in {\mathcal S}'({\mathbb R}^{n}), \phi \in {\mathcal S}({\mathbb R}^{n})$ and orthogonal $U\in {\mathbb R}^{n\times n}$ we have: 
\begin{equation}
\begin{split}
\langle \mathcal{F}[g_U], \phi({\mathbf x}) \rangle = \langle g_U, \mathcal{F} [\phi]({\mathbf x})\rangle = \langle g, \mathcal{F} [\phi](U^T{\mathbf x})\rangle  =  \\ 
\langle g, \mathcal{F} [\phi(U^T{\mathbf x})]\rangle  = \langle \mathcal{F}[g], \phi(U^T{\mathbf x})\rangle  =  \langle (\mathcal{F}[g])_U, \phi({\mathbf x})\rangle .
\end{split}
\end{equation}  
Therefore, $\mathcal{F}[g_U]  = (\mathcal{F}[g])_U$.
Thus, if $g = T_f\otimes \delta^{n-k}$, then 
\begin{equation}
\begin{split}
\left(\mathcal{F}[g_U]\right) = (T_r)_U = T_{r'}
\end{split}
\end{equation} 
where $r'({\mathbf x}) = r(U{\mathbf x}) = \hat f (U_k {\mathbf x})$ and $U_k\in {\mathbb R}^{k\times n}$ is a matrix consisting of first $k$ rows of $U$. Thus, $T_{r'}\in \mathcal{F}_k$. 

Let us show that by varying $f\in {\mathcal S}({\mathbb R}^{k})$ and $U$ in the expression $\hat f (U_k {\mathbf x})$ we can obtain any function from $\mathcal{F}_k$. For this, it is enough to show that $\mathcal{F}_k$ is equivalent to the following set of functions:
$$
\mathcal{Q} = \{g(U_k{\mathbf x})\vert  g\in {\mathcal S}({\mathbb R}^{k}), U_k \in {\mathbb R}^{k\times n}, U_k U_k^T = I_k\}
$$
The fact $\mathcal{Q}\subseteq \mathcal{F}_k$ is obvious.
Let us now prove that $\mathcal{Q} \supseteq \{g(P{\mathbf x})\vert  g\in {\mathcal S}({\mathbb R}^{k}), P\in {\mathbb R}^{k\times n}, \rank P = k\} = \mathcal{F}_k$. Indeed, if $f({\mathbf x}) = g(P{\mathbf x})$, then $f({\mathbf x}) = g'(U_k{\mathbf x})$ where $U_k = (PP^T)^{-1/2}P$ and $g'({\mathbf y}) = g((PP^T)^{1/2}{\mathbf y})$. By construction, $U_kU_k^T=I_k$ and $g'\in {\mathcal S}({\mathbb R}^{k})$. Thus, $\mathcal{Q}=\mathcal{F}_k$.

Therefore, $\mathcal{F}[\mathcal{G}_k] = \mathcal{F}_k$, and from the bijectivity of the Fourier transform we obtain $\mathcal{F}^{-1}[\mathcal{F}_k] = \mathcal{G}_k$.
\end{proof}


For any collection $f_1, \cdots, f_l\in  {\mathcal S}'({\mathbb R}^{n})$, ${\rm span}_{{\mathbb R}} \{f_i\}^l_{1}$ denotes $\{\sum_{i=1}^l \lambda_i f_i\vert  \lambda_i\in {\mathbb R}\}\subseteq {\mathcal S}'({\mathbb R}^{n})$, which is a linear space over ${\mathbb R}$.
The set $\mathcal{G}'_k$ has the following simple characterization:
\begin{theorem}\label{low-rank-char} For any $T\in {\mathcal S}'({\mathbb R}^{n})$, $T\in \mathcal{G}'_k$ if and only if 
\begin{equation}
\begin{split}
\dim {\rm span}_{{\mathbb R}} \{x_1 T, x_2 T, \cdots, x_n T\}\leq k.
\end{split}
\end{equation}
\end{theorem}
{\em Informally}, the theorem holds because any linear dependency  $\alpha_1 x_1 T+ \cdots + \alpha_n x_n T = 0$ over ${\mathbb R}$ implies that if $\alpha_1 x_1 + \cdots + \alpha_n x_n  \ne 0$, then $T=0$. This is equivalent to a statement that the support of $T$ is concentrated on a subspace $\alpha_1 x_1 + \cdots + \alpha_n x_n  = 0$. If $\dim {\rm span}_{{\mathbb R}} \{x_1 T, x_2 T, \cdots, x_n T\}\leq k$, then one can find $n-k$ such dependencies, which means that the support of $T$ is $k$-dimensional.

Let ${\mathcal B}({\mathbb R}^n)$ denote the Borel sigma-algebra on ${\mathbb R}^n$ and $\mathcal{P}$ denote a set of all Borel probability measures on ${\mathbb R}^n$. Let us now define 
\begin{equation}
\begin{split}
\mathcal{P}_k = 
\{\mu \in \mathcal{P}  \vert \exists{\mathbf v}_1, \cdots, {\mathbf v}_k\in {\mathbb R}^n,  
\forall A\in {\mathcal B}({\mathbb R}^n): 
\mu (A) = \mu(A\cap{\rm span}({\mathbf v}_1, \cdots, {\mathbf v}_k))\}
\end{split}
\end{equation}
i.e. $\mathcal{P}_k$ is a set of probability measures with all probability concentrated in some subspace ${\rm span}({\mathbf v}_1, \cdots, {\mathbf v}_k)$ whose dimension is not greater than $k$. It is easy to see that $T_\mu\in \mathcal{G}'_k$ for any $\mu\in \mathcal{P}_k $. 

\section{Examples of LDR formulations}\label{examples}
{\bf Maximum mean discrepancy PCA (MMD-PCA)}
Let $K:{\mathbb R}^n\times {\mathbb R}^n\to {\mathbb R}$ be a continuous Mercer kernel, and $\mathcal{H}_K$ be a reproducing kernel Hilbert space (RKHS) defined by $K$.
The kernel $K({\mathbf x}, {\mathbf y})$ defines the so-called kernel embedding of probability measures $\phi$~\cite{KernelMean}:
\begin{equation}
\mu \in \mathcal{P} \mathop\rightarrow\limits^{\phi} {\mathbb E}_{{\mathbf y}\sim \mu}K({\mathbf x},{\mathbf y})
= \int K({\mathbf x},{\mathbf y}) d\mu({\mathbf y}).
\end{equation}
The Maximum Mean Discrepancy (MMD) distance~\cite{MMD} is defined as the distance induced by metrics on
$\mathcal{H}_K$, i.e. for two probability measures $\mu, \nu \in \mathcal{P}$,
\begin{equation}
d_{\textsc{MMD}}(\mu, \nu) = \Vert \phi(\mu) -\phi(\nu)\Vert _{\mathcal{H}_K}.
\end{equation}

Let ${\mathbf x}_1, \cdots, {\mathbf x}_N\in {\mathbb R}^n$ be the dataset of points.
This dataset defines the empirical probabilistic measure $\mu_{\rm{data}}$ that corresponds to the tempered distribution 
$
T_{\mu_{\rm{data}}} = \frac{1}{N} \sum_{i=1}^N \delta^n({\mathbf x}-{\mathbf x}_i)
$.
We shall study a method concurrent to PCA that is based on solving the following problem:
\begin{equation}\label{MMD-task}
\min_{\nu\in \mathcal{P}_k} d_{\textsc{MMD}}(\mu_{\rm{data}}, \nu) = \min_{\nu\in \mathcal{P}_k} 
\Vert \phi(\mu_{\rm{data}}) -\phi(\nu)\Vert _{\mathcal{H}_K} 
\end{equation}
i.e. we shall attempt to approximate the empirical probabilistic measure $\mu_{\rm{data}}$ with another probabilistic measure $\nu$ which is supported in some $k$-dimensional subspace of ${\mathbb R}^n$. 
To our knowledge, the task~\eqref{MMD-task} has not been yet considered in the research field of LDR.
\begin{example}[Gaussian MMD-PCA]\label{observation} Let $k({\mathbf x}) = G^n_h ({\mathbf x})$ 
where
$G^n_h ({\mathbf x}) = \frac{e^{-\frac{\| {\mathbf x}\|^2}{2h^2}}}{(2\pi h^2)^{n/2}}$
is the radial Gaussian kernel on  ${\mathbb R}^n$ and $K({\mathbf x}, {\mathbf y}) = (k\ast k)({\mathbf x}-{\mathbf y}) = G^n_{2h} ({\mathbf x})$. For such a kernel, we have
\begin{equation}
d_{\textsc{MMD}}(\mu, \nu) = \Vert \psi(\mu) -\psi(\nu)\Vert _{L_2({\mathbb R}^n)},
\end{equation}
where $\psi(\mu) = \int k({\mathbf x}-{\mathbf y}) d\mu({\mathbf y})$ is just a smoothing of the distribution $\mu$ via the Weierstrass trasform.

In this example, as $h\to +0$, the optimal measure $\nu^\ast = \arg\min_{\nu\in  \mathcal{P}_k} \Vert \psi(\mu_{\rm{data}}) -\psi(\nu)\Vert _{L_2({\mathbb R}^n)}$ is supported in a $k$-dimensional subspace that contains the largest possible number of points from $\{{\mathbf x}_1, \cdots, {\mathbf x}_N\}$. 
Khachiyan demonstrated~\cite{Khachiyan} that the following problem is NP-hard: given $\{{\mathbf x}_1, \cdots, {\mathbf x}_N\}\subseteq {\mathbb R}^n$, find an $n-1$-dimensional subspace of ${\mathbb R}^n$ that contains at least $(1-\varepsilon)(1-\frac{1}{n})N$ points from the dataset. This indicates that in the regime $h\to +0$, the task~\eqref{MMD-task} is NP-hard. In other words, it is unlikely that the task admits an efficient algorithm, in general. In~\ref{numerical-mmd} we describe an algorithm for the Gaussian MMD-PCA.
\end{example}

\ifTR
{\em The dual form.} Let us define another Gaussian kernel $\gamma({\mathbf x}) =  e^{-\frac{h^2\vert {\mathbf x}\vert ^2}{2}} = \mathcal{F}[k]$. 
Let $p_{\rm{data}}({\mathbf x})$ denote the characteristic function of the random vector ${\mathbf X}_{\rm{data}}\sim\mu_{\rm{data}}$. By definition, $p_{\rm{data}}({\mathbf x}) = {\mathbb E}[e^{{\rm i}{\mathbf X}_{\rm{data}}^T{\mathbf x}}] = \frac{1}{N}\sum_{i=1}^N e^{{\rm i}{\mathbf x}^T_i{\mathbf x}}$.
Thus, $p_{\rm{data}} \propto \mathcal{F}^{-1}[\mu_{\rm{data}}]$ and $\mu_{\rm{data}} \propto \mathcal{F}[p_{\rm{data}}]$.

Using the isometry property of the Fourier transform for $L_2({\mathbb R}^n)$ and the convolution theorem, we see that:
\begin{equation}
d_{\textsc{MMD}}(\mu, \nu) = \Vert k \ast \mu -k\ast \nu\Vert _{L_2({\mathbb R}^n)}
\propto \Vert \gamma({\mathbf x}) (\mathcal{F}[\mu]({\mathbf x}) - \mathcal{F}[\nu]({\mathbf x}))\Vert _{L_{2}({\mathbb R}^n)}
\end{equation}
Thus, from Theorem~\ref{inclusion} we obtain that the task~\ref{MMD-task} is equivalent to:
\begin{equation}\label{dual}
\Vert p_{\rm{data}} - q\Vert _{L_{2, \gamma^2}({\mathbb R}^n)}  \rightarrow \min_{q\in \mathcal{M}_k}
\end{equation}
\else
\fi

{\bf The higher moments PCA (HM-MMD-PCA)}
Another natural approach to measuring the similarity of two distributions is based on the difference between moments:
\begin{equation}
d_{\textsc{HM}}(\mu, \nu)^2 = \sum_{s=1}^4\frac{\lambda_s}{n^s}\sum_{1\leq i_1, \cdots, i_s\leq n}(m_{i_1 \cdots i_s}-n_{i_1 \cdots i_s})^2
\end{equation}
where $m_{i_1 \cdots i_s} = {\mathbb E}_{{\mathbf X}\sim \mu}\left[{\mathbf X}[{i_1}]\cdots {\mathbf X}[{i_s}]\right]$ and $n_{i_1 \cdots i_s} = {\mathbb E}_{{\mathbf X}\sim \nu}\left[{\mathbf X}[{i_1}]\cdots {\mathbf X}[{i_s}]\right]$ are corresponding moments. The positive parameters $\lambda_1$, $\lambda_2$, $\lambda_3$, $\lambda_4$ are chosen to fix the relative importance of the mean, the co-variance, the co-skewness and the co-kurtosis. 

Thus, we will be interested in the following optimization task (analogous to~\ref{MMD-task}):
\begin{equation}\label{HM-task}
\min_{\nu\in \mathcal{P}_k} d_{\textsc{HM}}(\mu_{\rm{data}}, \nu)
\end{equation}

If we set $\lambda_2=1$ and $\lambda_1=\lambda_3=\lambda_4=0$, then the solution of the task~\eqref{HM-task} coinsides with the solution of the classical PCA. 
Let us briefly demonstrate that. Let $X=[{\mathbf x}_1, \cdots, {\mathbf x}_N]$ be the data matrix, $Y=[{\mathbf y}_1, \cdots, {\mathbf y}_N]$ be the SVD of $X$ truncated at $k$-th term, $\sigma_i(X)$ be an $i$th singular value of $X$. By $\mu_{\rm{pca}}$ we denote a probabilistic measure concentrated in points $\{{\mathbf y}_i\}_{i=1}^N$.
In that case we have $d_{\textsc{HM}}(\mu_{\rm{data}}, \mu_{\rm{pca}})^2 = \|\frac{1}{N} XX^T-\frac{1}{N}YY^T\|_F^2 = \frac{1}{N^2}\sum_{i=k+1}^{\min\{N, n\}}\sigma_i^4(X)$. But for any $\nu\in \mathcal{P}_k$ the covariance matrix ${\rm cov}(\nu) = [{\mathbb E}_{{\mathbf x}\sim \nu} x_i x_j]_{i,j\in [n]}$ is of rank $k$. Therefore, by Eckart-Young-Mirsky's theorem, we have $d_{\textsc{HM}}(\mu_{\rm{data}}, \nu)^2 = \|\frac{1}{N}XX^T-{\rm cov}(\nu)\|^2_F\geq \frac{1}{N^2}\sum_{i=k+1}^{\min\{N, n\}}\sigma_i^4(X)$. Thus, the minimum of $d_{\textsc{HM}}(\mu_{\rm{data}}, \nu)$ is attained at $\nu = \mu_{\rm{pca}}$.

Thus, the task~\eqref{HM-task} can be considered as a direct generalization of PCA that takes into account higher moments. Note that the distance based on higher moments is a special case of maximum mean discrepance metric, where  $K({\mathbf x}, {\mathbf y}) =  \sum_{s=1}^4\frac{\lambda_s}{n^s} ({\mathbf x}\cdot {\mathbf y})^s$. That is why we denote the task as HM-MMD-PCA. In Section~\ref{Approximate-MMD-PCA} we prove that there is an efficient 2-approximating algorithm for the HM-MMD-PCA.
In~\ref{numerical-hm} we additionally describe another algorithm for the HM-MMD-PCA based on a generic alternating scheme.

\ifTR
{\em The dual form.} Due to a well-known relationship between moments of the probability measure $\mu$ and its characteristic function $p$, i.e. ${\rm i}^s m_{i_1 \cdots i_s} = \frac{\partial^s p ({\mathbf 0})}{\partial x_{i_1}\cdots \partial x_{i_s}}$, the task~\ref{HM-task} is equivalent to:
\begin{equation}\label{HM-task-dual}
\sum_{s=1}^4\frac{\lambda_s}{n^s}\sum_{1\leq i_1, \cdots, i_s\leq n}\vert \frac{\partial^s p_{\rm data}({\mathbf 0})}{\partial x_{i_1}\cdots \partial x_{i_s}}-\frac{\partial^s q({\mathbf 0})}{\partial x_{i_1}\cdots \partial x_{i_s}}\vert ^2 \rightarrow \min_{q\in \mathcal{M}_k}
\end{equation}
Note that the maximum mean discrepancy distance and the distance based on higher moments are substantially different. Indeed, even if we set $h$ as a large value (which makes $\frac{1}{h}\approx 0$), the MMD distance, unlike the HM distance, neglects higher order derivatives of the characteristic functions in the neighborhood of the origin. 
\else
\fi
{\bf Wasserstein distance PCA (WD-PCA)}
Another significant distance between probability measures with the origins in the transport theory is the Wasserstein distance (see~\cite{villani2008optimal}). 

Let $({\mathbb R}^n, \Vert \cdot \Vert )$ be a 
Banach space and $p\geq 1$. Between any two Borel probability measures $\mu, \nu$ on ${\mathbb R}^n$ with $\int \Vert {\mathbf x}\Vert^p  d\mu < \infty$ and $\int \Vert {\mathbf x}\Vert^p  d\nu < \infty$ the $p$th Wasserstein distance is:
\begin{equation}
W_p (\mu, \nu) = (\inf_{\pi\in \Pi (\mu, \nu)} \int \Vert {\mathbf x} - {\mathbf y}\Vert^p  d\pi)^{1/p}
\end{equation}
where $\Pi (\mu, \nu)$ is a set of all couplings of $\mu$ and $\nu$. 
The Wasserstein distance defines another version of LDR problem:
\begin{equation}\label{Wasser-task}
\min_{\nu\in \mathcal{P}_k} W_p (\mu_{\rm{data}}, \nu) 
\end{equation}

In the~\ref{WD-Villani} one can find proofs that in the case of $l_1$ norm $\Vert {\mathbf x}\Vert =\sum_{i}\vert x_i\vert $ and $p=1$, the task~\eqref{Wasser-task} corresponds to the well-studied {\em robust PCA} problem~\cite{Candes}.  If, instead of the $l_1$-norm, we use the $l_2$-norm and set $p=1$, this leads to another well-studied task, which is known as {\em the outlier pursuit} problem~\cite{R1-PCA,NIPS2010_4005}. In the case of the $l_2$-norm and a general $p\geq 1$ we obtain {\em the $l_p$ subspace approximation problem}~\cite{deshpande_et_al,AmitVishnoi}. Note that, except for the $l_2$ subspace approximation problem, all these problems are NP-hard. In~\ref{numerical-wd} we describe an algorithm for the WD-PCA in the  case of $l_2$-norm and $p=1$.

\ifTR
Now, let us consider another optimization problem: for a given $X \in {\mathbb R}^{n\times N}$ solve
\begin{equation}\label{robust-pca}
\Vert X-L\Vert \rightarrow \min_{{\rm rank} (L)\leq k}
\end{equation}
where $\Vert \cdot\Vert $ is extended to ${\mathbb R}^{n\times N}$ by $\Vert [{\mathbf s}_1, \cdots, {\mathbf s}_N]\Vert  \eqdef \sum_{i} \Vert {\mathbf s}_i\Vert $.

The following simple theorem shows that the two tasks are connected so that one's solution directly leads to another's solution.
\begin{theorem}\label{transport}
Given data points $\{{\mathbf x}_1, \cdots, {\mathbf x}_N\}$, let $X = [{\mathbf x}_1, \cdots, {\mathbf x}_N]\in {\mathbb R}^{n\times N}$. Then, 
$$\min_{\nu\in {\mathcal P}_k}W (\mu_{\rm{data}}, \nu) = \frac{1}{N} \min_{Y\in {\mathbb R}^{n\times N}, {\rm rank}(Y)\leq k}{\Vert X-Y\Vert }$$
Moreover, $\min_{\nu\in {\mathcal P}_k}W (\mu_{\rm{data}}, \nu)$ is attained on $\nu^\ast$, where $\nu^\ast$ is a uniform distribution over $\{{\mathbf y}_i\}_{i=1}^N$ and $[{\mathbf y}_1, \cdots, {\mathbf y}_N] = \arg\min_{Y\in {\mathbb R}^{n\times N}, {\rm rank}(Y)\leq k}{\Vert X-Y\Vert }$.
\end{theorem}

Note that the case of $L_1$ norm $\Vert {\mathbf x}\Vert =\sum_{i}\vert x_i\vert $ in the task~\ref{robust-pca} corresponds to the well-studied {\em robust PCA} problem~\cite{Candes}. 
If, instead of the $L_1$-norm, we use the $L_2$-norm, this leads to another task:
\begin{equation}\label{outlier-pursuit}
\Vert X-L\Vert _{1,2}\rightarrow \min_{{\rm rank} (L)\leq k}
\end{equation}
where $\Vert S\Vert _{1,2} = \sum_{j}\sqrt{\sum_{i}s^2_{ij}}$, which known as {\em the outlier pursuit} problem~\cite{NIPS2010_4005}.
\else
\fi

{\bf Sufficient dimension reduction with optimized regression function (SDR-ORF).}
Given a labeled dataset $\{({\mathbf x}_i, y_i)\}_{i=1}^N$ where ${\mathbf x}_i\in {\mathbb R}^n, y_i \in {\mathcal C}$ (${\mathcal C}$ is a finite set of classes for a classification, or ${\mathbb R}$ for a regression problem), the sufficient dimension reduction problem can be informally described as a problem of finding vectors ${\mathbf w}_1, \cdots, {\mathbf w}_k\in {\mathbb R}^{n}$ such that conditional distributions satisfy $p(y\vert  {\mathbf w}^T_1 {\mathbf x}, \cdots, {\mathbf w}^T_k {\mathbf x})\approx p(y\vert  {\mathbf x})$ (possibly, under some additional assumptions on the form of $p(y\vert  {\mathbf x})$). 

We formulate the SDR-ORF problem as an optimization task:
\begin{equation}\label{sdr-problem}
\inf_{f\in \mathcal{F}_k} J(f)
\end{equation}
The object $f: {\mathbb R}^n\rightarrow {\mathbb R}$ is a smooth real-valued function. 
We assume that $f$ is a candidate for the regression function and $J(f)$ is a cost function that values how strongly $f$ fits in this role. In practice for the regression case and for the binary classification case with 0-1 outputs we use the following cost functions correspondingly:
\begin{equation}
\begin{split}
J(f) = \frac{1}{ N}\sum_{i=1}^N {\mathbb E}_{\boldsymbol{\epsilon} \sim N({\mathbf 0}, \upsilon^2 I_n)} \vert y_i-f({\mathbf x}_i+\boldsymbol{\epsilon})\vert ^2
\end{split}
\end{equation}
and
\begin{equation}
\begin{split}
J(f) = \frac{1}{ N}\sum_{i=1}^N {\mathbb E}_{\boldsymbol{\epsilon} \sim N({\mathbf 0}, \upsilon^2 I_n)} H\left(y_i, \frac{e^{f({\mathbf x}_i+\boldsymbol{\epsilon})}}{1+e^{f({\mathbf x}_i+\boldsymbol{\epsilon})}}\right)
\end{split}
\end{equation}
where $H (y, p) = - y \log p - (1-y) \log (1-p)$ and $\upsilon>0$ is a parameter. 

By requiring $f\in \mathcal{F}_k$, we assume that the regression function $f$ satisfies (for $k$ fixed in advance):
$
f({\mathbf x}) = g({\mathbf w}^T_1{\mathbf x}, \cdots, {\mathbf w}^T_k{\mathbf x})
$,
where ${\mathbf w}_1, \cdots, {\mathbf w}_k\in {\mathbb R}^n$. Thus, given an input ${\mathbf x}$, an output of $f$ depends on the projection of ${\mathbf x}$ onto ${\rm span}({\mathbf w}_1, \cdots, {\mathbf w}_k)$. The set ${\rm span}({\mathbf w}_1, \cdots, {\mathbf w}_k)$ is called the effective subspace. 
In~\ref{numerical-sdr} we describe an algorithm for the SDR-ORF problem.
\section{Reduction of the optimization problem to ordinary functions}\label{HowOpt}
The central problem that our paper addresses is the optimization of an objective function over $\mathcal{G}'_k$? 
In this section we suggest an approach based on penalty functions and kernels.

\subsection{The definition of the penalty function}
In this subsection we introduce a penalty function $R(f)$. 
Let $M:{\mathbb R}^n\times {\mathbb R}^n\to {\mathbb C}$\footnote{Throughout the paper the kernel that induces the MMD distance is denoted by $K$ and the kernel that is used to define a penalty is denoted by $M$.} be some bounded function such that $[M({\mathbf z}_i, {\mathbf z}_j)]_{i,j\in [x]}$ is a positive semidefinite matrix for any $\{{\mathbf z}_i\}_{i\in [x]}\subseteq {\mathbb R}^n, x\in {\mathbb N}$. For $f, g: {\mathbb R}^n\rightarrow {\mathbb C}$ let us denote
\begin{equation}
\langle f \vert  M \vert  g \rangle = \iint_{{\mathbb R}^n\times {\mathbb R}^n} f({\mathbf x})^\ast M({\mathbf x}, {\mathbf y}) g({\mathbf y}) d {\mathbf x} d {\mathbf y}.
\end{equation}

For $f,g \in L_1 ({\mathbb R}^n)$, 
\begin{equation}
\langle f \vert  M \vert  g \rangle \leq \sup_{{\mathbf x}, {\mathbf y}}|M({\mathbf x}, {\mathbf y})| \cdot \Vert f\Vert _{L_1}\Vert g\Vert _{L_1} < \infty.
\end{equation} 
For general $f,g \in {\mathcal S}' ({\mathbb R}^n)$ the expression $\langle f \vert  M \vert  g \rangle$ is defined if there are $f_\epsilon, g_\epsilon\in L_1 ({\mathbb R}^n)$ such that $T_{f_\epsilon} = f\ast G^n_\epsilon$, $T_{g_\epsilon} = g\ast G^n_\epsilon$ and $\lim_{\epsilon\rightarrow 0}\langle f_\epsilon \vert  M \vert  g_\epsilon \rangle  = A<\infty$. Then, $\langle f \vert  M \vert  g \rangle \eqdef A$. For example, for continuous $M$ we have $\langle \delta^n \vert  M \vert  \delta^n \rangle = M(0,0)$.

One can build a Gram matrix from the collection of functions $\{x_if\}_{i=1}^n$, $\begin{bmatrix}
\langle x_i f\vert  M \vert  x_j f\rangle
\end{bmatrix}_{1\leq i, j \leq n}$. Let us denote a real part of the Gram matrix $$\begin{bmatrix} 
\langle x_i f\vert  M \vert x_j f\rangle
\end{bmatrix}_{1\leq i, j \leq n}$$ by $M_f$. 

Theorem~\ref{low-rank-char} concludes, from $f\in \mathcal{G}_k$, that 
$\dim {\rm span}_{\mathbb R} \{x_1 f, x_2 f, \cdots, x_n f\}\leq k$. 
\begin{theorem}\label{rank-k}
Let $M({\mathbf x}, {\mathbf y})$ be a bounded Lipschitz function. If $f=(T_g\otimes \delta^{n-k})_U\in \mathcal{G}'_k$ is such that $\{x_i g\}_{i=1}^k\subseteq L_1({\mathbb R}^k)$,
then $\langle x_i f\vert  M \vert  x_j f\rangle$ is defined and $\rank M_f \leq k$.
\end{theorem}
 
\begin{definition} Let $A\in {\mathbb R}^{n\times n}$ be a positive semidefinite matrix with eigenvalues $\lambda_1 \geq \lambda_2 \geq \cdots \geq \lambda_n$ (with counting multiplicities). Then, the Ky Fan $k$--anti-norm of $A$ is $\Vert A\Vert _k = \sum_{i=1}^k \lambda_{n+1-k}$.
\end{definition}
Let 
\begin{equation}\label{R-to-M}
R(f) = \Vert M_f\Vert _{n-k}.
\end{equation}
By construction, by penalizing the value of $R(f)$, we enforce $M_f$ to be close to some matrix of rank $k$. Equivalently, we enforce a real part of the Gram matrix of $\{x_i f\}_{i\in [n]}$ to be of close to a rank $k$ matrix. By Theorem~\ref{low-rank-char}, the condition 
$\dim {\rm span}_{\mathbb R} \left\{x_1 f, x_2 f, \cdots, x_n f\right\}\leq k$ implies $f\in \mathcal{G}'_k$, therefore, we enforce $f$ to be close to some function from $\mathcal{G}'_k$. In the next section we will justify the latter informal logic by reducing the optimization over $\mathcal{G}'_k$ to the optimization  over ordinary functions with the penalty function $R(f)$. 

\subsection{Proper kernels}
For a function $M({\mathbf x}, {\mathbf y}): {\mathbb R}^n \times {\mathbb R}^n\rightarrow {\mathbb C}$, let us denote by ${\rm O}_M$ a linear operator between ${\rm Dom}({\rm O}_M)$ and $L_2({\mathbb R}^n)$ given by ${\rm O}_M[f] = \int_{{\mathbb R}^n} M({\mathbf x}, {\mathbf y}) f({\mathbf y}) d {\mathbf y}$ where ${\rm Dom}({\rm O}_M) = \{f\in L_2({\mathbb R}^n) \mid {\rm O}_M[f] \in L_2({\mathbb R}^n)\}$. 
For any operator $O$ between spaces $\mathcal{H}_1$ and $\mathcal{H}_2$, we denote its range by ${\rm Range\,} [O] = \{O(x)\vert  x\in \mathcal{H}_1\}$.

\begin{definition} The function $M({\mathbf x}, {\mathbf y}): {\mathbb R}^n \times {\mathbb R}^n\rightarrow {\mathbb C}$ is called the {\em proper kernel} if and only if 
\begin{enumerate}
\item ${\rm O}_M: L_2({\mathbb R}^n)\to L_2({\mathbb R}^n)$ is a properly defined, strictly positive and self-adjoint operator,
\item $\max_{{\mathbf x}, {\mathbf y}}\vert M({\mathbf x}, {\mathbf y})\vert  < \infty$, 
\item $\overline{{\rm Range\,} [{\rm O}_M]\cap {\mathcal S} ({\mathbb R}^n)} = {\mathcal S} ({\mathbb R}^n)$.
\end{enumerate}
\end{definition}
Note that the latter definition implies that $M({\mathbf y}, {\mathbf x}) = M({\mathbf x}, {\mathbf y})^\ast$ (modulo some null set) and $\langle f, {\rm O}_M [f]\rangle_{L_2({\mathbb R}^n)} > 0, \forall f\in L_2({\mathbb R}^n), f\ne {\mathbf 0}$.
\begin{example} 
The Gaussian kernel is of special interest in applications:
$
M({\mathbf x}, {\mathbf y}) = G^n_\sigma ({\mathbf x}-{\mathbf y}).
$
\end{example} 
It is captured by the following lemma:
\begin{lemma} If $\zeta, \hat{\zeta}\in C ({\mathbb R}^n)$ are bounded, $\forall {\mathbf x}\,\, \hat{\zeta}({\mathbf x})> 0$, then $M({\mathbf x}, {\mathbf y}) = \zeta({\mathbf x}-{\mathbf y}) $ is a proper kernel.
\end{lemma}
\begin{proof}
Verification of the first three conditions is easy, so we only check the fourth condition.
Let us denote linear operators $C_{\zeta}[f] = \zeta\ast f$ and $O_{g}[f]({\mathbf x}) = g({\mathbf x})f({\mathbf x})$. Then we have $\mathcal{F}[C_{\zeta}[L_{2}({\mathbb R}^n)]] = O_{\hat{\zeta}}[L_{2}({\mathbb R}^n)]\supseteq C_c^\infty ({\mathbb R}^n)$. Therefore, ${\rm Range\,} [{\rm O}_M] = C_{\zeta}[L_{2}({\mathbb R}^n)]\supseteq \mathcal{F}^{-1}[C_c^\infty ({\mathbb R}^n)]$. Since $C_c^\infty ({\mathbb R}^n)$ is dense in ${\mathcal S} ({\mathbb R}^n)$, then $\mathcal{F}^{-1}[C_c^\infty ({\mathbb R}^n)]$ also has this property. Thus, $\overline{{\rm Range\,} [{\rm O}_M]\cap {\mathcal S} ({\mathbb R}^n)} = {\mathcal S} ({\mathbb R}^n)$.
\end{proof}

Besides the Gaussian kernel the lemma also captures a case of the Laplace kernel $\zeta({\mathbf x}) = e^{-\vert {\mathbf x}\vert }$. It is well-known that the Fourier tranform of the Laplace kernel is the Poisson kernel: $\hat{\zeta}({\mathbf x}) = \frac{c_n}{(1+\vert {\mathbf x}\vert ^2)^{\frac{n+1}{2}}}$ (which is also proper). 


For $I: \mathcal{G}'_k\cup {\mathcal S} ({\mathbb R}^n)\rightarrow {\mathbb R}^+$, it is natural to reduce the optimization task {\em over tempered distributions}
\begin{equation}\label{main-task}
I(f)\rightarrow \min\limits_{f\in \mathcal{G}'_k}
\end{equation}
to an optimization task {\em over ordinary functions with a penalty term} $R$,
\begin{equation}\label{sequence}
I(f)+\lambda \Vert M_f\Vert _{n-k} = I(f)+\lambda R(f) \rightarrow \inf\limits_{f\in \mathfrak{F}},
\end{equation}
where we assume that the set of functions $\mathfrak{F}$ is rich enough  to approximate weakly solutions of~\eqref{main-task}, i.e. $\overline{\mathfrak{F}}^\ast \supset \mathcal{G}'_k$. Since we cannot guarantee that the minimum in~\eqref{sequence} is attainable, we substitute it by infimum. For this reduction to be effective it is desirable to have the following property: if a sequence $\{f_{n}\}\subset \mathfrak{F}$ is such that $I(f_n)+\lambda_n R(f_n)-\inf\limits_{f\in \mathfrak{F}}\left(I(f)+\lambda_n R(f)\right)\rightarrow +0$ for $\lambda_n\mathop\rightarrow\limits^{n\rightarrow \infty} +\infty$ (i.e. $\{f_n\}$ solves~\eqref{sequence} for arbitrarily large values of the regularization parameter), then there exists a growing subsequence $\{n_k\}$ such that $T_{f_{n_k}}\rightarrow^\ast T$ (weakly) where $T$ is a solution of~\eqref{main-task}. 

We make a thorough theoretical analysis of the case  $\mathfrak{F} = {\mathcal S} ({\mathbb R}^n)$. If to formulate in a simplified way, for the last property to hold, the sequence ${\rm Tr} (M_{f_n})$ should be bounded.  
Details on the conditions under which this reduction holds can be found in the following subsection. 

\subsection{Regular solutions and reduction theorems for $\mathfrak{F} = {\mathcal S} ({\mathbb R}^n)$}\label{generaltheory}
 For a sequence $\{f_s\}^\infty_{s=1}\subseteq {\mathcal S}' ({\mathbb R}^n)$, $\mathop{\rm Lim}\limits_{s\rightarrow \infty} f_s$ denotes a set of points $f\in {\mathcal S}' ({\mathbb R}^n)$, such that there exists a growing sequence $\{s_i\}\subseteq {\mathbb N}$ and $\lim_{i\rightarrow \infty} f_{s_i} = f$.

For $I: \mathcal{G}'_k\cup {\mathcal S} ({\mathbb R}^n)\rightarrow {\mathbb R}^+$, it is natural to reduce the optimization task~\eqref{main-task}
to an optimization task over ordinary functions with a penalty term~\eqref{sequence}.
To have an equivalence between~\eqref{main-task} and~\eqref{sequence} we need to assume that $I$'s behaviour when approaching $f\in \mathcal{G}'_k$ from a set ${\mathcal S} ({\mathbb R}^n)$ is continuous, i.e. for any sequence $\{f_i\}\subseteq {\mathcal S} ({\mathbb R}^n)$ such that $T_{f_i}\rightarrow^\ast f\in \mathcal{G}'_k$, we have $\lim_{i\rightarrow \infty} I(T_{f_i}) = I(f)$.

Let us introduce the notion of a regular solution both for~\eqref{main-task} and~\eqref{sequence}. Let 
\begin{equation}
\mathcal{B}_k = \bigcup_{C>0} \overline{\{f\in \mathcal{G}_k\vert  \Tr(M_{f})\leq C\}}^\ast.
\end{equation} 

\begin{definition} Any $f\in {\rm Arg} \min \limits_{f\in \mathcal{G}'_k} I(f)\bigcap \mathcal{B}_k$ is called a regular solution of~\eqref{main-task}.
\end{definition}
In other words, $\mathcal{B}_k$ formalizes a set of distributions from $\mathcal{G}'_k$, that can be approached through sequences $\{f_i\}\subseteq \mathcal{G}_k$, for which $\Tr(M_{f_i})$ does not blow up. Obviously, $\mathcal{G}_k\subseteq \mathcal{B}_k\subseteq \mathcal{G}'_k$. In applications, regular solutions include all ${\rm Arg} \min \limits_{f\in \mathcal{G}'_k} I(f)$ if we choose the kernel $M$ correctly. This regularity is important for a reduction to the penalty form~\eqref{sequence}, because when approaching a non-regular solution we are unable to guarantee a bounded behaviour of $M_f$ (and of $R(f)$). 
\begin{definition}
A sequence $\{f_i\}_1^\infty\subseteq {\mathcal S} ({\mathbb R}^n)$ is said to solve~\eqref{sequence} if 
\begin{equation}\label{f-sequence}
I(f_i)+\lambda_i R(f_i) \leq  \inf\limits_{f\in {\mathcal S} ({\mathbb R}^n)}I(f)+\lambda_i R(f) +\epsilon_i
\end{equation}
where $\epsilon_i\rightarrow +0$ and $\lambda_i\rightarrow +\infty, i\rightarrow +\infty$. If, additionally, $\Tr(M_{f_i})$ is bounded, then $\{f_i\}_1^\infty$ is said to solve~\eqref{sequence} regularly.
\end{definition}
Let us define
\begin{equation}
{\rm rsol\,} (I(f), R(f)) = \bigcup_{\{f_i\}_1^\infty\rm{\,r.\,solves\,(11)}} \mathop{\rm Lim}\limits_{i\rightarrow \infty} T_{f_i}.
\end{equation}

\begin{theorem}\label{metrization} If $M$ is a proper kernel, then
${\rm rsol\,} (I(f), R(f)) \subseteq {\rm Arg} \min \limits_{f\in \mathcal{G}'_k} I(f)$. 
\end{theorem}
\begin{theorem}\label{existence} If $M$ is a proper kernel and ${\rm rsol\,} (I(f), R(f))\ne \emptyset$, then
$${\rm Arg} \min \limits_{f\in \mathcal{G}'_k} I(f) \bigcap \mathcal{B}_k \subseteq {\rm rsol\,} (I(f), R(f)).$$ 
\end{theorem}

\begin{theorem}[Reduction theorem]
If $M$ is a proper kernel, ${\rm Arg} \min \limits_{f\in \mathcal{G}'_k} I(f)\subseteq \mathcal{B}_k$ and ${\rm rsol\,} (I(f), R(f))\ne \emptyset$, then
${\rm rsol\,} (I(f), R(f)) = {\rm Arg} \min \limits_{f\in \mathcal{G}'_k} I(f)$.  
\end{theorem}

Suppose that we now solve a sequence of problems~\eqref{sequence} and find $\{f_s\}_1^\infty$. According to Theorems~\ref{metrization} and~\ref{existence}, the following are potential scenarios:

(1) $\Tr(M_{f_s})$ blows up and the convergence is not guaranteed. This situation can be avoided by controlling $\Tr(M_{f})$ in an optimization process. In practice, when $f$ has a parameterized form, this can be done by bounding parameters.

If $\Tr(M_{f_s})$ does not blow up, we still have two subcases:

(2.1) $\mathop{\rm Lim}\limits_{s\rightarrow \infty} T_{f_s}\ne \emptyset$. This implies a positive outcome to approach~\eqref{sequence} to the optimization problem, Problem~\eqref{main-task}.

(2.2) $\mathop{\rm Lim}\limits_{s\rightarrow \infty} T_{f_s} = \emptyset$. This exotic situation can happen only if a sequence $T_{f_s}$ leaves any sequentially compact subset of ${\mathcal S}' ({\mathbb R}^n)$. Bounding parameters also tackles this case.

Let us now concentrate on the task~\eqref{sequence} and describe the alternating scheme for its solution.

\section{The alternating scheme}\label{mod-alt-right}
We will concentrate on problem~\eqref{sequence}. It is known~\cite{HIAI20131568} that the Ky Fan anti-norm is a concave function, i.e. $R(\phi) = \Vert M_\phi\Vert _{n-k}$ depends on $M_\phi$ in a concave way. It can be shown that the dependence of $R(\phi)$ on $\phi$ is both non-convex and non-concave, i.e. we deal with a non-convex optimization task.

\ifTR
Thus, even if we define $I(\phi)$ as a convex function, an optimization task
\begin{equation}\label{task0}
I(\phi)\rightarrow \min_{\phi: R(\phi)\leq \epsilon}
\end{equation}
{\em will not be dual} to its penalty form: 
\begin{equation}\label{task}
I(\phi)+\lambda R(\phi) \rightarrow \min_{\phi}
\end{equation}
since $R(\phi)$ is not a convex function of $\phi$. 
Though, it is well-known that the constrained optimization over a finite dimensional space can be reduced to its penalty form for a sufficiently large coefficient $\lambda$~\cite{Bertsekas} (p. 281). Therefore, we will concentrate on problem~\ref{task}. Indeed, in~\ref{task} we penalize the value of $R(\phi)$, forcing it to be small, which could serve as a substitute of the constraint $R(\phi)\leq \epsilon$.

Recalling that $R(f) = \min_{A\in {\mathbb R}^{n\times n}, \rank A \leq k}\Vert \sqrt{M_f} -A \Vert ^2_F$, the latter problem can be rewritten as a minimization of $I(\phi)+\lambda \Vert \sqrt{M_f} -A \Vert ^2_F$ over two objects: $\phi$ and $A\in {\mathbb R}^{n\times n}, \rank A \leq k$.
Let us introduce a natural algorithm for~\ref{task}, which we call {\em the alternating scheme}, because in that algorithm we just simply optimize alternatingly over 2 arguments, $\phi$ and $A$.
\begin{algorithm}
\caption{Alternating scheme}\label{alternate}
$A_0 \longleftarrow  {\mathbf 0}$

\For{$t = 1, \cdots, T$}{
$\phi_{t}\longleftarrow  \arg\min\limits_{\phi} I(\phi)+\lambda \Vert \sqrt{M_\phi} -A_{t-1} \Vert ^2_F$

$A_{t} \longleftarrow  \arg\min_{A\in {\mathbb R}^{n\times n}, \rank A \leq k} \Vert \sqrt{M_{\phi_t}} -A \Vert _F$
}

Find $\{{\mathbf v}_i\}_{1}^n$ s.t. $A_{N}{\mathbf v}_i=\lambda_i{\mathbf v}_i$, $\lambda_1\geq \cdots\geq \lambda_n$ 

\textbf{Output:} ${\mathbf v}_1, \cdots, {\mathbf v}_k$
\end{algorithm}

Further we are developing different modifications and aspects of that algorithm.
\else
\fi
Let $\mathcal{B}(H_1, H_2)$ denote a set of bounded linear operators between Hilbert spaces $H_1$ and $H_2$. 
For $O\in \mathcal{B}(H_1, H_2)$ the rank of $O$ is defined as $\dim {\mathcal R}(O)$. Let $L^r_2({\mathbb R}^{n})$ be the Hilbert space (over ${\mathbb R}$) of real-valued functions from $L_2({\mathbb R}^{n})$ (i.e. the real-valued $L_2$-space) and $L^\ast_2({\mathbb R}^{n}) = L^r_2({\mathbb R}^{n}) \times L^r_2({\mathbb R}^{n})$. The space $L^\ast_2({\mathbb R}^{n})$ is equivalent to $L_2({\mathbb R}^{n})$ treated as a linear space over ${\mathbb R}$. Below we do not distinguish $[\phi_1, \phi_2]\in L^\ast_2({\mathbb R}^{n})$ and  $\phi_1+ {\rm i}\phi_2\in L_2({\mathbb R}^{n})$.
It is easy to see that any $O\in \mathcal{B} (L^\ast_2({\mathbb R}^{n}), {\mathbb R}^{n})$ can be given by formula:
\begin{equation}
O[\phi]_{i} ={\rm Re\,} \langle O_{i},  \phi\rangle_{L_2({\mathbb R}^{n})}, O_{i}\in L_2({\mathbb R}^{n}), i=\overline{1,n},
\end{equation}
i.e. $O\in \mathcal{B} (L^\ast_2({\mathbb R}^{n}), {\mathbb R}^{n})$ can be identified with a vector of functions
$O = \begin{bmatrix}
O_{i}
\end{bmatrix}_{i=\overline{1,n}}, O_{i}\in L_2({\mathbb R}^{n})$ and the Hilbert–Schmidt norm on $\mathcal{B} (L^\ast_2({\mathbb R}^{n}), {\mathbb R}^{n})$ (i.e. $\sqrt{{\rm Tr\,}O^\dag O}$) is
\begin{equation}\label{Bnorm}
\Vert O\Vert _\ast = \sqrt{\sum_{i=1}^n \Vert O_{i}\Vert ^2_{L_2({\mathbb R}^{n})}}.
\end{equation}
Recall that for a Mercer kernel $M$, ${\rm O}_M[\phi]({\mathbf x}) = \int_{{\mathbb R}^n} M({\mathbf x},{\mathbf y})\phi({\mathbf y})d{\mathbf y}$ is a positive operator whose domain is ${\rm Dom}({\rm O}_M) = \{f\in L_2({\mathbb R}^n)\mid {\rm O}_M[f]\in L_2({\mathbb R}^n)\}$ and range is a subset of $L_2({\mathbb R}^n)$. If we assume that  ${\rm Dom}({\rm O}_M)$ is dense in $L_2({\mathbb R}^n)$, then its adjoint ${\rm O}^\dag_M$ and the square root $\sqrt{{\rm O}_M}:{\rm Dom}({\rm O}_M)\to L_2({\mathbb R}^n)$ can be properly defined~\cite{bernau_1968}. Thus, ${\rm O}_M$ is self-adjoint. For any complex-valued function $f$ such that ${\rm Tr\,}M_f<\infty$ let us introduce a linear operator $S_f: L^\ast_2({\mathbb R}^{n}) \rightarrow  {\mathbb R}^{n}$ by the following rule:
\begin{equation}
S_f[\phi]_{i} ={\rm Re\,} \langle \sqrt{{\rm O}_M}[x_i f({\mathbf x})], \phi \rangle_{L_2({\mathbb R}^{n})},
\end{equation}
i.e. $(S_f)_{i} = \sqrt{{\rm O}_M}[x_if({\mathbf x})], i=\overline{1,n}$. In the latter definition the expression $\langle \sqrt{{\rm O}_M}[x_i f({\mathbf x})], \phi \rangle_{L_2({\mathbb R}^{n})}$ is finite due to $\langle \sqrt{{\rm O}_M}[x_if({\mathbf x})],\sqrt{{\rm O}_M}[x_if({\mathbf x})]\rangle_{L_2({\mathbb R}^n)} = (M_f)_{ii}<\infty$ and the Cauchy-Schwarz inequality.
 
\begin{theorem}\label{square-root-operator}
Let $M$ be a Mercer kernel such that  ${\rm Dom}({\rm O}_M)$ is dense in $L_2({\mathbb R}^n)$ and $\,\Tr M_f<\infty$. Then, $S_f\in \mathcal{B} (L^\ast_2({\mathbb R}^{n}), {\mathbb R}^{n})$ and $S_f S^\dag_f = M_f$. Moreover, 
\begin{equation}
R(f) = \min\limits_{S\in \mathcal{B} (L^\ast_2({\mathbb R}^{n}), {\mathbb R}^{n}), \rank S \leq k} \Vert S_f - S\Vert _\ast^2
\end{equation}
 and the minimum is attained at $S = P_f S_f$ where $P_f=\sum_{i=1}^k {\mathbf u}_i {\mathbf u}_i^\dag$ and $\{{\mathbf u}_i\}^k_{1}$ are unit eigenvectors of $M_{f}$ corresponding to the $k$ largest eigenvalues (counting multiplicities). 
\end{theorem}
\begin{proof} 
The boundedness of $S_f$ follows from the Cauchy-Schwarz inequality:
\begin{equation}
\begin{split}
\mid  S_f[\phi]_i\mid  ^2 = \mid  {\rm Re\,}\langle \sqrt{{\rm O}_M}[x_i f], \phi\rangle\mid  ^2 \leq 
\langle \sqrt{{\rm O}_M}[x_i f], \sqrt{{\rm O}_M}[x_i f]\rangle \langle \phi, \phi\rangle  = 
\langle x_i f, {\rm O}_M[x_i f]\rangle \langle \phi, \phi\rangle 
\end{split}
\end{equation}
and therefore:
\begin{equation}
\Vert S_f[\phi]\Vert ^2 = \sum_{i=1}^n \vert S_f[\phi]_i\vert ^2  \leq \Tr M_f 
 \Vert \phi\Vert ^2_{L_2({\mathbb R}^{n})}.
\end{equation}
Thus, we have checked that $S_f$ is bounded. 

By definition, $S^\dag_f:  {\mathbb R}^{n}\rightarrow L^r_2({\mathbb R}^{n})\times L^r_2({\mathbb R}^{n})$ and $\langle {\mathbf u}, S_f[\phi_1, \phi_2]\rangle = \langle S^\dag_f[{\mathbf u}], [\phi_1, \phi_2]\rangle$, ${\mathbf u}\in {\mathbb R}^{n}$, $[\phi_1, \phi_2]\in L^r_2({\mathbb R}^{n})\times L^r_2({\mathbb R}^{n})$. Let us denote $f_1 = {\rm Re\,}f, f_2 = {\rm Im\,}f$. It is easy to see that the following operator satisfies the latter identity:
\begin{equation}
O[{\mathbf u}]  = \begin{bmatrix}
\sqrt{{\rm O}_M}[f_1({\mathbf x}) {\mathbf x}^T{\mathbf u} ], \sqrt{{\rm O}_M}[f_2({\mathbf x}) {\mathbf x}^T{\mathbf u} ]
\end{bmatrix}.
\end{equation}
Since the adjoint is unique, then $S^\dag_f = O$. Let us calculate  $S_f S^\dag_f$:
\begin{equation}
\begin{split}
{\mathbf u} \xrightarrow{S^\dag_f} \begin{bmatrix}
\sqrt{{\rm O}_M}[f_1({\mathbf x}) {\mathbf x}^T{\mathbf u} ], \sqrt{{\rm O}_M}[f_2({\mathbf x}) {\mathbf x}^T{\mathbf u} ]
\end{bmatrix} \xrightarrow{S_f}  \\
\begin{bmatrix}
\langle x_1 f_1({\mathbf x}),  \sqrt{{\rm O}_M}[\sqrt{{\rm O}_M}[f_1({\mathbf x}) {\mathbf x}^T{\mathbf u} ]] \rangle  \\
\cdots \\
\langle x_n f_1({\mathbf x}),  \sqrt{{\rm O}_M}[\sqrt{{\rm O}_M}[f_1({\mathbf x}) {\mathbf x}^T{\mathbf u} ]] \rangle 
\end{bmatrix} + 
\begin{bmatrix}
 \langle x_1 f_2({\mathbf x}),  \sqrt{{\rm O}_M}[\sqrt{{\rm O}_M}[f_2({\mathbf x}) {\mathbf x}^T{\mathbf u} ]] \rangle \\
\cdots \\
 \langle x_n f_2({\mathbf x}),  \sqrt{{\rm O}_M}[\sqrt{{\rm O}_M}[f_2({\mathbf x}) {\mathbf x}^T{\mathbf u} ]] \rangle
\end{bmatrix} 
= \\
\begin{bmatrix}
\sum_{j=1}^2\langle x_1 f_j({\mathbf x}),  {\rm O}_M[f_j({\mathbf x}) {\mathbf x}^T{\mathbf u} ] \rangle \\
\cdots \\
\sum_{j=1}^2\langle x_n f_j({\mathbf x}),  {\rm O}_M[f_j({\mathbf x}) {\mathbf x}^T{\mathbf u} ] \rangle
\end{bmatrix} = 
\begin{bmatrix}
{\rm Re\,}\langle x_i f, M[x_j f]\rangle
\end{bmatrix}_{1\leq i, j \leq n} {\mathbf u} = M_f {\mathbf u}
\end{split}
\end{equation}
Thus, $S_f S^\dag_f = M_f$. Since $\Tr S_f S^\dag_f < \infty$ and $\Vert  S^\dag_f[{\mathbf u}]\Vert  ^2 \leq \langle {\mathbf u}, M_f {\mathbf u}\rangle$, we obtain $S^\dag_f$ is a bounded operator.

Let ${\mathbf u}_1, \cdots {\mathbf u}_n$ be orthonormal eigenvectors of $\mathcal{M}_{f} = S_f S^\dag_f$ and $\lambda_1 \geq \cdots \geq \lambda_{n'}>0$ be corresponding nonzero eigenvalues. For $\sigma_i = \sqrt{\lambda_i}$ let us define
$
{\mathbf v}_i = \frac{S_{f}^\dag[{\mathbf u}_i]}{\sigma_i}
$.
Vector ${\mathbf v}_i$ corresponds to a pair of functions
\begin{equation}
\begin{split}
{\mathbf v}_i = \frac{1}{\sigma_i}\begin{bmatrix}
\sqrt{{\rm O}_M}[f_1({\mathbf x}) {\mathbf x}^T{\mathbf u}_i ] \\ \sqrt{{\rm O}_M}[f_2({\mathbf x}) {\mathbf x}^T{\mathbf u}_i ]
\end{bmatrix} \in L^r_2({\mathbb R}^{n})\times L^r_2({\mathbb R}^{n})
\end{split}
\end{equation}
It is easy to see that ${\mathbf v}_1, \cdots {\mathbf v}_{n'}$ is an orthonormal basis in $\Ima  S_f^\dag$, and 
$S_{f}^\dag$ can be expanded in the following way:
\begin{equation}
S_{f}^\dag = \sum_{i=1}^{n'} \sigma_i {\mathbf v}_i {\mathbf u}^\dag_i,
\end{equation}
and therefore, 
SVD for $S_{f}$ is
\begin{equation}
S_{f} = \sum_{i=1}^{n'} \sigma_i  {\mathbf u}_i {\mathbf v}_i^\dag.
\end{equation}
By the Eckart-Young-Mirsky theorem (see Theorem 4.4.7 from~\cite{hsing2015theoretical}), an optimal $S$ in $\min\limits_{S\in \mathcal{B} (L^\ast_2({\mathbb R}^{n}), {\mathbb R}^{n}), \rank S \leq k} \Vert S_f - S\Vert _\ast^2$ is defined by a truncation of SVD for $S_{f}$ at $k$th term, i.e.
\begin{equation}\label{optS}
S = \sum_{i=1}^{k} \sigma_i  {\mathbf u}_i {\mathbf v}_i^\dag =  P_{f} S_f,
\end{equation}
where $P_{f} = \sum_{i=1}^k {\mathbf u}_i {\mathbf u}_i^\dag$ is a projection operator to first $k$ principal components of $\mathcal{M}_{f}$. Moreover, $\Vert S_f - P_f S_f\Vert ^2 = \sum_{i=k+1}^{n'} \sigma^2_i = \Vert M_f\Vert _{n-k} = R(f)$.
\end{proof}

Given the new representation $R(f) = \min\limits_{S\in \mathcal{B} (L^\ast_2({\mathbb R}^{n}), {\mathbb R}^{n}), \rank S \leq k} \Vert S_f - S\Vert _\ast^2$ we have
\begin{equation}
\begin{split}
\min_{f\in \mathfrak{F}}  I(f)+\lambda R(f) = \min_{\scriptsize\begin{matrix} f\in \mathfrak{F}, \\ S\in \mathcal{B} (L^\ast_2({\mathbb R}^{n}), {\mathbb R}^{n}), \rank S \leq k\end{matrix}} I(f)+\lambda \Vert S_f - S\Vert _\ast^2
\end{split}
\end{equation}
Thus, it is natural to view the Task~\eqref{sequence} as a minimization of $I(\phi)+\lambda \Vert S_\phi - S\Vert _\ast^2$ over two objects: $f \in \mathfrak{F}$ and $S\in \mathcal{B} (L^\ast_2({\mathbb R}^{n}), {\mathbb R}^{n}): \rank S \leq k$. 
The simplest approach to minimize a function over two arguments is to optimize alternatingly, i.e. first over $f$, and then over $S: \rank S \leq k$, and so on. Theorem~\ref{square-root-operator} gives that the minimization over $S$ is equivalent to the truncation of ${\rm SVD}(S_{f})$ at the $k$-th term.
This idea, that we dub the {\em alternating scheme} (AS), is described in Algorithm~\ref{alternate-mod}. 

\begin{algorithm}
\begin{algorithmic}
\caption{The alternating scheme (AS) for~\eqref{sequence}}\label{alternate-mod}
\State $(P_0,S_{\phi_{0}}) \longleftarrow $ Initialize

\For{$t = 1, \cdots, T$}

\State $\phi_{t}\longleftarrow  \arg\min\limits_{\phi\in \mathfrak{F}} I(\phi)+\lambda \Vert S_\phi - P_{t-1}S_{\phi_{t-1}}\Vert ^2_\ast$ (minimizing over $\phi$)

\State Calculate $M_{\phi_t}$ and find $\{{\mathbf v}_i\}^n_1$ s.t. $M_{\phi_t}{\mathbf v}_i=\lambda_i{\mathbf v}_i$, $\lambda_1\geq \cdots\geq \lambda_n$

\State $P_{t} \longleftarrow  \sum_{i=1}^k {\mathbf v}_i {\mathbf v}_i^T$ (Truncated ${\rm SVD}(S_{\phi_t})$ is $P_{t}S_{\phi_t}$)

\EndFor

\State \textbf{Output:} ${\mathbf v}_1, \cdots, {\mathbf v}_k$
\end{algorithmic}
\end{algorithm}

The alternating algorithm~\ref{alternate-mod} allows for a reformulation in the dual space. By this we mean that in Algorithm~\ref{alternate-mod} we substitute $\widehat{\phi}_{t}$ for the original $\phi_{t}$. If the primal Algorithm~\ref{alternate-mod} deals with  operators $S_{\phi}, S_{\phi_{t-1}}$, the dual version deals with vectors of functions $\sqrt{{\widehat G_{\sigma}}}\frac{\partial {\widehat\phi}}{\partial {\mathbf x}}, \sqrt{{\widehat G_{\sigma}}}\frac{\partial {\widehat\phi}_{t-1}}{\partial {\mathbf x}}$. Details of the dual algorithm can be found in~\ref{mod-alt-right2}.

The objective $I(f)+\lambda R(f)$ can have many local minima due to the effect of the penalty term $R(f)$. Therefore, the Alternating Scheme~\ref{alternate-mod} is strongly dependant on the initialization step. One of such initialization procedures for the task~\eqref{MMD-task} is described in the next section.

\section{An approximate algorithm for the MMD-PCA}\label{Approximate-MMD-PCA}
Let us analyze the task~\eqref{MMD-task} in the case where $K({\mathbf x}, {\mathbf y})=({\mathbf x}\cdot {\mathbf y}) H({\mathbf x}, {\mathbf y})$ and $H$ is a Mercer kernel (by construction, $K$ is also a Mercer kernel). In this section we demonstrate that, given a distribution $f({\mathbf x})=\frac{1}{N}\sum_{i=1}^N \delta^n({\mathbf x}-{\mathbf x}_i)$, a good guess for a $k$-dimensional space in which an optimal solution is supported is a span of the first $k$ principal components of $H_f$ (see the Algorithm~\ref{approximate}). 

\begin{algorithm}
\begin{algorithmic}
\caption{An approximate algorithm for~\eqref{MMD-task} where $K({\mathbf x}, {\mathbf y})=({\mathbf x}\cdot {\mathbf y}) H({\mathbf x}, {\mathbf y})$}\label{approximate}
\State \textbf{Input:} ${\mathbf x}_1, \cdots, {\mathbf x}_N$, $f({\mathbf x})=\frac{1}{N}\sum_{i=1}^N \delta^n({\mathbf x}-{\mathbf x}_i)$

\State Calculate $H_{f} = \frac{1}{N^2}\sum_{i=1}^N\sum_{j=1}^N {\mathbf x}_i {\mathbf x}_j^T H({\mathbf x}_i, {\mathbf x}_j)$.

\State Calculate ${\rm SVD}(H_{f})$: $H_{f} = \sum_{i=1}^n \lambda_i {\mathbf v}_i {\mathbf v}_i^T$ where $\lambda_1\geq \cdots \geq \lambda_n$

\State $P\longleftarrow \sum_{i=1}^k {\mathbf v}_i {\mathbf v}_i^T$
\State \textbf{Output:} ${\mathbf x}'_i = P{\mathbf x}_i, i\in [N]$, $f'({\mathbf x}) = \frac{1}{N}\sum_{i=1}^N \delta^n({\mathbf x}-{\mathbf x}'_i)$
\end{algorithmic}
\end{algorithm}

A specifics of this type of kernels is that the MMD distance (induced by $K$) till an optimal solution of the MMD-PCA task is bounded below by the Ky Fan $k$-antinorm of $H_f$, as shown in the following theorem.

\begin{theorem}\label{MMD-repr} Let $K({\mathbf x}, {\mathbf y}) = ({\mathbf x}\cdot {\mathbf y}) H({\mathbf x}, {\mathbf y})$ where $H: {\mathbb R}^n\times {\mathbb R}^n\to {\mathbb R}$ is a Mercer kernel, ${\rm Dom}({\rm O}_K)$ is dense in $L_2({\mathbb R}^n)$ and $f$ is such that ${\rm Tr\,}H_f < \infty$. Then,
\begin{equation}
\inf_{f'\in \mathcal{G}_k}\|f-f'\|^2_{K}\geq \sum_{i=k+1}^n \lambda_i
\end{equation}
where $\|g\|_{K}^2 = \langle g| K|g\rangle$, $\lambda_1\geq \cdots \geq \lambda_n$ are eigenvalues (counting multiplicities) of $H_f$.
\end{theorem}
\begin{proof}[Sketch.]
Let us apply Theorem~\ref{square-root-operator} to the kernel $H$ and the function $f$.
Recall that an element $O\in \mathcal{B}(L^\ast_2({\mathbb R}^n), {\mathbb R}^n)$  can be identified with a vector-function $[O_i]_{i=1}^n, O_i\in L_2({\mathbb R}^n)$ where $O[\phi]_{i} ={\rm Re\,} \langle O_{i},  \phi\rangle_{L_2({\mathbb R}^{n})}$. Since $S_f$ corresponds to $[\sqrt{{\rm O}_{H}}[x_i f({\mathbf x})]]_{i=1}^n$, the representation of Theorem~\ref{square-root-operator} gives us
\begin{equation}
\begin{split}
\|H_f\|_{n-k} = \min_{S\in \mathcal{B} (L^\ast_2({\mathbb R}^{n}), {\mathbb R}^{n}), \rank S \leq k}\sum_{i=1}^n\|\sqrt{{\rm O}_{H}}[x_i f({\mathbf x})]-S_i({\mathbf x})\|^2_{L_2({\mathbb R}^n)}
\end{split}
\end{equation}
The restriction $\rank S \leq k$ is equivalent to ${\rm dim\,}\mathcal{L}_S\leq k$ where $\mathcal{L}_S = \{\sum_{i=1}^n\xi_i S_i({\mathbf x})\mid \xi_i\in {\mathbb R}\}$ and $S$ corresponds to $[S_i]_{i=1}^n$. 

Let $f'\in \mathcal{G}_k$. By Theorem~\ref{low-rank-char} we have ${\rm dim\,}{\rm span}_{\mathbb R}(\{x_1f', \cdots, x_n f'\})\leq k$. By Theorem~\ref{rank-k}, $\langle x_if'|H|x_if'\rangle$ is finite, therefore $\sqrt{{\rm O}_{H}}[x_i f'({\mathbf x})]$ can be properly defined and is in $L_2({\mathbb R}^n)$. Therefore, 
\begin{equation}
{\rm dim\,}{\rm span}_{\mathbb R}(\{\sqrt{{\rm O}_{H}}[x_1 f'({\mathbf x})], \cdots, \sqrt{{\rm O}_{H}}[x_n f'({\mathbf x})]\}\leq k.
\end{equation}
For any $f'\in \mathcal{G}_k$ one can set $S_i =  \sqrt{{\rm O}_{H}}[x_i f'({\mathbf x})]$ and search over all possible $f'\in \mathcal{G}_k$ in the minimization operator. Thus,
\begin{equation}
\begin{split}
R(f) \leq \inf_{f'\in \mathcal{G}_k}\sum_{i=1}^n\|\sqrt{{\rm O}_{H}}[x_i f({\mathbf x})]-\sqrt{{\rm O}_{H}}[x_i f'({\mathbf x})]\|^2_{L_2({\mathbb R}^n)}= \inf_{f'\in \mathcal{G}_k}\|f-f'\|^2_{K}.
\end{split}
\end{equation}
\end{proof}

Let $\mu_{\rm data}$ be a uniform distribution over $\{{\mathbf x}_i\}_{i=1}^N$ and $d_{\rm MMD}$ be the MMD distance induced by $K({\mathbf x}, {\mathbf y})=({\mathbf x}\cdot {\mathbf y}) H({\mathbf x}, {\mathbf y})$. The last theorem can be applied to a smoothed empirical distribution $f_\varepsilon({\mathbf x}) =\frac{1}{N}\sum_{i=1}^N G^n_\varepsilon ({\mathbf x}-{\mathbf x}_i)$ and then, we can send $\varepsilon\to 0$. All the more, the inequality will be satisfied if we search over $\mu\in \mathcal{P}_k$ due to $T_\mu\in \mathcal{G}_k$. Thus,
\begin{equation}
\begin{split}
\inf_{\mu\in \mathcal{P}_k}d_{\rm MMD}(\mu_{\rm data}, \mu)^2\geq \lim_{\varepsilon\to 0}\inf_{f'\in \mathcal{G}_k}\|f_\varepsilon -f'\|^2_{K}\geq \sum_{i=k+1}^n \lambda_i
\end{split}
\end{equation}
where $\lambda_1\geq \cdots \geq \lambda_n$ are eigenvalues (counting multiplicities) of $H_f$, $f({\mathbf x})=\frac{1}{N}\sum_{i=1}^N \delta^n({\mathbf x}-{\mathbf x}_i)$. Thus, $(\sum_{i=k+1}^n \lambda_i)^{1/2}$ is a lower bound of the solution of~\eqref{MMD-task}.

For such an important practical case as the HM-MMD-PCA, a multiple of the square root of the Ky Fan $k$-antinorm of $H_f$ is also an upper bound. 
\begin{theorem}\label{polynomial}
Let $H({\mathbf x}, {\mathbf y}) = P({\mathbf x}\cdot {\mathbf y})$ where $P(x) = c_0+c_1 x + \cdots + c_{l-1} x^{l-1}$, $c_i\geq 0, i\in [l-1]$, $f({\mathbf x}) = \frac{1}{N}\sum_{i=1}^N \delta^n({\mathbf x}-{\mathbf x}_i)=T_{\mu_{\rm data}}$ and $\lambda_1\geq \lambda_2\geq \cdots\geq \lambda_n$ are eigenvalues of $H_f$. Then,
\begin{equation}
\inf_{\mu\in\mathcal{P}_k}d_{\rm MMD}(\mu_{\rm data}, \mu) \leq \sqrt{l} (\sum_{i=k+1}^n \lambda_i)^{1/2}.
\end{equation}
\end{theorem}
The following corollary is straightforward from the last theorem.
\begin{corollary}\label{2-approx} A $2$-approximating solution of the task~\eqref{HM-task} can be efficiently found by the Algorithm~\ref{approximate}.
\end{corollary}
For the case when $H$ is the Gaussian kernel the situation is slightly trickier.
\begin{theorem}\label{Gaussian-MMD}
Let $H({\mathbf x}, {\mathbf y}) = e^{-\frac{\sigma^2\|{\mathbf x}-{\mathbf y}\|^2}{2}}$, $m_i = \int_{{\mathbb R}^n}  \|{\mathbf x}\|^i |f({\mathbf x})|d{\mathbf x}<\infty, i\in [4]$ and $\lambda_1\geq \lambda_2\geq \cdots\geq \lambda_n$ are eigenvalues of $H_f$. Then,
\begin{equation}
\begin{split}
\inf_{f'\in\mathcal{G}_k}\|f-f'\|_K^2 \leq 
M\sum_{j=k+1}^n \lambda_j^{1/2},
\end{split}
\end{equation}
where $M = {\mathcal O}(m_2+\sqrt{2}\sigma \sqrt{m_4 m_2}+\sqrt{2}\sigma m_3)$.
\end{theorem}
An analogous theorem can be proved for $H({\mathbf x}, {\mathbf y}) = (1+\sigma^2 \|{\mathbf x}- {\mathbf y}\|^2)^{-\frac{n+1}{2}}$, i.e. the Poisson kernel.

\section{Experiments}\label{SDR-exp}
The alternating scheme~\ref{alternate-mod} is a general optimization method that needs to be specified for every optimization task. We designed numerical specifications of the alternating scheme~\ref{alternate-mod} for all 4 optimization tasks:~\eqref{MMD-task},~\eqref{HM-task},~\eqref{Wasser-task} and~\eqref{sdr-problem} and made experiments with all of them. Details of the algorithms, i.e. numerical methods to minimize over $\phi$ and calculate $M_{\phi_t}$, can be found in Appendix. Note that for WD-PCA~\eqref{Wasser-task} we exploit the alternating scheme in the initial form (i.e.~\ref{alternate-mod}), and for MMD-PCA~\eqref{MMD-task}, HM-MMD-PCA~\eqref{HM-task} and SDR-ORF~\eqref{sdr-problem} we use the dual version of the scheme. 

{\bf Behaviour of the Gaussian MMD-PCA for small $h$.}
We studied the difference in the behavior of PCA and a solution of~\eqref{MMD-task}, for the distance function induced by the kernel $K({\mathbf x}, {\mathbf y}) = \frac{1}{\sqrt{(8\pi h^2)^n}} e^{-\frac{\| {\mathbf x}\| ^2}{8h^2}}$, obtained by the alternating scheme~\ref{alternate-mod} (AS for MMD-PCA), for the case when $h$ is small compared to the standard deviation of features.
Experiments show that they are sharply different when data points are sampled along a low-dimensional manifold ${\mathfrak M}$, which is bent globally, goes through the origin $O$ and has a large curvature at $O$ (see Fig.~\ref{tangent}). Since generated points do not lie on an affine subspace, the global nature of PCA makes it hard to interprete principal directions.

We select a smooth function $f: {\mathbb R}^{n-1}\rightarrow {\mathbb R}$, such that $f({\mathbf 0}) = 0$ and generate points in the following way: points ${\mathbf x}_{1}, {\mathbf x}_{2},\cdots ,{\mathbf x}_{N} \sim [-10,10]^{n-1}$ are sampled uniformly, after calculation of $y_{i} = f({\mathbf x}_{i})$ we add some noise: ${\mathbf z}_i = ({\mathbf x}_i, y_i) + \boldsymbol{\epsilon}_i, \boldsymbol{\epsilon}_i\sim {\mathcal N}(0,0.01 I_n)$. Both PCA and MMD-PCA are applied to the dataset (first 3 pictures on Figure~\ref{tangent}). As we see, MMD-PCA, unlike PCA, tries to catch ideal alignments of points rather that searching for a global alignment of points (which is non-existent). This property of MMD-PCA makes it a promising tool for a calculation of the tangent space to a data manifold at a given point. Fourth picture shows that when we have 2 equally important directions in data such that the first principal direction of PCA is between them (red line), and we set $k=1$, then MMD-PCA (green line) always chooses one of those directions. These experimental results are aligned with the theoretical observation given in Example~\ref{observation}, in which we show that the Gaussian MMD-PCA task for $h\to 0+$ is equivalent to finding a $k$-dimensional subspace that contains as many points of a dataset as possible. Thus, the Gaussian MMD-PCA can be considered as a method that can be potentially used to tackle the latter NP-hard problem. Some informal discussion of this problem can be found in~\cite{max-subset}.
\begin{figure*}
\begin{subfigure}{1.0\textwidth}
\centering
\begin{tabular}{cccc}
\includegraphics[scale = 0.2]{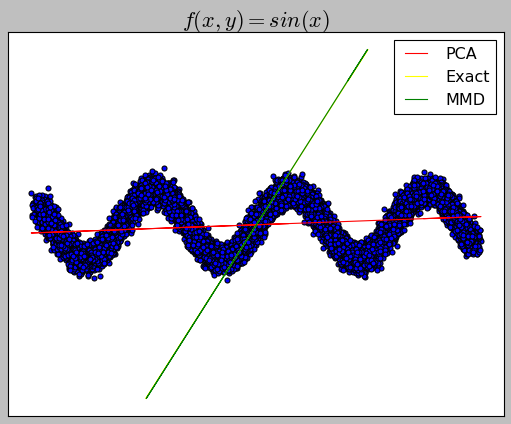} & \includegraphics[scale = 0.2]{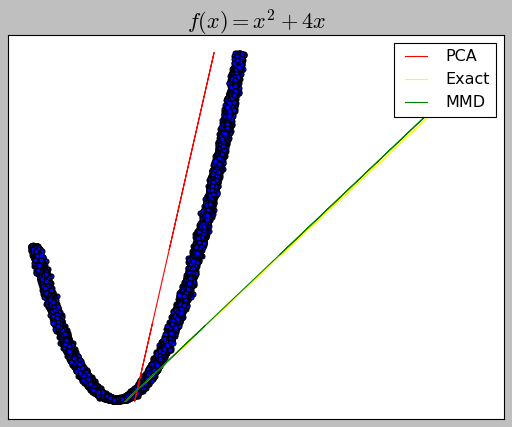} & \includegraphics[scale = 0.2]{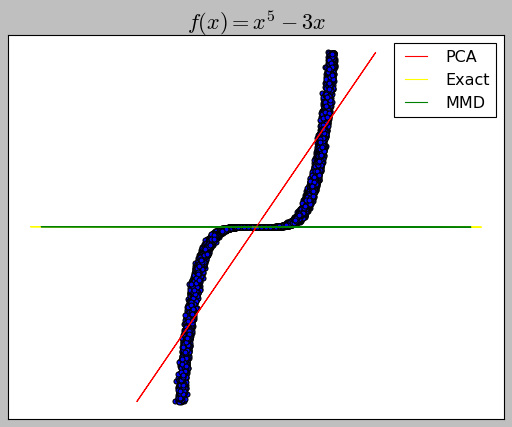} & \includegraphics[scale = 0.2]{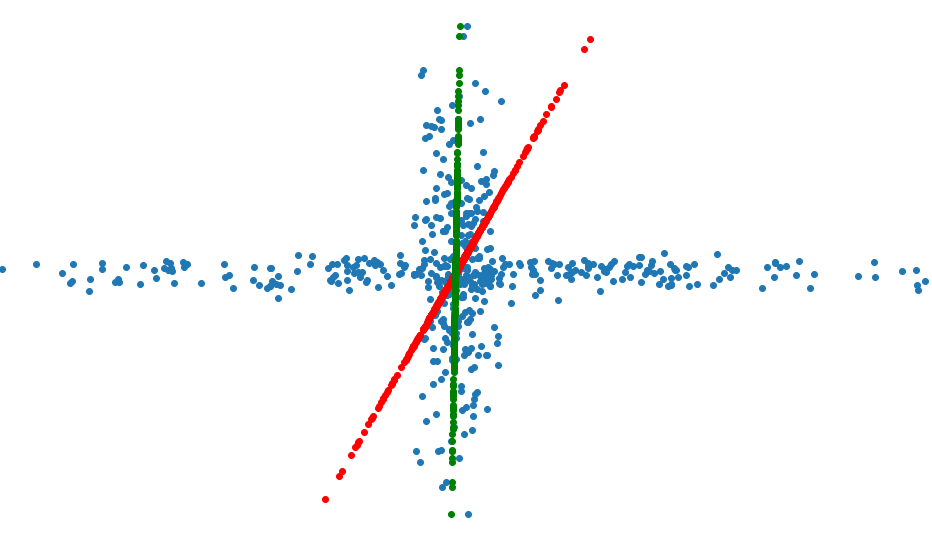}
\end{tabular}
\caption{Visualization of outputs of the PCA and MMD-PCA methods. MMD-PCA (green line) tends to select a subcollection of points that sharply aligns along the local direction (i.e. the tangent line), whereas the first principal component (red line) reflects the global shape of data.}\label{tangent}
\end{subfigure}
\begin{subfigure}{1.0\textwidth}
\centering
\begin{tabular}{cc}
\includegraphics[height=0.3\linewidth]{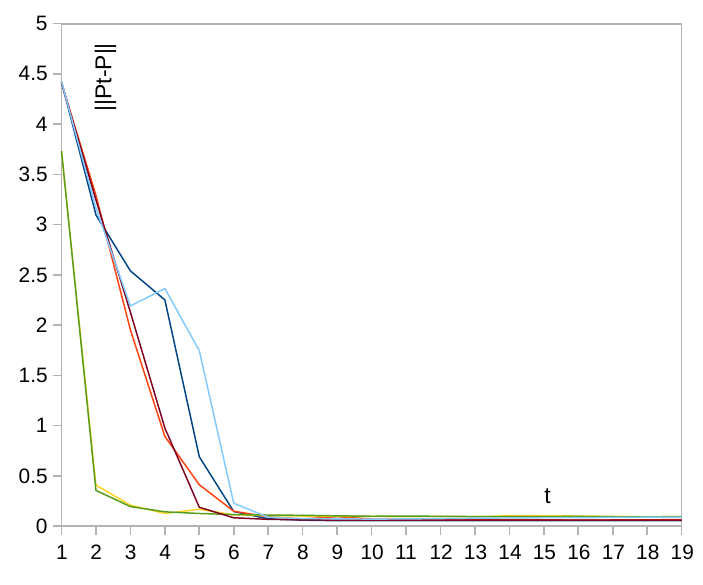} & \includegraphics[height=0.3\linewidth]{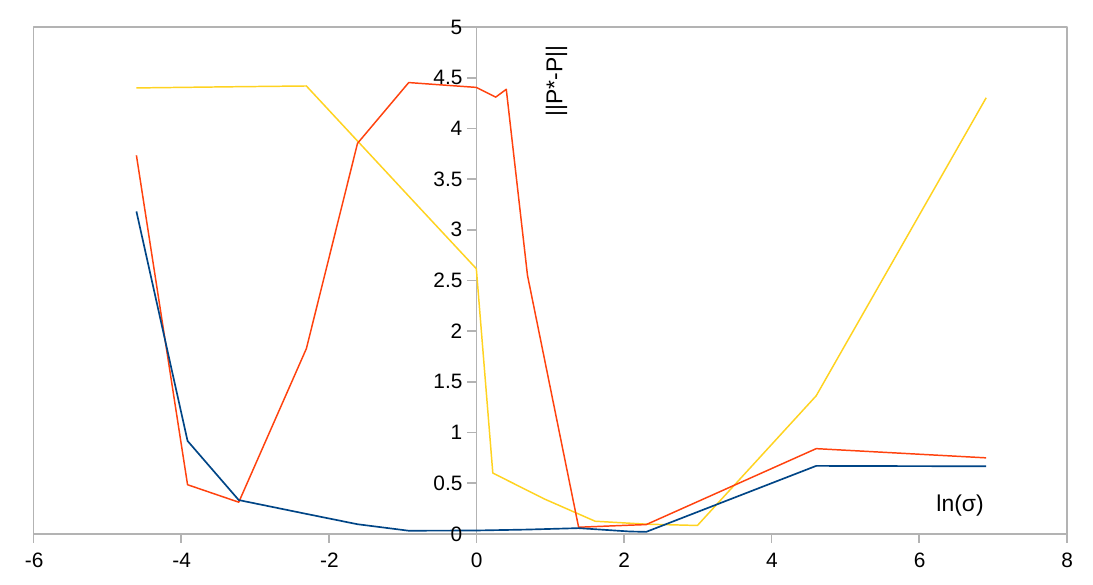} 
\end{tabular}
\caption{The dependence of $\Vert P_t-P\Vert _F$ on $t$ for different values of parameters $\updelta$ and $\lambda$. Left plot: $\Vert P_t-P\Vert _F$:  \sqbox{yellow} $\updelta = 0.05$, $\lambda = 20.0$, case I, \sqbox{green} $\updelta = 0.05, \lambda = 20.0$, case II, \sqbox{blue} $\updelta = 0.05, \lambda = 100.0$, case I, \sqbox{red} $\updelta = 0.05, \lambda = 100.0$, case II, \sqbox{babyblueeyes} $\updelta = 0.1, \lambda = 100.0$, case I, \sqbox{brown} $\updelta = 0.1, \lambda = 100.0$, case II. Right plot: $\Vert P^\ast-P\Vert _F$ as a function of $\ln \sigma$:\sqbox{red} MMD-PCA,  \sqbox{blue} HM-MMD-PCA, \sqbox{yellow} WD-PCA.}\label{dependance}
\end{subfigure}
\end{figure*}

{\bf Experiments with outlier detection (MMD-PCA, HM-MMD-PCA, WD-PCA).}
Following the experiment setup of~\cite{NIPS2010_4005}, we choose parameters $N=n=400, \updelta=0.05$ (or $0.1$), $k=10$ and generate random matrices $A\in {\mathbb R}^{N(1-\updelta)\times k}, B\in {\mathbb R}^{n\times k}$ whose entries are iid as ${\mathcal N}(0,1)$. Then, columns of the matrix $BA^T\in {\mathbb R}^{n\times N(1-\updelta)}$ (whose rank is $\leq k$) are concatenated with columns of the matrix $C\in {\mathbb R}^{n\times N\updelta}$: $X = {\rm concat} (BA^T, C)\in {\mathbb R}^{n\times N}$. The entries in $C$ are either iid as ${\mathcal N}(0,1)$ (case I) or  $N\updelta$ copies of the same vector whose entries are iid as ${\mathcal N}(0,1)$ (case II). Let $X = [{\mathbf x}_1, \cdots, {\mathbf x}_N]$, i.e. columns of $X$ are the data points.
Thus, $N(1-\updelta)$ columns of $BA^T$ lie in a $k$-dimensional subspace of ${\mathbb R}^{n}$ and $N\updelta$ columns of $C$ are outliers, and solutions of tasks~\eqref{MMD-task},~\eqref{HM-task} or~\eqref{Wasser-task} for this dataset are expected to be supported in a column space of $BA^T$.

After every iteration (step $t$ of the alternating scheme~\ref{alternate-mod}) we calculate the Frobenius distance between the projection operator $P_t$ of~\ref{alternate-mod} and the projection operator $P$ to the column space of $BA^T$, i.e. $\Vert P_t-P\Vert _F$. For the task~\eqref{Wasser-task}, the dependence of $\Vert P_t-P\Vert _F$ on $t$ for different values of parameters $\updelta$ and $\lambda$ is shown in Figure~\ref{dependance}. For tasks~\eqref{MMD-task},~\eqref{HM-task} the behaviour of the alternating scheme is similar, 7 iterations are enough to approach the optimal subspace. 

One of main goals of this experimental setup was to study how the kernel $M$, that defines the regularizer $R(f)$ by equation\eqref{R-to-M}, affects the quality of a solution. 
Besides the speed of convergence we were interested in how $\Vert P^\ast-P\Vert _F$, where $P^\ast = \lim_{t\rightarrow \infty}P_t$ is the final projection operator (e.g. $P_{20}$ in practice), depends on the parameter ${\sigma}$ of the kernel $M=G^n_{\sigma}$ (bandwidth). It is natural to expect the quality of the solution $P^\ast$ to degrade as ${\sigma}\rightarrow +\infty$ (this corresponds to $M({\mathbf x}, {\mathbf y})\rightarrow 0$), and, less trivially, as ${\sigma}\rightarrow 0$ (this corresponds to $M({\mathbf x}, {\mathbf y})\rightarrow \delta^n({\mathbf x} - {\mathbf y})$). As the right plot on Figure~\ref{dependance} shows, for the HM-MMD-PCA, the solution $P^\ast$ is close to the correct $P$ if the bandwidth $\sigma$ is in interval $[e^{-2},e^{3}]$ and it  degrades beyond that interval. For the Gaussian MMD-PCA the degrading occurs beyond $[e^{1.3},e^3]$. For the WD-PCA the interval for $\sigma$ is sligtly narrower than $[e^{1.3},e^3]$. Our numerical specification of the alternating scheme for WD-PCA involves training regularized Generative Adversarial Network (see for details~\ref{numerical-wd}) and are based on numerically unstable algorithms for the Wasserstein distance minimization.  
Finding numerical specifications for WD-PCA with a more stable behavior is a future work.

{\bf Experiments with SDR-ORF.}
We made experiments on the standard datasets, \href{https://archive.ics.uci.edu/ml/datasets/heart+disease}{Heart}, \href{https://archive.ics.uci.edu/ml/datasets/Breast+Cancer+Wisconsin+(Diagnostic)}{Breast Cancer}, \href{https://archive.ics.uci.edu/ml/datasets/Ionosphere}{Ionosphere}, \href{https://www4.stat.ncsu.edu/~boos/var.select/diabetes.html}{Diabetes}, \href{https://archive.ics.uci.edu/ml/machine-learning-databases/housing/}{Boston House Prices} and \href{https://archive.ics.uci.edu/ml/machine-learning-databases/wine/}{Wine Quality}. First, we applied the Sliced Inverse Regression algorithm (SIR)~\cite{Li91} to the training set and calculated the effective subspace for $k=2, 3$. All points were projected onto that space and we obtained two- or three-dimensional representations of input points. In the last step we applied the ten nearest neighbors algorithm (KNN) to predict outputs (based on reduced inputs) on the test set (for the regression case, the 10-KNN regression was used). The same scheme was repeated with PCA, Kernel Dimensionality Reduction (KDR) algorithm~\cite{Fukumizu} and the alternating scheme~\ref{alternate-mod} (AS) adapted for the SDR-ORF. 

We experimented with the dual version of algorithm~\ref{alternate-mod}, setting (after the data was standardized) the kernel's parameter ${\sigma} = 0.8$\footnote{Since the role of the parameter ${\sigma}$ is similar to that of the bandwidth in the kernel density estimation, we use Silverman's rule of thumb to set ${\sigma} = N^{-1/(n+4)}$.} and $\lambda = 10.0$. Details of its numerical implementation can be found in~\ref{numerical-sdr}.
In the table~\ref{SIR-KDR} one can see the obtained test set accuracy on the classification tasks and R$^2$ on the regression tasks. As we see from the table~\ref{SIR-KDR}, after reducing the dimension of an input to $k=2, 3$, we are still able to obtain good accuracy of prediction on a test set and the AS for the SDR-ORF is competitive in comparison with other methods. Note that all listed datasets are of moderate size and our Python scripts managed to compute an effective subspace in 3-5 minutes on a PC with GTX Titan X (Pascal), Intel Core i7-7700K (4.20 GHz), 64 GB RAM.
\begin{table*}
\centering
\begin{tabular}{ |c|cc|cc|cc|cc| }
 \hline 
 \backslashbox{Dataset}{Method} & PCA & & SIR & & KDR &  &  AS~\ref{alternate-mod}   &   \\ [0.5ex]
 \hline
Dimension $k$ & 2 & 3 & 2 & 3  & 2  & 3 &  2 &  3   \\ [0.5ex]
 \hline
 Heart (acc)   &  79.80 &  79.46 &   82.49   & 81.82 &   {\bf 86.33}  & {\bf  88.77}  & 81.48 & 83.50    \\
 Breast (acc) &  93.46 &  93.65 &  97.30   & 96.73 &  93.13 & 95.95  & {\bf 97.88} & {\bf  97.69}    \\
Ionosphere (acc) &  80.29 &  86.57 & {\bf 89.14} & 89.43 &  83.43 &  86.29  & 88.29 &  {\bf  90.57}   \\
Diabetes (R$^2$) &  25.34 &  28.72 & {\bf 43.47}  & 43.61 &  41.82 &  44.30 & 43.07 & {\bf 44.48}   \\
Boston (R$^2$) &  56.42 &  67.12 & 76.03 & 74.29 &  {\bf 77.88} &  {\bf  79.97} & 73.21 & 77.88   \\
Wine (R$^2$) &  93.91 &  94.12 &  {\bf 98.68} & {\bf  99.24} &  98.30 &  96.02 & 97.10 & 96.93  \\
 \hline 
\end{tabular} 
\caption{The  cross-validated accuracies/R$^2$ of KNN on 2 or 3-dimensional input representations.}
\label{SIR-KDR}
\end{table*}


{\bf Experiments with the shadow/black removal.} We made experiments with Yale B dataset~\cite{Georghiades}, which is a popular benchmark for testing robust versions of PCA. 
That dataset  contains images of 28 human subjects under 9 poses and 64 illumination conditions. 
Test images used in the experiments are  cropped and re-sized to 168x192 images, making the dimensionality of every image 32256. Thus, each human subject  corresponds to a collection of 32256-dimensional vectors that lie on some low-dimensional subspace $\mathcal{L}$ of ${\mathbb R}^{32256}$. We search for this subspace, assuming that its dimension is either 1 or 5, using  PCA and the Algorithm~\ref{approximate} with the kernel $K({\mathbf x},{\mathbf y}) = ({\mathbf x}\cdot {\mathbf y}) e^{-\frac{\|{\mathbf x}-{\mathbf y}\|^2}{n}}$ (which we simply call Gauss). Our experiments
showed that behaviour of PCA and Gauss are quite similar if the dimension of $\mathcal{L}$ is 5, though Gauss removes more shadows and preserves more details of an original image if the dimension of $\mathcal{L}$ is 1 (see Figures~\ref{opk} and~\ref{okko}). A processing of each human subject by Gauss takes seconds on Google Colab. 
\begin{table}[ht]
\centering
\addtolength{\tabcolsep}{-4pt}  
\begin{tabular}{llll}
\includegraphics[width=0.24\columnwidth]{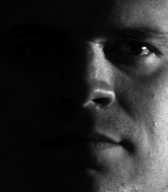} &
\includegraphics[width=0.24\columnwidth]{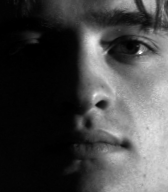} &
\includegraphics[width=0.24\columnwidth]{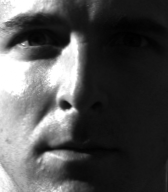} &
\includegraphics[width=0.24\columnwidth]{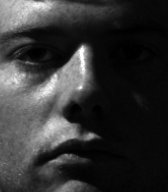} \\
\newline
\includegraphics[width=0.24\columnwidth]{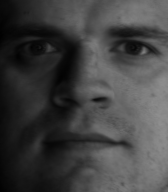} &
\includegraphics[width=0.24\columnwidth]{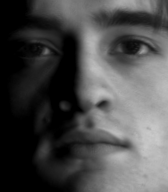} &
\includegraphics[width=0.24\columnwidth]{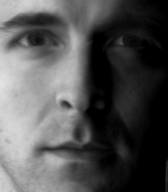} &
\includegraphics[width=0.24\columnwidth]{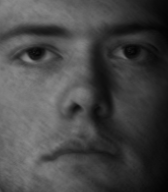} \\
\newline
\includegraphics[width=0.24\columnwidth]{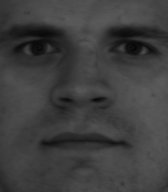}&
\includegraphics[width=0.24\columnwidth]{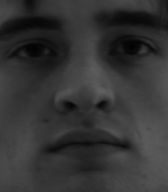}&
\includegraphics[width=0.24\columnwidth]{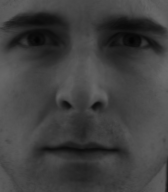}&
\includegraphics[width=0.24\columnwidth]{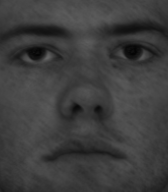} \\
\end{tabular}
\addtolength{\tabcolsep}{4pt}
\caption{Original images (the first row) and their projections to 1-dimensional subspaces computed by PCA (the second row) and computed by Gauss (the third row).}
\label{opk}
\end{table}

\begin{table}[ht]
\centering
\addtolength{\tabcolsep}{-2pt}  
\begin{tabular}{ccccccccc}
\hspace{-10pt}\includegraphics[width=0.11\columnwidth]{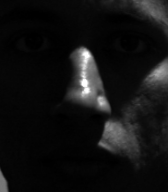} &\hspace{-10pt}
\includegraphics[width=0.11\columnwidth]{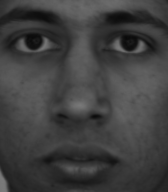} &\hspace{-10pt}
\includegraphics[width=0.11\columnwidth]{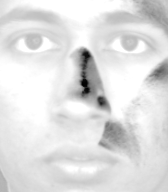}&\hspace{-10pt}
\includegraphics[width=0.11\columnwidth]{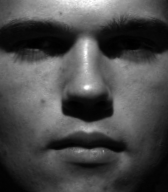} &\hspace{-10pt}
\includegraphics[width=0.11\columnwidth]{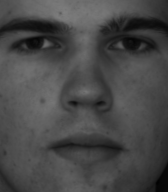} &\hspace{-10pt}
\includegraphics[width=0.11\columnwidth]{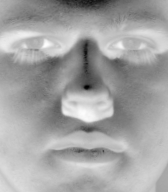}&\hspace{-10pt}
\includegraphics[width=0.11\columnwidth]{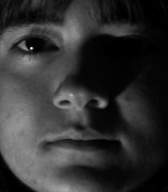} &\hspace{-10pt}
\includegraphics[width=0.11\columnwidth]{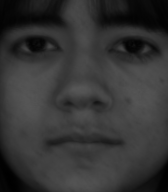} &\hspace{-10pt}
\includegraphics[width=0.11\columnwidth]{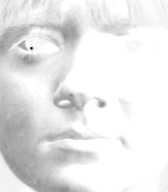}\\
\end{tabular}
\addtolength{\tabcolsep}{2pt}
\caption{Original image, projected image and the difference between them (Gauss).}
\label{okko}
\end{table}

{\bf Experiments with the background modeling.} For these experiments we use the dataset for testing background estimation algorithms SBMnet~\cite{Jodoin}. The dataset contains frames of short videos and the frame of a background for each video (so called the ground truth). Spatial resolutions of the videos vary from 240x240 to 800x600. Thus, a collection of frames of every video is a set of high-dimensional vectors (with a dimension up to 480000) that, again, lie on a low dimensional subspace $\mathcal{L}$. We assume that the dimension of $\mathcal{L}$ is 5. 
We recover $\mathcal{L}$ using PCA and the Algorithm~\ref{approximate} for  kernels $K({\mathbf x},{\mathbf y}) = \sum_{i=1}^4 (\frac{{\mathbf x}\cdot {\mathbf y}}{n})^{i}$, $K({\mathbf x},{\mathbf y}) = ({\mathbf x}\cdot {\mathbf y}) e^{-\frac{a\|{\mathbf x}-{\mathbf y}\|^2}{n}}$, $K({\mathbf x},{\mathbf y}) = ({\mathbf x}\cdot {\mathbf y}) e^{-\frac{a\|{\mathbf x}-{\mathbf y}\|}{\sqrt{n}}}$ and $K({\mathbf x},{\mathbf y}) = ({\mathbf x}\cdot {\mathbf y}) (1+\frac{a\|{\mathbf x}-{\mathbf y}\|^2}{n})^{-\frac{n+1}{2}}$ (which we simply call Kurtosis, Gauss, Laplace and Poisson respectively). Recall that, according to corollary~\ref{2-approx}, this algorithm is 2-approximating for Kurtosis. Subsequently, we compute the median of the vectors, projected onto $\mathcal{L}$, and define the latter to be the recovered background image (see Figure~\ref{prd}). Measures of consistency with the ground truth backgrounds are then calculated using Python scripts downloaded from~\cite{scenebackgroundmodeling}. Six measures are used: AGE (average of the gray-level absolute difference between a ground truth image and a computed background image), pEPs (percentage of pixels in  a computed background whose value differs from the value of the corresponding pixel in a ground truth by more than a threshold), pCEPS (percentage of pixels whose 4-connected neighbors are also error pixels), MSSSIM (estimate of the perceived visual distortion), PSNR (Peak-Signal-to-Noise-Ratio, or $10\log_{10}(\frac{(L-1)^2}{MSE})$ where $L$ is the maximum number of grey levels and MSE is the mean squared error between a ground truth and a computed background images), CQM (Color image Quality Measure). Codes that compute listed metrics can be found in~\cite{scenebackgroundmodeling}. As shown on Table~\ref{SBM}, experiments again demonstrated very similar behavior of PCA, Kurtosis, Gauss, Laplace and Poisson with very close accuracies 
of the background reconstruction. Best measures of consistency with the ground truth images were achieved for Gauss ($a=5.0, 25.0$) and Laplace ($a=5.0$). For a comparison with other methods, we also give accuracies of other methods based on low-rank approximation and an accuracy of a state-of-the-art method that was specifically tailored for that task~\cite{Javed}. On figure~\ref{backgrounds} one can see that background images computed by PCA and Gauss are almost identical, though Gauss is less likely than PCA to add local artefacts, such as blurs, noise etc.  
\begin{table}[ht]
\centering
\addtolength{\tabcolsep}{-4pt}  
\begin{tabular}{ccc}
\includegraphics[width=0.32\columnwidth]{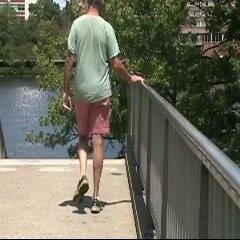} &
\includegraphics[width=0.32\columnwidth]{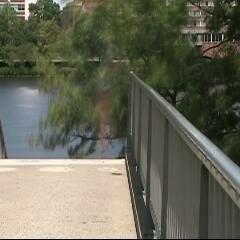} &
\includegraphics[width=0.32\columnwidth]{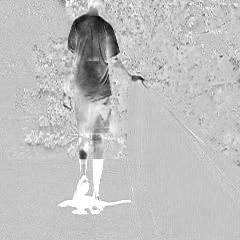}\\
\newline
\end{tabular}
\addtolength{\tabcolsep}{4pt}
\caption{Original image, its projection, and their grayscale difference (Gauss).}
\label{prd}
\end{table}

\begin{table}[ht]
\centering
\addtolength{\tabcolsep}{-4pt}  
\begin{tabular}{cccc}
& \multicolumn{2}{c}{Input images} & \\
\newline
\includegraphics[width=0.23\columnwidth]{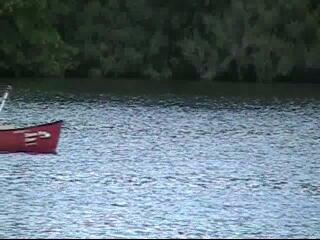} &
\includegraphics[width=0.25\columnwidth]{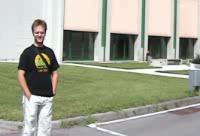} &
\includegraphics[width=0.21\columnwidth]{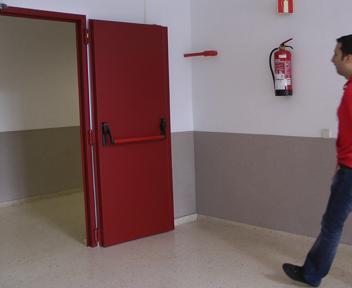} &
\includegraphics[width=0.255\columnwidth]{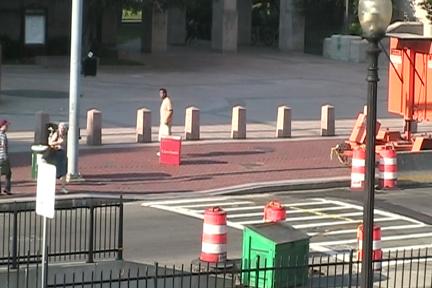}\\
\newline
\includegraphics[width=0.23\columnwidth]{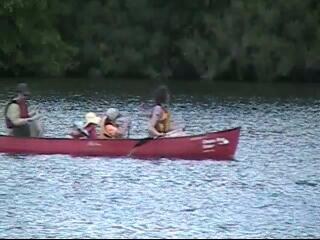} &
\includegraphics[width=0.25\columnwidth]{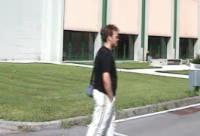} &
\includegraphics[width=0.21\columnwidth]{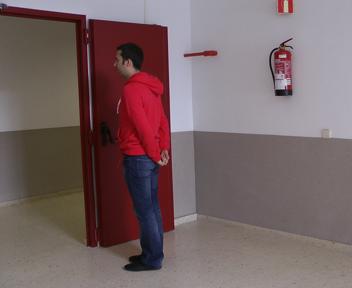} &
\includegraphics[width=0.255\columnwidth]{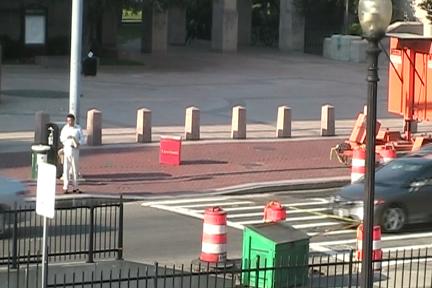}\\
\newline
\includegraphics[width=0.23\columnwidth]{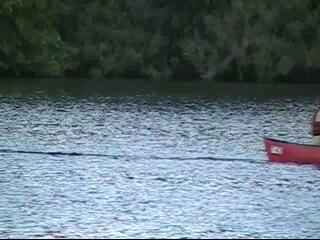} &
\includegraphics[width=0.25\columnwidth]{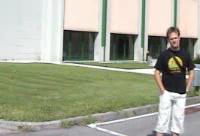} &
\includegraphics[width=0.21\columnwidth]{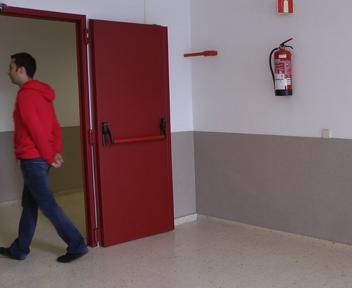} &
\includegraphics[width=0.255\columnwidth]{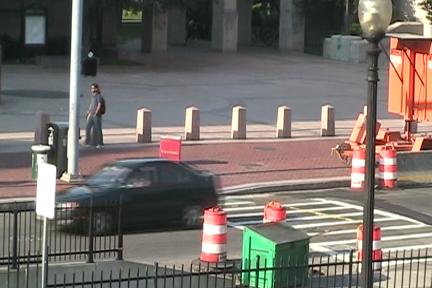}\\
\newline
& \multicolumn{2}{c}{A background computed by PCA} & \\
\newline
\includegraphics[width=0.23\columnwidth]{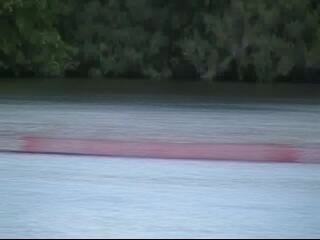} &
\includegraphics[width=0.25\columnwidth]{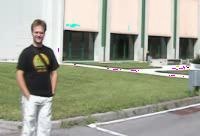} &
\includegraphics[width=0.21\columnwidth]{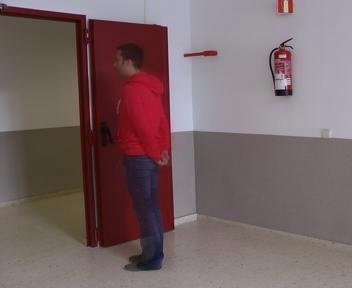} &
\includegraphics[width=0.255\columnwidth]{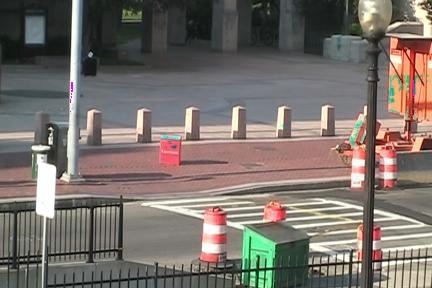}\\
\newline
& \multicolumn{2}{c}{A background computed by Gauss} & \\
\newline
\includegraphics[width=0.23\columnwidth]{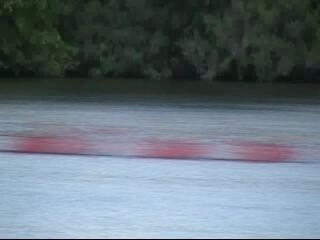} &
\includegraphics[width=0.25\columnwidth]{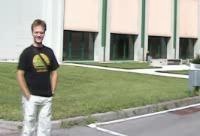} &
\includegraphics[width=0.21\columnwidth]{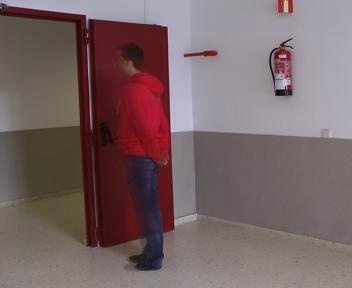} &
\includegraphics[width=0.255\columnwidth]{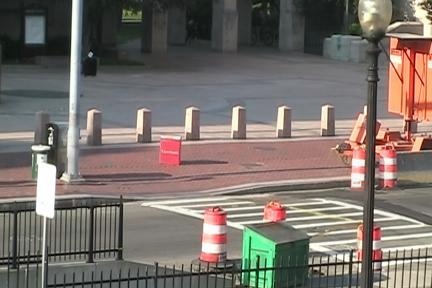}\\
\end{tabular}
\addtolength{\tabcolsep}{4pt}
\caption{Computed backgrounds are almost identical, though noise is more often in PCA's output.}
\label{backgrounds}
\end{table}

\begin{table*}
\centering
\scriptsize
\begin{tabular}{ |c|c|c|c|c|c|c|}
 \hline 
 Method & AGE & pEPs & pCEPS & MSSSIM & PSNR & CQM  \\ [0.5ex]
 \hline
MSCL (SOTA)~\cite{Javed} & {\bf 5.9547}	& {\bf 0.0524} & {\bf 0.0171} & 0.9410 & {\bf 30.8952} & {\bf 31.7049} \\
 \hline
BRTF~\cite{Qibin}  & 9.5385 & 0.1140 & 0.0876 & {\bf 0.9621} & 28.4655 & 29.3246   \\ [0.5ex]
 \hline
GoDec~\cite{Tianyi}  & 11.5934 & 0.1584 & 0.0974 & 0.8854 & 24.9954 & 25.9955 \\
\thickhline
PCA  & 9.3774  & 0.0904 & 0.0522 & 0.9027 & 26.1549 & 27.5052 \\
 \hline
Kurtosis  & 9.3509 & 0.0936 & 0.0544 &  0.9032 & 26.1475 & 27.468 \\
 \hline
Gauss ($a=0.2$)  & 9.251  & 0.09 &   0.0521 &  0.9025 & 26.1649 & 27.464 \\
 \hline
Gauss ($a=1.0$)  & 8.8679 & 0.0876 & 0.05  &   0.9049 &   26.7391 &    28.0609 \\
 \hline
Gauss ($a= 5.0$) & 8.85  &   0.0876 &  {\bf 0.0497} & 0.9045 & 26.7254 & 28.0586\\
\hline
Gauss ($a= 25.0$) & 8.8781 & 0.0886 & 0.0509 & {\bf 0.9065} & {\bf 27.0038} & {\bf 28.3913}\\
 \hline
Laplace ($a=0.2$) & 9.0745 & 0.089 &  0.0511 & 0.9032 & 26.3269 & 27.619\\
 \hline
Laplace ($a=1.0$) & 8.9428 & 0.088 & 0.0505 & 0.904 &  26.5121 & 27.819\\
 \hline
Laplace ($a=5.0$) & {\bf 8.8424} & {\bf 0.0873} & 0.0498 & 0.906 &  26.8228 & 28.1728\\
 \hline
Poisson ($a=0.04$) & 9.2481 & 0.0905 & 0.0523 & 0.9021 & 26.173 & 27.4645\\
 \hline
Poisson ($a=1.0$) & 9.2483 & 0.0906 & 0.0523 & 0.9022 & 26.173 & 27.4644\\
 \hline
Poisson ($a=25.0$) & 9.2481 & 0.0906 & 0.0523 & 0.9021 & 26.173 & 27.4646\\
 \hline
\end{tabular} 
\caption{Measures of consistency with the ground truth background image for the SBMnet dataset.}
\label{SBM}
\end{table*}

The processing of the whole SBMnet dataset using PCA/Kurtosis/Gauss/ Laplace takes approximately the same time --- 25 minutes on a cluster equipped with Intel Xeon Platinum 8168 Processors (33M Cache, 2.70 GHz) and 1TB RAM. 
The code is available on \href{https://github.com/k-nic/Alternating-Scheme}{github} to facilitate the reproducibility of our results.

{\bf Scalability of algorithms.} A major practical limitation of the alternating scheme~\ref{alternate-mod} comes from the fact that it involves an optimization over a set of functions $\mathfrak{F}$, which in applications is either a feedforward neural network (as in our specifications of the AS for SDR-ORF, MMD-PCA, HM-MMD-PCA) or a generative neural network (WD-PCA). A speed of optimization is also strongly dependant on the objective's landscape.   
Thus, for large scale datasets, with a dimension of vectors $\gg 10^3$, and a sophisticated structure of a regression function (SDR-ORF) or a data distribution (MMD-PCA, HM-MMD-PCA, WD-PCA), the alternating scheme is substantially slower in comparison with other popular methods  (such as PCA for the UDR, or SIR/KDR for the SDR).

In the special case of MMD-PCA (that includes HM-MMD-PCA), the approximate algorithm~\ref{approximate} can be used as a surrogate of the alternating scheme. It requires the same time as PCA and can be applied to datasets with a dimension of vectors $\sim 10^6-10^7$. Also, the Algorithm~\ref{approximate} can be used for an initialization of the alternating scheme.
\section{Conclusions}
We develop a new optimization framework for LDR problems. 
The alternating scheme for the optimization task demonstrates both the computational efficiency and the applicability to real-world data. The algorithm performs quite stably when we vary most of the hyperparameters, though it crucially depends on two parameters, the bandwidth of the ``smoothing'' kernel $M$, $\sigma$, and the penalty parameter $\lambda$.  We believe that the MMD-PCA/HM-MMD-PCA/WD-PCA methods for UDR could be used as an alternative to PCA in study fields in which data demonstrate ``heavy-tailed'' and ``non-Gaussian'' behavior, such as financial applications or computer vision. Also, our formulation of SDR-ORF is free from any assumptions on the distribution of input-output pairs, which makes it an alternative to other methods of efficient subspace estimation.
A more detailed report on these topics is a subject of future research.

\bibliographystyle{unsrt}  
\newcommand{\noopsort}[1]{}

\appendix

\section{Proofs for section~$\lowercase{\ref{classes}}$}
\subsection{Proof of Theorem~$\lowercase{\ref{easy}}$: given for completeness}
\begin{proof}
The inclusion $\mathcal{G}'_k \subseteq \overline{\mathcal{G}_k}^\ast $ follows from a well-known fact that ${\mathcal S}({\mathbb R}^{k})$ is dense in ${\mathcal S}'({\mathbb R}^{k})$. I.e. for any $f\in {\mathcal S}'({\mathbb R}^{k})$ one can always find a sequence $\{f_i\}\subseteq {\mathcal S}({\mathbb R}^{k})$ such that $T_{f_i} \rightarrow^\ast f$. Therefore, for any $(f\otimes \delta^{n-k})_U\in \mathcal{G}'_k$ there is a sequence $\{(T_{f_i}\otimes \delta^{n-k})_U\}\subseteq \mathcal{G}_k$ such that $(T_{f_i}\otimes \delta^{n-k})_U\rightarrow^\ast (f\otimes \delta^{n-k})_U$. Thus, $\mathcal{G}'_k \subseteq \overline{\mathcal{G}_k}^\ast $.

Since $\mathcal{G}_k \subseteq \mathcal{G}'_k$, to prove $\mathcal{G}'_k = \overline{\mathcal{G}_k}^\ast $ it is enough to show that $\mathcal{G}'_k$ is sequentially closed. 

We need a simple fact from a theory of distributions.
\begin{lemma}\label{steingauz}
If $T_i\rightarrow^\ast T$  and $\phi_i\rightarrow \phi$, then $\langle T_i, \phi_i\rangle \rightarrow \langle T, \phi\rangle$.
\end{lemma}
\begin{proof}[Proof of Lemma] Schwartz space ${\mathcal S}({\mathbb R}^{n})$ is a Fr{\'e}chet space, therefore the Banach-Steinhaus theorem applies to ${\mathcal S}'({\mathbb R}^{n})$. Since $T_i\rightarrow^\ast T$, we have  $\sup_i \vert \langle T_i, \phi\rangle\vert  <\infty$ for any $\phi\in {\mathcal S}({\mathbb R}^{n})$. From the Banach-Steinhaus theorem, applied to a set $\{T_i\}_1^\infty$, we obtain for any $\epsilon>0$, there is a neighbourhood $U$ of ${\mathbf 0}\in {\mathcal S}({\mathbb R}^{n})$ such that $\vert \langle T_i, \phi\rangle\vert <\epsilon$ whenever $\phi\in U$. Thus, $\vert \langle T_i, \phi_i-\phi\rangle\vert <\epsilon$ for a  large enough $i$. From that we conclude that $\langle T_i, \phi_i\rangle \rightarrow \langle T, \phi\rangle$.
\end{proof}
For any $T\in {\mathcal S}'({\mathbb R}^{n})$ and $\psi\in {\mathcal S}({\mathbb R}^{n-k})$, let us define $T^\psi\in {\mathcal S}'({\mathbb R}^{k})$ as $\langle T^\psi, \phi\rangle = \langle T, \phi\otimes \psi\rangle$.

Suppose that $\{f_i\}_1^\infty\subseteq {\mathcal S}'({\mathbb R}^{k})$, $\{U_i\}_1^\infty$ are such that $(f_i\otimes \delta^{n-k})_{U_i}\rightarrow^\ast f$. We need to prove that $f\in \mathcal{G}'_k$. Since a set of orthogonal matrices is compact, then one can always find a subsequence $\{U_{n_i}\}$ such that $U_{n_i}\rightarrow U$. Since $(f_{n_i}\otimes \delta^{n-k})_{U_{n_i}}\rightarrow^\ast f$ and $\phi(U_{n_i}{\mathbf x})\rightarrow \phi(U{\mathbf x})$ (for any fixed $\phi\in {\mathcal S}({\mathbb R}^{n})$), using lemma~\ref{steingauz} we obtain:
\begin{equation}
\begin{split}
\langle f_{n_i}\otimes \delta^{n-k}, \phi\rangle = 
\langle (f_{n_i}\otimes \delta^{n-k})_{U_{n_i}}, \phi(U_{n_i}{\mathbf x}) \rangle \rightarrow \langle f, \phi(U{\mathbf x}) \rangle = \langle f_{U^T}, \phi({\mathbf x}) \rangle
\end{split}
\end{equation}
Thus, we have $
f_{n_i}\otimes \delta^{n-k} \rightarrow^\ast f_{U^T}$. 
From the last we see that $f_{n_i}  \rightarrow^\ast f^\psi_{U^T}$ where $\psi$ is such that $\psi({\mathbf 0})=1$.
Therefore, $f_{U^T} = f^\psi_{U^T} \otimes \delta^{n-k}$ and $f = (f^\psi_{U^T} \otimes \delta^{n-k})_U\in \mathcal{G}'_k$.
\end{proof}

\subsection{Proof of Theorem~$\lowercase{\ref{low-rank-char}}$}
\begin{proof}[Proof of Theorem~\ref{low-rank-char} ($\Rightarrow$)]
Let us prove that from $T = (f\otimes \delta^{n-k})_U$, $f\in {\mathcal S}'({\mathbb R}^{k})$, $U^TU=I_n$ it follows that $\dim {\rm span}_{\mathbb R} \{x_1 T, x_2 T, \cdots, x_n T\}\leq k$. 

It is easy to see that $x_i [f\otimes \delta^{n-k}] = 0$ if $i>k$. If $U = [{\mathbf u}_1, \cdots, {\mathbf u}_n]^T$, then for $i>k$ we have $0 = (x_i [f\otimes \delta^{n-k}])_U = {\mathbf u}^T_i{\mathbf x}(f\otimes \delta^{n-k})_U = {\mathbf u}^T_i{\mathbf x} T$.

Thus, we have $n-k$ orthogonal vectors, ${\mathbf u}_{k+1}, \cdots, {\mathbf u}_{n}$, such that 
\begin{equation}
\begin{bmatrix}
x_1 T& \cdots & x_n T
\end{bmatrix}{\mathbf u}_i = 0.
\end{equation} 
Using standard linear algebra we obtain there are at most $k'$ distributions $x_{i_1}T, \cdots, x_{i_{k'}}T, k'\leq k$ that form a basis of ${\rm span}_{{\mathbb R}} \{x_i T\}^n_{1}$. 
\end{proof}
To prove the second part of theorem we need the following lemma.
\begin{lemma}\label{lemma} If $T\in  {\mathcal S}' ({\mathbb R}^n)$ is such that $y_i T = 0$ for any $i>k$, then $T\in \mathcal{G}'_k$.
\end{lemma}
\begin{proof}[Proof of lemma] 
Recall from functional analysis, for $f\in {\mathcal S}' ({\mathbb R}^n)$, the tempered distribution $\frac{\partial f}{\partial x_i}$ is defined by the condition $\langle \frac{\partial f}{\partial x_i}, \phi\rangle = -\langle f, \frac{\partial \phi}{\partial x_i}\rangle$.
Once the Fourier transform is applied, our lemma's dual version is equivalent to the following formulation: if $\frac{\partial f}{\partial x_i} = 0, i>k$, then $f\in \overline{\mathcal{F}_k}^\ast$. Let us prove it in this formulation. 

Recall that a set of infinitely differentiable functions with a compact support is denoted by $C^\infty_c ({\mathbb R})$.
Suppose $\phi\in {\mathcal S} ({\mathbb R}^n)$ and $p\in  C^\infty_c ({\mathbb R})$ are chosen in such a way that $\int^{\infty}_{-\infty} p(y_i) d y_i =1$, $\textbf{supp\,} p \subseteq [A, B]$. Let us define:
\begin{equation}
r({\mathbf x}) = \int^{x_i}_{-\infty} \phi({\mathbf x}_{-i}, y_i) d y_i - \int^{x_i}_{-\infty} p(y_i) d y_i \int^{\infty}_{-\infty} \phi({\mathbf x}_{-i}, y_i) d y_i
\end{equation}
It is easy to see that for any $\alpha\in {\mathbb N}^{n-1}, \alpha'\in {\mathbb N}, \beta\in {\mathbb N}^{n-1}, \beta'\in {\mathbb N}$ we have (at least one derivative over $x_i$ is present):
\begin{equation}
\begin{split}
{\mathbf x}^\alpha_{-i} x^{\alpha'}_i \frac{\partial^{\beta, 1+\beta'} r}{\partial {\mathbf x}^\beta_{-i} \partial x^{1+\beta'}_i}
 =  {\mathbf x}^\alpha_{-i} x^{\alpha'}_i \frac{\partial^{\beta, \beta'} [\phi({\mathbf x}) - p(x_i) \int^{\infty}_{-\infty} \phi({\mathbf x}_{-i}, y_i) d y_i]}{\partial {\mathbf x}^\beta_{-i} \partial x^{\beta'}_i} =\\
 {\mathbf x}^\alpha_{-i} x^{\alpha'}_i \frac{\partial^{\beta, \beta'} \phi({\mathbf x})}{\partial {\mathbf x}^\beta_{-i} \partial x^{\beta'}_i} - x^{\alpha'}_i \frac{\partial^{\beta'} p(x_i)}{\partial x^{\beta'}_i} \int^{\infty}_{-\infty} {\mathbf x}^\alpha_{-i}  \frac{\partial^{\beta} \phi({\mathbf x}_{-i}, y_i)}{\partial {\mathbf x}^\beta_{-i}} d y_i
\end{split}
\end{equation}
The terms ${\mathbf x}^\alpha_{-i} x^{\alpha'}_i \frac{\partial^{\beta, \beta'} \phi({\mathbf x})}{\partial {\mathbf x}^\beta_{-i} \partial x^{\beta'}_i}$ and $x^{\alpha'}_i \frac{\partial^{\beta'} p(x_i)}{\partial x^{\beta'}_i}$ are bounded by the definition of ${\mathcal S} ({\mathbb R}^n), C^\infty_c ({\mathbb R})$. The boundedness of $\int^{\infty}_{-\infty} {\mathbf x}^\alpha_{-i}  \frac{\partial^{\beta} \phi({\mathbf x}_{-i}, y_i)}{\partial {\mathbf x}^\beta_{-i}} d y_i$ is a consequence of the inequality  
$\vert {\mathbf x}^\alpha_{-i}  \frac{\partial^{\beta} \phi({\mathbf x}_{-i}, y_i)}{\partial {\mathbf x}^\beta_{-i}}\vert  \leq \frac{C}{1+y_i^2} $ (which holds because $\phi\in {\mathcal S} ({\mathbb R}^n)$).

Analogously (when not a single derivative over $x_i$ is present):
\begin{equation}
\begin{split}
{\mathbf x}^\alpha_{-i} x^{\alpha'}_i \frac{\partial^{\beta} r}{\partial {\mathbf x}^\beta_{-i} }
 = x^{\alpha'}_i \int^{x_i}_{-\infty}  {\mathbf x}^\alpha_{-i} \frac{\partial^{\beta} \phi({\mathbf x}_{-i}, y_i)}{\partial {\mathbf x}^\beta_{-i}} d y_i - 
x^{\alpha'}_i\int^{x_i}_{-\infty} p(y_i) d y_i \int^{\infty}_{-\infty} {\mathbf x}^\alpha_{-i} \frac{\partial^{\beta} \phi({\mathbf x}_{-i}, y_i)}{\partial {\mathbf x}^\beta_{-i}} d y_i = \\
 = x^{\alpha'}_i (1-\int^{x_i}_{-\infty} p(y_i) d y_i ) \int^{x_i}_{-\infty}  {\mathbf x}^\alpha_{-i} \frac{\partial^{\beta} \phi({\mathbf x}_{-i}, y_i)}{\partial {\mathbf x}^\beta_{-i}} d y_i - 
x^{\alpha'}_i \int^{x_i}_{-\infty} p(y_i) d y_i \int^{\infty}_{x_i}  {\mathbf x}^\alpha_{-i} \frac{\partial^{\beta} \phi({\mathbf x}_{-i}, y_i)}{\partial {\mathbf x}^\beta_{-i}} d y_i 
\end{split}
\end{equation}
The second term is 0 when $x_i\leq A$. It is also bounded when $x_i > A$ because $\vert {\mathbf x}^\alpha_{-i} \frac{\partial^{\beta} \phi({\mathbf x}_{-i}, y_i)}{\partial {\mathbf x}^\beta_{-i}}\vert \leq \frac{C'}{(1+y^2_i)^{\alpha'+1}}$. Therefore,
\begin{equation}
\left\vert  x^{\alpha'}_i \int\limits^{\infty}_{x_i}  {\mathbf x}^\alpha_{-i} \frac{\partial^{\beta} \phi({\mathbf x}_{-i}, y_i)}{\partial {\mathbf x}^\beta_{-i}}d y_i\right\vert  \leq \vert x_i\vert ^{\alpha'} \int\limits^{\infty}_{x_i}\frac{C'}{(1+y^2_i)^{\alpha'+1}}d y_i.
\end{equation}
The latter is bounded because $\lim_{x_i\rightarrow +\infty} \vert x_i\vert ^{\alpha'} \int^{\infty}_{x_i}\frac{C'}{(1+y^2_i)^{\alpha'+1}}d y_i = 0$.

The first term is 0 when $x_i \geq B$ and for $x_i < B$ it satisfies:
\begin{equation}
\left\vert  x^{\alpha'}_i \int\limits^{x_i}_{-\infty}  {\mathbf x}^\alpha_{-i} \frac{\partial^{\beta} \phi({\mathbf x}_{-i}, y_i)}{\partial {\mathbf x}^\beta_{-i}}d y_i\right\vert  \leq \vert x_i\vert ^{\alpha'} \int\limits^{x_i}_{-\infty}  \frac{C'}{(1+y^2_i)^{\alpha'+1}}d y_i.
\end{equation}
The latter is also bounded, since $\lim_{x_i\rightarrow -\infty} \vert x_i\vert ^{\alpha'} \int^{x_i}_{-\infty}\frac{C'}{(1+y^2_i)^{\alpha'+1}}d y_i = 0$.

Thus, ${\mathbf x}^\alpha \frac{\partial^{\beta} r({\mathbf x})}{\partial {\mathbf x}^\beta}$ is bounded and $r\in {\mathcal S} ({\mathbb R}^n)$. Therefore, $\frac{\partial f}{\partial x_i} = 0$ implies:
\begin{equation}
\langle f, \frac{\partial r}{\partial x_i}\rangle = 0 \Rightarrow f[\phi] = f[p(x_i) \int^{\infty}_{-\infty} \phi({\mathbf x}_{-i}, y_i) d y_i].
\end{equation}
Since this sequence of arguments works for any $i>k$, we can apply them sequentially to initial $\phi\in {\mathcal S} ({\mathbb R}^n)$ w.r.t. $x_{k+1}, ..., x_{n}$. Thus, for any $p_{k+1}, ..., p_{n}\in C_c ({\mathbb R})$ such that $\int^{\infty}_{-\infty} p_i(y_i) d y_i =1$ we obtain:
\begin{equation}
f[\phi] = f[p_{k+1}(x_{k+1}) \cdots p_{n}(x_{n}) \int\limits_{{\mathbb R}^{n-k}} \phi({\mathbf x}_{1:k}, {\mathbf x}_{k+1:n}) d {\mathbf x}_{k+1:n}].
\end{equation}
Moreover, since $C^\infty_c ({\mathbb R})$ is dense in ${\mathcal S} ({\mathbb R})$, we can assume that $p_{k+1}, ..., p_{n}\in {\mathcal S} ({\mathbb R})$.
For the inverse Fourier transform $T = \mathcal{F}^{-1}[f]$ the latter condition becomes equivalent to:
\begin{equation}
\langle T, \phi\rangle = \langle T, p'_{k+1}(x_{k+1}) \cdots p'_{n}(x_{n}) \phi({\mathbf x}_{1:k}, {\mathbf 0}_{k+1:n}) \rangle
\end{equation}
for any $p'_{k+1}, ..., p'_{n}\in {\mathcal S} ({\mathbb R})$ such that $p'_i(0)=1$. Let us define $p'_i(x_i) = e^{-x^2_i}$. It is easy to check that $T =  g \otimes \delta^{n-k}$ where $g\in {\mathcal S}' ({\mathbb R}^k), \langle g, \psi\rangle = \langle T, e^{-\vert {\mathbf x}_{k+1:n}\vert ^2} \psi({\mathbf x}_{1:k}) \rangle$ for $\psi\in {\mathcal S} ({\mathbb R}^k)$. Thus, $T\in \mathcal{G}'_k$ and lemma is proved.
\end{proof}
\begin{proof} [Proof of Theorem~\ref{low-rank-char} ($\Leftarrow$)] If $\dim {\rm span}_{\mathbb R} \{x_1 T, x_2 T, \cdots, x_n T\}\leq k$, then 
\begin{equation}
\dim \{{\mathbf v}\in {\mathbb R}^n\vert  [x_1 T, \cdots, x_n T]{\mathbf v} = 0\}\geq n-k.
\end{equation} 
Thus, there exist at least $n-k$ orthonormal vectors ${\mathbf v}_{k+1}, \cdots, {\mathbf v}_{n}$, such that $[x_1 T, \cdots, x_n T]{\mathbf v}_i = 0$. Therefore, $[x_1 T, \cdots, x_n T]{\mathbf v}_i = ({\mathbf v}^T_i {\mathbf x}) T = 0$. 

Let us complete ${\mathbf v}_{k+1}, \cdots, {\mathbf v}_{n}$ to form an orthonormal basis of ${\mathbb R}^n$: ${\mathbf v}_{1}, \cdots, {\mathbf v}_{n}$. Let us define a matrix $V = [{\mathbf v}_{1}, \cdots, {\mathbf v}_{n}]$. It is easy to see that:
\begin{equation}
\left(({\mathbf v}^T_i {\mathbf x}) T\right)_V = ({\mathbf v}^T_i V {\mathbf x}) T_V = x_i T_V
\end{equation}
Since for $i>k$ we have $({\mathbf v}^T_i {\mathbf x}) T = 0$, then $x_i T_V = 0$. Using lemma~\ref{lemma} we obtain $T_V\in \mathcal{G}'_k$. Therefore, $(T_V)_{V^T} = T\in \mathcal{G}'_k$. Theorem proved.
\end{proof}

\section{Structure of WD-PCA}\label{WD-Villani}
Recall that $({\mathbb R}^n, \Vert \cdot \Vert )$ is a 
Banach space and $p\geq 1$. 
Now, let us consider an optimization problem: for a given $X \in {\mathbb R}^{n\times N}$ solve
\begin{equation}\label{robust-pca}
\Vert X-L\Vert_p \rightarrow \min_{{\rm rank} (L)\leq k}
\end{equation}
where $\Vert \cdot\Vert_p $ is a norm on ${\mathbb R}^{n\times N}$ defined by $\Vert [{\mathbf s}_1, \cdots, {\mathbf s}_N]\Vert  \eqdef (\sum_{i=1}^N \Vert {\mathbf s}_i\Vert^p)^{1/p} $.

The following simple theorem shows that the two tasks are connected so that the solution of one directly leads to another's solution.
\begin{theorem}\label{transport}
Given data points $\{{\mathbf x}_1, \cdots, {\mathbf x}_N\}$, let $X = [{\mathbf x}_1, \cdots, {\mathbf x}_N]\in {\mathbb R}^{n\times N}$. Then, 
\begin{equation}
\min_{\nu\in {\mathcal P}_k}W_p (\mu_{\rm{data}}, \nu) = \frac{1}{N^p} \min_{Y\in {\mathbb R}^{n\times N}, {\rm rank}(Y)\leq k}{\Vert X-Y\Vert_p }.
\end{equation}
Moreover, $\min_{\nu\in {\mathcal P}_k}W_p (\mu_{\rm{data}}, \nu)$ is attained on $\nu^\ast$, where $\nu^\ast$ is a uniform distribution over $\{{\mathbf y}_i\}_{i=1}^N$ and $[{\mathbf y}_1, \cdots, {\mathbf y}_N] \in \arg\min_{Y\in {\mathbb R}^{n\times N}, {\rm rank}(Y)\leq k}{\Vert X-Y\Vert _p}$.
\end{theorem}
\begin{proof}
Let us prove first that $\inf_{\mu\in {\mathcal P}_k}W_p (\mu_{\rm{data}}, \mu) \leq \frac{1}{N} \Vert X-Y^\ast\Vert_p $ where 
\begin{equation}
Y^\ast = [{\mathbf y}_1, \cdots, {\mathbf y}_N] \in \arg\min_{Y\in {\mathbb R}^{n\times N}, {\rm rank}(Y)\leq k}{\Vert X-Y\Vert_p }.
\end{equation}
Let $\pi$ be a uniform distribution over $\{({\mathbf x}_i, {\mathbf y}_i)\}_{i=1}^N$ and $\mu^\ast$ be a uniform distribution over $\{ {\mathbf y}_i\}_{i=1}^N$. Since $\pi\in \Pi(\mu_{{\rm data}}, \mu^\ast)$, we obtain $W_p (\mu_{{\rm data}}, \mu^\ast) \leq (\frac{1}{N}\sum_{i=1}^N\Vert {\mathbf x}_i- {\mathbf y}_i\Vert^p)^{1/p}  = \frac{1}{N^p}\Vert X-Y^\ast\Vert_p $. The support of $\mu^\ast$ is $k$-dimensional, because ${\rm rank}(Y^\ast) \leq k$. Thus, we have $\mu^\ast\in {\mathcal P}_k$ and 
$\inf_{\mu\in {\mathcal P}_k}W_p (\mu_{\rm{data}}, \mu) \leq W_p (\mu_{{\rm data}}, \mu^\ast) \leq \frac{1}{N^p} \Vert X-Y^\ast\Vert_p $. Now, if we prove the inverse inequality, i.e. $\inf_{\mu\in {\mathcal P}_k}W_p (\mu_{\rm{data}}, \mu) \geq \frac{1}{N^p} \Vert X-Y^\ast\Vert_p $, this will imply that $\inf_{\mu\in {\mathcal P}_k}W_p (\mu_{\rm{data}}, \mu) = \frac{1}{N^p} \Vert X-Y^\ast\Vert_p $ and therefore, $\inf_{\mu\in {\mathcal P}_k}W_p (\mu_{\rm{data}}, \mu) = W_p(\mu_{{\rm data}}, \mu^\ast)$. This will in the end give us
$\mu^\ast \in \arg \inf_{\mu\in {\mathcal P}_k}W_p (\mu_{\rm{data}}, \mu)$.

Let $\{\mu_t\}_{1}^\infty$ be such that 
$\mu_t\in {\mathcal P}_k$ and $W_p (\mu_{\rm{data}}, \mu_t) - \inf_{\mu\in {\mathcal P}_k}W_p (\mu_{\rm{data}}, \mu) \rightarrow 0$. Let  $L_t$ denote a $k$-dimensional support of $\mu_t$ and $P_t$ is a projection operator onto $L_t$. 

Let $\mu^\ast_{t}$ be a uniform distribution over $\{P_{t}{\mathbf x}_1, \cdots, P_{t}{\mathbf x}_N\}$, i.e. $\mu^\ast_{t}(A) = \frac{1}{N}\sum_{i=1}^N [P_{t}{\mathbf x}_i\in A]$. It is easy to see that $W_p(\mu^\ast_{t}, \mu_{{\rm data}})\leq W_p(\mu_{t}, \mu_{{\rm data}})$, because $\mu^\ast_{t}$ and $\mu_{t}$ share the same $k$-dimensional support $L_{t}$, but the ``transportation of a mass'' concentrated in point ${\mathbf x}_i$ of the empirical distribution $\mu_{{\rm emp}}$ can be most optimally done by just moving it to $P_{t}{\mathbf x}_i$ (i.e. to the closest point on $L_{t}$). Thus, we have $\inf_{\mu\in {\mathcal P}_k}W_p (\mu_{\rm{data}}, \mu) \leq W_p (\mu_{\rm{data}}, \mu^\ast_{t})  \leq W_p (\mu_{\rm{data}}, \mu_{t})$, and therefore, $W_p (\mu_{\rm{data}}, \mu^\ast_{t}) - \inf_{\mu\in {\mathcal P}_k}W_p (\mu_{\rm{data}}, \mu) \rightarrow 0$. 

Since a set of projection operators is compact, one can always extract a subsequence $\{P_{t_s}\}_{s=1}^\infty$, such that $P_{t_s} \rightarrow P$. It is easy to see that $\mu^\ast_{t_s}\rightarrow \mu^{\ast\ast}$ (i.e. $W_p (\mu^\ast_{t_s}, \mu^{\ast\ast})\rightarrow 0$) where $\mu^{\ast\ast}$ is a uniform distribution over $\{P{\mathbf x}_1, \cdots, P{\mathbf x}_N\}$. For that distribution we have 
\begin{equation}
\begin{split}
W_p (\mu_{\rm{data}}, \mu^{\ast\ast}) = \lim_{s\rightarrow \infty} W_p (\mu_{\rm{data}}, \mu^\ast_{t_s}) = \inf_{\mu\in {\mathcal P}_k}W_p (\mu_{\rm{data}}, \mu).
\end{split}
\end{equation}
Thus, the infinum is attained on $\mu^{\ast\ast}$.

It is easy to see that $ W_p(\mu_{{\rm data}}, {\mu^{\ast\ast}}) = \frac{1}{N^p} \Vert X-PX\Vert_p $. Since ${\rm rank}(PX)\leq k$ we obtain $W_p (\mu_{\rm{data}}, \mu^{\ast\ast}) \geq \frac{1}{N^p} \min_{Y\in {\mathbb R}^{n\times N}, {\rm rank}(Y)\leq k}{\Vert X-Y\Vert_p } = \Vert X-Y^\ast\Vert_p$. This completes the proof.
\end{proof}

Note that in the case of $l_1$ norm and $p=1$, i.e. $\Vert {\mathbf x}\Vert =\sum_{i}\vert x_i\vert $, the task~\ref{robust-pca} corresponds to the well-studied {\em robust PCA} problem~\cite{Candes}. 
If, instead of the $l_1$-norm, we use the $l_2$-norm and $p\geq 1$, this leads to another task:
\begin{equation}\label{outlier-pursuit}
\Vert X-L\Vert _{p,2}\rightarrow \min_{{\rm rank} (L)\leq k}
\end{equation}
where $\Vert [{\mathbf s}_1, \cdots, {\mathbf s}_N]\Vert _{p,2} = (\sum_{i=1}^N\|{\mathbf s}_i\|_2^p)^{1/p}$. This task has many applications in mathematics and is known as {\em  the subspace approximation problem}~\cite{AmitVishnoi} .

\section{Proper kernels and proof of Theorem~$\lowercase{\ref{rank-k}}$}

\subsection{Proof of Theorem~$\lowercase{\ref{rank-k}}$}
Let us first show that $\langle f \vert  M \vert  g \rangle$ is defined for  any $f = (T_a\otimes \delta^{n-k})_U\in \mathcal{G}_k$ and $g = (T_b\otimes \delta^{n-k})_V\in \mathcal{G}_k$ where $a,b\in L_1({\mathbb R}^k)$.
 We have 
\begin{equation}
\begin{split}
T_{f_\epsilon} = (T_a\otimes \delta^{n-k})_U\ast G_\epsilon^n =  ((T_a\ast  G_\epsilon^k) \otimes  T_{G_\epsilon^{n-k}} )_U
\end{split}
\end{equation}
Let us denote $a_\epsilon = a\ast G_\epsilon^k$ and $b_\epsilon = b\ast G_\epsilon^k$.
It is easy to see that 
\begin{equation}
f_\epsilon = (a_\epsilon ({\mathbf x}_{1:k})  G_\epsilon^{n-k} ({\mathbf x}_{k+1:n}))_U\in {\mathcal S} ({\mathbb R}^n).
\end{equation} 
From a well-known property of the Weierstrass transform we have
\begin{equation}
\Vert f_\epsilon\Vert _{L_1} = \Vert a_\epsilon\Vert _{L_1} \cdot \Vert G_\epsilon^{n-k}\Vert _{L_1}\leq \Vert a\Vert _{L_1}.
\end{equation}
From this we obtain that 
\begin{equation}
\begin{split}
\vert \langle f_\epsilon \vert  M \vert  g_\epsilon \rangle\vert  \leq \max_{{\mathbf x}, {\mathbf y}} \vert M({\mathbf x}, {\mathbf y})\vert  \,\, \Vert f_\epsilon\Vert _{L_1}\Vert g_\epsilon\Vert _{L_1} \leq  \max_{{\mathbf x}, {\mathbf y}} \vert M({\mathbf x}, {\mathbf y})\vert  \,\, \Vert a\Vert _{L_1}\Vert b\Vert _{L_1}< \infty.
\end{split}
\end{equation}
Thus, $\langle f_\epsilon \vert  M \vert  g_\epsilon \rangle$ is properly defined and
\begin{equation}
\begin{split}
\langle f_\epsilon \vert  M \vert  g_\epsilon \rangle = 
\int\limits_{{\mathbb R}^n\times {\mathbb R}^n} a^\ast_\epsilon ({\mathbf x}_{1:k})   G_\epsilon^{n-k} ({\mathbf x}_{k+1:n}) M(U^T{\mathbf x}, V^T {\mathbf y})  b_\epsilon ({\mathbf y}_{1:k})   G_\epsilon^{n-k} ({\mathbf y}_{k+1:n})  d{\mathbf x}d{\mathbf y} = \\
\int\limits_{{\mathbb R}^k\times {\mathbb R}^k} a^\ast_\epsilon({\mathbf x}_{1:k})  M_\epsilon({\mathbf x}_{1:k}, {\mathbf y}_{1:k}) b_\epsilon({\mathbf y}_{1:k}) d{\mathbf x}_{1:k}d{\mathbf y}_{1:k}
\end{split}
\end{equation}
where 
\begin{equation}
\begin{split}
M_\epsilon({\mathbf x}_{1:k},  {\mathbf y}_{1:k}) =  \int\limits_{{\mathbb R}^{n-k}\times {\mathbb R}^{n-k}} G_\epsilon^{n-k} ({\mathbf x}_{k+1:n})  M(U^T{\mathbf x}, V^T {\mathbf y})  G_\epsilon^{n-k} ({\mathbf y}_{k+1:n}) d{\mathbf x}_{k+1:n}d{\mathbf y}_{k+1:n}.
\end{split}
\end{equation}
Let $U_k, V_k\in {\mathbb R}^{n\times n}$ be matrices that comprise the first $k$ rows of $U, V$ correspondingly and $n-k$ zero rows below. Also, let $L$ denote the Lipschitz constant for $M$ such that $\vert M({\mathbf x}, {\mathbf y})-M({\mathbf x}', {\mathbf y}')\vert \leq L (\vert {\mathbf x}-{\mathbf x}'\vert +\vert {\mathbf y}-{\mathbf y}'\vert )$.
For the function $M_\epsilon({\mathbf x}_{1:k},  {\mathbf y}_{1:k})$ we have:
\begin{equation}
\begin{split}
\vert M_\epsilon({\mathbf x}_{1:k},  {\mathbf y}_{1:k}) - M(U_k^T{\mathbf x}, V_k^T {\mathbf y})\vert  = 
\vert \int\limits_{{\mathbb R}^{2n-2k}} G_\epsilon^{n-k} ({\mathbf x}_{k+1:n}) \big( M(U^T{\mathbf x}, V^T {\mathbf y}) - \\ - M(U_k^T{\mathbf x}, V_k^T {\mathbf y})\big) G_\epsilon^{n-k} ({\mathbf y}_{k+1:n})  d{\mathbf x}_{k+1:n}d{\mathbf y}_{k+1:n} \vert  
\leq \\ 
L \vert \int\limits_{{\mathbb R}^{2n-2k}} G_\epsilon^{n-k} ({\mathbf x}_{k+1:n}) \big( \vert (U-U_k)^T{\mathbf x}\vert  + \vert (V-V_k)^T {\mathbf y}\vert \big) 
\cdot G_\epsilon^{n-k} ({\mathbf y}_{k+1:n})d{\mathbf x}_{k+1:n}d{\mathbf y}_{k+1:n} \vert 
= \\
L \vert \int\limits_{{\mathbb R}^{2n-2k}} G_\epsilon^{n-k} ({\mathbf x}_{k+1:n}) \left( \vert {\mathbf x}_{k+1:n}\vert  + \vert {\mathbf y}_{k+1:n}\vert \right) \cdot G_\epsilon^{n-k} ({\mathbf y}_{k+1:n})d{\mathbf x}_{k+1:n}d{\mathbf y}_{k+1:n} \vert  
= \\
2 L \int\limits_{{\mathbb R}^{n-k}} \vert {\mathbf x}_{k+1:n}\vert  G_\epsilon^{n-k} ({\mathbf x}_{k+1:n}) d{\mathbf x}_{k+1:n} =  
2 L\epsilon^{n-k}  \int\limits_{{\mathbb R}^{n-k}} \vert {\mathbf x}_{k+1:n}\vert  G_1^{n-k} ({\mathbf x}_{k+1:n}) d{\mathbf x}_{k+1:n}.
\end{split}
\end{equation}

Thus, there exists bounded $\tilde {M}({\mathbf x}_{1:k}, {\mathbf y}_{1:k}) =  M(U_k^T{\mathbf x}, V_k^T {\mathbf y})$ such that
\begin{equation}
M_\epsilon({\mathbf x}_{1:k},  {\mathbf y}_{1:k}) \mathop\rightarrow\limits^{\epsilon\rightarrow 0}  \tilde {M}({\mathbf x}_{1:k}, {\mathbf y}_{1:k}) {\rm \,\,\,in\,\,}L_\infty ({\mathbb R}^{2k}).
\end{equation}
Further we assume that $\epsilon>0$ is small enough, so that $M_\epsilon({\mathbf x}_{1:k},  {\mathbf y}_{1:k}) \leq C = 2\max \vert  M({\mathbf x},{\mathbf y})\vert $.
Now we have:
\begin{equation}
\begin{split}
\vert \langle f_\epsilon \vert  M \vert  g_\epsilon \rangle - 
\int\limits_{{\mathbb R}^k\times {\mathbb R}^k} a^\ast({\mathbf x}_{1:k})  \tilde {M}({\mathbf x}_{1:k}, {\mathbf y}_{1:k}) b({\mathbf y}_{1:k})   d{\mathbf x}_{1:k}d{\mathbf y}_{1:k}\vert  = \\
\vert \int\limits_{{\mathbb R}^k\times {\mathbb R}^k} \big(a^\ast_\epsilon({\mathbf x}_{1:k})  M_\epsilon({\mathbf x}_{1:k}, {\mathbf y}_{1:k}) b_\epsilon({\mathbf y}_{1:k}) - a^\ast({\mathbf x}_{1:k})  \tilde {M}({\mathbf x}_{1:k}, {\mathbf y}_{1:k}) b({\mathbf y}_{1:k})\big)  d{\mathbf x}_{1:k}d{\mathbf y}_{1:k}\vert  =  \\
\vert \int\limits_{{\mathbb R}^k\times {\mathbb R}^k} M_\epsilon({\mathbf x}_{1:k}, {\mathbf y}_{1:k}) a^\ast_\epsilon({\mathbf x}_{1:k})\big( b_\epsilon({\mathbf y}_{1:k})  - b({\mathbf y}_{1:k})\big)  d{\mathbf x}_{1:k}d{\mathbf y}_{1:k} + \\
\int\limits_{{\mathbb R}^k\times {\mathbb R}^k} M_\epsilon({\mathbf x}_{1:k}, {\mathbf y}_{1:k})b({\mathbf y}_{1:k}) \big(a^\ast_\epsilon({\mathbf x}_{1:k})   -  a^\ast({\mathbf x}_{1:k})  \big)  d{\mathbf x}_{1:k}d{\mathbf y}_{1:k} + \\
\int\limits_{{\mathbb R}^k\times {\mathbb R}^k}  a^\ast({\mathbf x}_{1:k})   b({\mathbf y}_{1:k})\big(M_\epsilon({\mathbf x}_{1:k}, {\mathbf y}_{1:k}) - \tilde {M}({\mathbf x}_{1:k}, {\mathbf y}_{1:k})\big)  d{\mathbf x}_{1:k}d{\mathbf y}_{1:k} \vert  \leq \\
C\Vert a_\epsilon\Vert _{L_1} \Vert  b_\epsilon - b\Vert _{L_1} + C\Vert b\Vert _{L_1} \Vert  a_\epsilon - a\Vert _{L_1} + \Vert a^\ast({\mathbf x}_{1:k})   b({\mathbf y}_{1:k})\Vert _{L_1} \Vert  M_\epsilon -  \tilde {M}\Vert _{L_\infty}.
\end{split}
\end{equation}
It is well-known (e.g. see Theorem 2.25 from~\cite{Advanced}) that 
$
\Vert a_\epsilon-  a\Vert _{L_p}$, $\Vert b_\epsilon  -  b\Vert _{L_p} \rightarrow 0$, $\Vert a_\epsilon\Vert _{L_1} \leq \Vert a\Vert _{L_1}$ and $\Vert  M_\epsilon -  \tilde {M}\Vert _{L_\infty}\rightarrow 0$. Thus, $\lim_{\epsilon\rightarrow 0}\langle f_\epsilon \vert  M \vert  g_\epsilon \rangle$ exists and $\langle f \vert  M \vert  g \rangle$ is defined.

Let us now prove that $\rank M_f \leq k$.
The function $f\in \mathcal{G}'_k$ is such that $f = (T_{g} \otimes \delta^{n-k})_U$ where $\{x_ig\}_{i=1}^k\subseteq L_1({\mathbb R}^k)$ and $U = \begin{bmatrix}
{\mathbf w}_1, \cdots, {\mathbf w}_n
\end{bmatrix}$ is an orthogonal matrix. By construction,
\begin{equation}
\begin{split}
\langle x_i f\vert  M \vert  x_j f\rangle = \langle (x_i f)_{U^T}\vert  M(U^T{\mathbf x}, U^T{\mathbf y}) \vert  (x_j f)_{U^T}\rangle = 
\langle {\mathbf w}^T_i {\mathbf x}\, T_{g} \otimes \delta^{n-k}\vert  M(U^T{\mathbf x}, U^T{\mathbf y}) \vert  {\mathbf w}^T_j {\mathbf x}\, T_{g} \otimes \delta^{n-k}\rangle.
\end{split}
\end{equation}
Let us now denote $V = \begin{bmatrix}
{\mathbf u}_1, \cdots, {\mathbf u}_n
\end{bmatrix}\in {\mathbb R}^{k\times n}$ a submatrix of $U$ in which only first $k$ rows of $U$ are present.
Then, the latter integral is equal to
\begin{equation}
\begin{split}
\iint\limits_{{\mathbb R}^{k}\times {\mathbb R}^{k}} {\mathbf u}^T_i {\mathbf x}_{1:k} {\mathbf y}_{1:k}^T {\mathbf u}_j  g({\mathbf x}_{1:k})^\ast M(V^T{\mathbf x}_{1:k}, V^T{\mathbf y}_{1:k}) g({\mathbf y}_{1:k})  d {\mathbf x}_{1:k} d {\mathbf y}_{1:k} = {\mathbf u}^T_i B {\mathbf u}_j
\end{split}
\end{equation}
where 
\begin{equation}
\begin{split}
B = \begin{bmatrix} \langle x_i g \vert  M' \vert  x_j g \rangle
\end{bmatrix}_{1\leq i, j \leq k}, M'({\mathbf x}_{1:k}, {\mathbf y}_{1:k}) = M(V^T{\mathbf x}_{1:k}, V^T{\mathbf y}_{1:k})
\end{split}
\end{equation}
 is the Gram matrix of the collection $\{x_ig({\mathbf x}_{1:k})\}_{i=1}^k\subseteq L_1 ({\mathbb R}^k)$. 

Obviously, $\rank M_f = \rank  \begin{bmatrix}  {\rm Re\,}
{\mathbf u}^T_i B {\mathbf u}_j
\end{bmatrix}_{1\leq i, j \leq n} = \rank  V^T  ({\rm Re\,} B) V \leq \rank V = k$.

\section{Proofs of Theorem~$\lowercase{\ref{metrization}}$ and~$\lowercase{\ref{existence}}$}
For any $f = (T_l\otimes \delta^{n-k})_U\in \mathcal{G}_k$ and $\sigma>0$, let us define $f_\sigma$ as
\begin{equation}
\begin{split}
T_{f_\sigma} = (T_l\otimes \delta^{n-k})_U \ast G^{n}_{\sigma} = (T_{l_\sigma}\otimes T_{G^{n-k}_{\sigma}})_U\\
f_\sigma = (l_\sigma ({\mathbf x}_{1:k})  G^{n-k}_{\sigma}({\mathbf x}_{k+1:n}) )_U\\
l_\sigma = l\ast G^{k}_{\sigma}.
\end{split}
\end{equation}
We have $T_{f_\sigma} \rightarrow^\ast (T_{l}\otimes \delta^{n-k})_U$ as $\sigma\rightarrow +0$.
\begin{lemma}\label{entries} For any $f\in \mathcal{G}_k$, $\lim_{\sigma\rightarrow +0} \langle x_i f_\sigma \vert  M \vert  x_j f_\sigma\rangle = 0$, for any $(i,j)\notin \{1,...,k\}^2$, and $\sup_{\sigma\in [0,1]} \langle x_i f_\sigma \vert  M \vert  x_j f_\sigma\rangle < \infty$, for any $(i,j)\in \{1,...,k\}^2$.
\end{lemma}
\begin{proof}
W.l.o.g. we can assume that $f = T_l\otimes \delta^{n-k}, l\in {\mathcal S} ({\mathbb R}^k)$.
If $i>k, j\leq k$ we have
\begin{equation}
\begin{split}
\langle x_i f_\sigma \vert  M \vert  x_j f_\sigma\rangle = 
\frac{1}{(2\pi\sigma^2)^{n-k}}\iint_{{\mathbb R}^n\times {\mathbb R}^n} x_i y_j e^{-\frac{\vert {\mathbf x}_{k+1:n}\vert ^2}{2\sigma^2}} l_\sigma({\mathbf x}_{1:k})  M({\mathbf x}, {\mathbf y}) e^{-\frac{\vert {\mathbf y}_{k+1:n}\vert ^2}{2\sigma^2}} l_\sigma({\mathbf y}_{1:k})  d {\mathbf x} d {\mathbf y} = \\ 
\int_{{\mathbb R}^n} \frac{1}{\sqrt{2\pi\sigma^2}^{n-k}} x_i e^{-\frac{\vert {\mathbf x}_{k+1:n}\vert ^2}{2\sigma^2}} l_\sigma({\mathbf x}_{1:k}) P({\mathbf x}) d {\mathbf x} 
\end{split}
\end{equation}
where $P({\mathbf x}) = \int_{{\mathbb R}^n} \frac{1}{\sqrt{2\pi\sigma^2}^{n-k}} y_j M({\mathbf x}, {\mathbf y}) e^{-\frac{\vert {\mathbf y}_{k+1:n}\vert ^2}{2\sigma^2}} l_\sigma({\mathbf y}_{1:k})  d {\mathbf y}$. Using the Hölder inequality we obtain
\begin{equation}
\begin{split}
\vert \langle x_i f_\sigma \vert  M \vert  x_j f_\sigma\rangle\vert  \leq \Vert \frac{1}{\sqrt{2\pi\sigma^2}^{n-k}} x_i e^{-\frac{\vert {\mathbf x}_{k+1:n}\vert ^2}{2\sigma^2}} l_\sigma({\mathbf x}_{1:k})\Vert _{L_1 ({\mathbb R}^n)} \Vert P\Vert _{L_{\infty} ({\mathbb R}^n)} = \\ \Vert \frac{1}{\sqrt{2\pi\sigma^2}^{n-k}} x_i e^{-\frac{\vert {\mathbf x}_{k+1:n}\vert ^2}{2\sigma^2}}\Vert _{L_1 ({\mathbb R}^{n-k})} \Vert l_\sigma\Vert _{L_1 ({\mathbb R}^{k})} \Vert P\Vert _{L_{\infty} ({\mathbb R}^n)}
\end{split}
\end{equation}
Since $\vert M({\mathbf x}, {\mathbf y})\vert  \leq \gamma$ for some $\gamma$, we have
\begin{equation}
\begin{split}
\vert P({\mathbf x})\vert  \leq 
\gamma \Vert \frac{1}{\sqrt{2\pi\sigma^2}^{n-k}} y_je^{-\frac{\vert {\mathbf y}_{k+1:n}\vert ^2}{2\sigma^2}}l_\sigma({\mathbf y}_{1:k})\Vert _{L_1 ({\mathbb R}^n)} = 
\\
\gamma \Vert \frac{1}{\sqrt{2\pi\sigma^2}^{n-k}} e^{-\frac{\vert {\mathbf y}_{k+1:n}\vert ^2}{2\sigma^2}}\Vert _{L_1 ({\mathbb R}^{n-k})} \Vert y_j l_\sigma({\mathbf y}_{1:k})\Vert _{L_1 ({\mathbb R}^k)} =  \gamma \Vert y_j l_\sigma({\mathbf y}_{1:k})\Vert _{L_1 ({\mathbb R}^k)}.
\end{split}
\end{equation}
Thus,
\begin{equation}
\begin{split}
\vert \langle x_i f_\sigma \vert  M \vert  x_j f_\sigma\rangle\vert  \leq \Vert \frac{1}{\sqrt{2\pi\sigma^2}^{n-k}} x_i e^{-\frac{\vert {\mathbf x}_{k+1:n}\vert ^2}{2\sigma^2}}\Vert _{L_1 ({\mathbb R}^{n-k})} \Vert l_\sigma\Vert _{L_1 ({\mathbb R}^k)} \gamma \Vert y_jl_\sigma\Vert _{L_1 ({\mathbb R}^k)}. 
\end{split}
\end{equation}
Using $\Vert l_\sigma\Vert _{L_1 ({\mathbb R}^k)} -\Vert l\Vert _{L_1 ({\mathbb R}^k)} \mathop\rightarrow\limits^{\sigma\rightarrow +0} 0$, $\Vert y_jl_\sigma\Vert _{L_1 ({\mathbb R}^k)} -\Vert y_j l\Vert _{L_1 ({\mathbb R}^k)} \mathop\rightarrow\limits^{\sigma\rightarrow +0} 0$, we see the boundedness of $\Vert l_\sigma\Vert _{L_1 ({\mathbb R}^k)} \gamma \Vert y_jl_\sigma\Vert _{L_1 ({\mathbb R}^k)} $ and proceed
\begin{equation} 
\leq C \Vert \frac{1}{\sqrt{2\pi\sigma^2}^{n-k}} x_i e^{-\frac{\vert {\mathbf x}_{k+1:n}\vert ^2}{2\sigma^2}}\Vert _{L_1 ({\mathbb R}^{n-k})} .
\end{equation}
It is easy to see that $\Vert \frac{1}{\sqrt{2\pi\sigma^2}^{n-k}} x_i e^{-\frac{\vert {\mathbf x}_{k+1:n}\vert ^2}{2\sigma^2}}\Vert _{L_1 ({\mathbb R}^{n-k})} \rightarrow 0$ as $\sigma\rightarrow 0$, therefore $\langle x_i f_\sigma \vert  M \vert  x_j f_\sigma\rangle \rightarrow 0$.

Similarly, we can prove that $\langle x_i f_\sigma \vert  M \vert  x_j f_\sigma\rangle \rightarrow 0$ if $i, j> k$. 

The entries of the main $k\times k$ minor $[\langle x_i f_\sigma \vert  M \vert  x_j f_\sigma\rangle]_{1\leq i,j\leq k}$ are bounded, due to
\begin{equation}
\begin{split}
{\rm Tr\,} M_{f_\sigma} = \frac{1}{(2\pi\sigma^2)^{n-k}}\iint_{{\mathbb R}^n\times {\mathbb R}^n} {\mathbf x} \cdot {\mathbf y} e^{-\frac{\vert {\mathbf x}_{k+1:n}\vert ^2}{2\sigma^2}} l_\sigma({\mathbf x}_{1:k}) M({\mathbf x}, {\mathbf y}) e^{-\frac{\vert {\mathbf y}_{k+1:n}\vert ^2}{2\sigma^2}} l_\sigma({\mathbf y}_{1:k})  d {\mathbf x} d {\mathbf y} \leq \\
\frac{\gamma }{(2\pi\sigma^2)^{n-k}}\iint_{{\mathbb R}^n\times {\mathbb R}^n} (\vert {\mathbf x}_{1:k} \cdot {\mathbf y}_{1:k}\vert  +\vert {\mathbf x}_{k+1:n} \cdot {\mathbf y}_{k+1:n}\vert ) e^{-\frac{\vert {\mathbf x}_{k+1:n}\vert ^2+\vert {\mathbf y}_{k+1:n}\vert ^2}{2\sigma^2}} l_\sigma({\mathbf x}_{1:k})   l_\sigma({\mathbf y}_{1:k})  d {\mathbf x} d {\mathbf y} \leq  \\
\gamma \iint_{{\mathbb R}^n\times {\mathbb R}^n} \vert {\mathbf x}_{1:k} \cdot {\mathbf y}_{1:k}\vert l_\sigma({\mathbf x}_{1:k})   l_\sigma({\mathbf y}_{1:k})  d {\mathbf x}_{1:k} d {\mathbf y}_{1:k} +  \gamma\Vert l_\sigma\Vert ^2_{L_1} (n-k)\sigma^2 \leq \\  
\gamma\sum_{j=1}^k\Vert y_jl_\sigma\Vert ^2_{L_1 ({\mathbb R}^k)}  + \gamma\Vert l_\sigma\Vert ^2_{L_1} (n-k)\sigma^2.
\end{split}
\end{equation}
Again, using $\Vert l_\sigma\Vert _{L_1 ({\mathbb R}^k)} -\Vert l\Vert _{L_1 ({\mathbb R}^k)} \mathop\rightarrow\limits^{\sigma\rightarrow +0} 0$, $\Vert y_jl_\sigma\Vert _{L_1 ({\mathbb R}^k)} -\Vert y_j l\Vert _{L_1 ({\mathbb R}^k)} \mathop\rightarrow\limits^{\sigma\rightarrow +0} 0$, we obtain the boundedness of RHS.
\end{proof}
\begin{corollary}\label{cor} For any $f\in \mathcal{G}_k$, $\lim_{\sigma\rightarrow 0} R(f_\sigma) = 0$.
\end{corollary}
\begin{proof}
W.l.o.g. we can assume that $f = T_l\otimes \delta^{n-k}, l\in {\mathcal S} ({\mathbb R}^k)$.
By lemma, all entries of $M_{f_\sigma}$ except those of the main $k\times k$ minor approach $0$ as $\sigma\rightarrow 0$. This means that 
$\lim_{\sigma\rightarrow +0}Q(f_\sigma) = 0 $, 
where $Q(f_\sigma) = \sum_{i=k+1}^n\langle x_i f_\sigma \vert  M \vert  x_i f_\sigma\rangle$. Let ${\mathbf v}_1, \cdots, {\mathbf v}_n$ be unit eigenvectors of $M_{f_\sigma}$ corresponding to the eigenvalues $\lambda_1\geq \cdots \geq \lambda_n$, $P = \sum_{i=k+1}^n {\mathbf e}_i{\mathbf e}^T_i$, then
\begin{equation}
\begin{split}
R(f_\sigma) = \sum_{i=k+1}^n \lambda_i = \min_{p_i\in [0,1], \sum_{1}^n p_i = n-k} \sum_{i=1}^n \lambda_i p_i \leq  
\sum_{i=1}^n \lambda_i {\rm Tr\,}(P{\mathbf v}_i {\mathbf v}_i^T P) = {\rm Tr\,}(P M_{f_\sigma} P) = Q(f_\sigma)
\end{split}
\end{equation}
Since $R(f_\sigma) \leq Q(f_\sigma)$, we obtain $\lim_{\sigma\rightarrow 0} R(f_\sigma) = 0$.
\end{proof}
\subsubsection{Proof of Theorem~\ref{metrization}}%
\begin{proof}
Suppose that a sequence $\{f_i\}^\infty_{s=1}\subseteq {\mathcal S} ({\mathbb R}^n)$ regularly solves (7) and $T\in \mathop{\rm Lim}\limits_{i\rightarrow \infty } f_i$. W.l.o.g. we can assume  that $T_{f_i}\rightarrow^\ast T$ and $\Tr(M_{f_i})$ is bounded and $I(f_i) + \lambda_i R(f_i) \leq \inf\limits_{f\in {\mathcal S} ({\mathbb R}^n)} I(f) + \lambda_i R(f) +\epsilon_i, \epsilon_i\rightarrow 0$.
Below we use continuity of $I$ and corollary~\ref{cor}:
\begin{equation}
\begin{split}
\inf\limits_{f\in {\mathcal S} ({\mathbb R}^n)} I(f) + \lambda_i R(f) \leq \inf\limits_{f\in \mathcal{G}_k}\inf_{\sigma>0} I(f_\sigma) + \lambda_i R(f_\sigma) \leq 
\inf\limits_{f\in \mathcal{G}_k}\lim_{\sigma\rightarrow +0} I(f_\sigma) + \lambda_i R(f_\sigma)\leq \inf\limits_{f\in \mathcal{G}_k} I(f)
\end{split}
\end{equation} 
from which we conclude that $\lambda_i R(f_i) \leq \inf\limits_{f\in \mathcal{G}_k} I(f) + \epsilon_i$ and, therefore, $R(f_i)\mathop\rightarrow\limits^{i\rightarrow \infty} 0$.

For each $l$, let us define $P_l$ as the projection operator to a subspace spanned by first principal components of the matrix $\sqrt{M_{f_l}}$, i.e.
\begin{equation}
P_l = \sum_{i=1}^k {\mathbf v}^l_i {\mathbf v}^l_i\rm{ }^T,
\end{equation}
where ${\mathbf v}^l_1, ..., {\mathbf v}^l_k$ are orthonormal eigenvectors that correspond to $k$ largest eigenvalues of $\sqrt{M_{f_l}}$. From the Eckart-Young-Mirsky theorem we see that $R(f_l) = \Vert \sqrt{M_{f_l}}-P_l \sqrt{M_{f_l}}\Vert ^2_F$. Since a set of all projection operators $\{P\in {\mathbb R}^{n\times n}\vert  P^2 = P, P^T = P\}$ is a compact subset of ${\mathbb R}^{n^2}$, one can always find a projection operator $P = \sum_{i=1}^k {\mathbf v}_i {\mathbf v}^T_i$ and a growing subsequence $\{l_s\}$ such that $\Vert P_{l_s}-P\Vert _F\rightarrow 0$ as $s\rightarrow\infty$. Thus, for the subsequence $\{f_{l_s}\}$ we have
\begin{equation}
\begin{split}
\Vert \sqrt{M_{f_{l_s}}} -  P\sqrt{M_{f_{l_s}}}\Vert _F =  \Vert \sqrt{M_{f_{l_s}}} - P_{l_s}\sqrt{M_{f_{l_s}}}+P_{l_s}\sqrt{M_{f_{l_s}}}- P\sqrt{M_{f_{l_s}}}\Vert _F\leq \\
\Vert \sqrt{M_{f_{l_s}}} - P_{l_s}\sqrt{M_{f_{l_s}}}\Vert _F + \Vert P_{l_s}- P\Vert _F \Vert \sqrt{M_{f_{l_s}}}\Vert _F = \sqrt{R(f_{l_s})} + \Vert P_{l_s}- P\Vert _F \sqrt{\Tr(M_{f_s})}\end{split}
\end{equation}
and using the boundedness of $\Tr(M_{f_s})$ we obtain $\Vert \sqrt{M_{f_{l_s}}} - P\sqrt{M_{f_{l_s}}}\Vert _F\rightarrow 0$.

Since $\Vert \sqrt{M_{f_{l_s}}} - P\sqrt{M_{f_{l_s}}}\Vert _F\rightarrow 0$, let us complete ${\mathbf v}_1, ..., {\mathbf v}_k$ to an orthonormal basis ${\mathbf v}_1, ..., {\mathbf v}_n$ and make the change of variables $y_i = {\mathbf v}^T_i {\mathbf x}$. Let us denote $V = \begin{bmatrix}
{\mathbf v}_1, ..., {\mathbf v}_n
\end{bmatrix}$ and let $V^T = \begin{bmatrix}
{\mathbf w}_1, ..., {\mathbf w}_n
\end{bmatrix}$. Then, after that change of variables any function $f({\mathbf x})$ corresponds to $f'({\mathbf y}) = f(V{\mathbf y})$ and the kernel $M$ corresponds to $M'({\mathbf y}, {\mathbf y}') = M(V{\mathbf y}, V{\mathbf y}')$. After we apply that change of variables in the integral expression of $\langle x_i f\vert  M \vert  x_j f\rangle$, we obtain
\begin{equation}
\begin{split}
\langle x_i f \vert  M \vert  x_j f\rangle = \langle {\mathbf w}^T_i {\mathbf y} f'\vert  M' \vert  {\mathbf w}^T_j {\mathbf y} f'\rangle =  {\mathbf w}^T_i
\begin{bmatrix}
\langle y_{i'} f' \vert  M' \vert  y_{j'} f'\rangle
\end{bmatrix}_{n\times n} {\mathbf w}_j\Rightarrow \\
{\rm Re\,}\langle x_i f \vert  M \vert  x_j f\rangle =  {\mathbf w}^T_i
\begin{bmatrix}
{\rm Re\,} \langle y_{i'} f' \vert  M' \vert  y_{j'} f'\rangle
\end{bmatrix}_{n\times n} {\mathbf w}_j.
\end{split}
\end{equation}
I.e. $M_f = V M'_{f'} V^T$, or $M'_{f'} = V^T M_f V$. Note that $P = V I_n^k V^T$ where $I^k_{n}$ is a diagonal matrix whose main $k\times k$ minor is the identity matrix, and all other entries are zeros.
Using the fact that the Frobenius norm of orthogonally similar matrices are equal and the identity $V^T \sqrt{M_{f_{l_s}}} V = \sqrt{V^T M_{f_{l_s}} V}$, we obtain
\begin{equation}
\begin{split}
\Vert \sqrt{M_{f_{l_s}}} - P\sqrt{M_{f_{l_s}}}\Vert _F =  \Vert V^T \sqrt{M_{f_{l_s}}} V - V^T P\sqrt{M_{f_{l_s}}}V\Vert _F = \\
\Vert \sqrt{V^T M_{f_{l_s}} V} - V^T V I_n^k V^T \sqrt{M_{f_{l_s}}}V\Vert _F =  \Vert \sqrt{M'_{f'_{l_s}}} - I^k_{n}\sqrt{M'_{f'_{l_s}}}\Vert _F.
\end{split}
\end{equation}
 
Thus, the property $\Vert \sqrt{M_{f_{l_s}}} - P\sqrt{M_{f_{l_s}}}\Vert _F\rightarrow 0$ implies 
\begin{equation}
{\rm Re\,}\langle y_i f'_{l_s} \vert  M' \vert  y_j f'_{l_s}\rangle \rightarrow 0,{\rm \,\, if\,\, }i>k.
\end{equation}
Moreover, for $i=j$ we have ${\rm Re\,}\langle y_i f'_{l_s} \vert  M' \vert  y_j f'_{l_s}\rangle = \langle y_i f'_{l_s} \vert  M' \vert  y_j f'_{l_s}\rangle$.
It is easy to see that after the change of variables we still have $f'_{l_s}\rightarrow^\ast T_{V}$. 
Since $f'_{l_s}\in {\mathcal S} ({\mathbb R}^n)$, we have $y_i f'_{l_s}\in {\mathcal S} ({\mathbb R}^n)$ and, therefore, $y_i f'_{l_s}\in L_2 ({\mathbb R}^n)$. Let us treat now $M'$ as an operator ${\rm O}_{M'}: L_2 ({\mathbb R}^n)\rightarrow L_2 ({\mathbb R}^n), {\rm O}_{M'}[f]({\mathbf x}) = \int_{{\mathbb R}^n}M'({\mathbf x}, {\mathbf y})f({\mathbf y}) d {\mathbf y}$. 
Let us take any function $\phi\in L_2 ({\mathbb R}^n)$ such that $\psi = {\rm O}_{M'}[\phi]\in {\mathcal S} ({\mathbb R}^n)$.
Since ${\rm O}_{M'}$ is a strictly positive self-adjoint operator, by the Cauchy-Schwarz inequality, we obtain
\begin{equation}
\vert \langle y_i f'_{l_s}, {\rm O}_{M'}[\phi]\rangle\vert  \leq \sqrt{\langle y_i f'_{l_s} \vert  M' \vert  y_i f'_{l_s}\rangle} \sqrt{\langle \phi, {\rm O}_{M'}[\phi]\rangle} .
\end{equation}
Therefore, for any $\psi\in {\rm Range\,} [{\rm O}_{M'}]\cap {\mathcal S} ({\mathbb R}^n)$ and $i>k$ we have $\lim_{s\rightarrow \infty} \langle y_i f'_{l_s}, \psi\rangle = \lim_{s\rightarrow \infty} \langle  f'_{l_s}, y_i\psi\rangle = 0$. Since $f'_{l_s}\rightarrow^\ast T_V$ we obtain $\langle   T_V, y_i\psi\rangle = \langle   y_i T_V, \psi\rangle = 0$ for any $\psi\in {\rm Range\,} [{\rm O}_{M'}]\cap {\mathcal S} ({\mathbb R}^n)$. But the denseness of ${\rm Range\,} [{\rm O}_{M'}]\cap {\mathcal S} ({\mathbb R}^n)$ in ${\mathcal S} ({\mathbb R}^n)$ implies that $y_i T_V = 0$.

Using lemma~\ref{lemma} and $(T_V)_{V^T} = T$ we obtain $T\in \mathcal{G}'_k$.
Thus, we proved that $T_{f_i}\rightarrow T\in \mathcal{G}'_k$.

Since $I(f_i)\leq I(f_i) + \lambda_i R(f_i) \leq \inf\limits_{f\in \mathcal{G}'_k} I(f) +\epsilon_i$ and $I$ is continuous, we finally get that $I(T)\leq \inf\limits_{f\in \mathcal{G}'_k} I(f) $, i.e. $T\in {\rm Arg}\min\limits_{f\in \mathcal{G}'_k} I(f)$. 
\end{proof}
\subsubsection{Proof of Theorem~$\lowercase{\ref{existence}}$}%
\begin{proof}
Suppose $f^\ast\in {\rm Arg} \min \limits_{f\in \mathcal{G}'_k} I(f)\bigcap \mathcal{B}_k$,
i.e. $f^\ast \in \mathcal{B}_k$ and $I(f^\ast) = \min \limits_{f\in \mathcal{G}'_k} I(f)$. 
Since $f^\ast \in \mathcal{B}_k$, then there exists a sequence $\{s^i\}\subseteq  \mathcal{G}_k$ such that $T_{s^i}\rightarrow^\ast f^\ast$ and $\sup\limits_{i}{\rm Tr\,}M_{s^i} < \infty$. 

Let us define $s^i_\sigma\in {\mathcal S} ({\mathbb R}^n)$ as
$
T_{s^i_\sigma} = T_{s^i}\ast G_\sigma^n$.
Since $\lim_{\sigma\rightarrow 0} R(s^i_\sigma) =0 $ (lemma~\ref{entries}), there exists $\sigma_i>0$, such that $R(s^i_\sigma) < \frac{1}{i}$ whenever $0<\sigma \leq \sigma_i$. Also, by definition ${\rm Tr\,}M_{s^i} = \lim_{\sigma\rightarrow 0}{\rm Tr\,} M_{s^i_\sigma}$. Therefore, there exists $\sigma'_i>0$, such that ${\rm Tr\,} M_{s^i_\sigma} < {\rm Tr\,}M_{s^i}+1$  whenever $0<\sigma \leq \sigma'_i$.

If we set $\sigma_i^\ast = \min\{\sigma_i, \sigma'_i, \frac{1}{i}\}$, then a sequence $\{s^i_{\sigma^\ast_i}\}\subseteq {\mathcal S} ({\mathbb R}^n)$ satisfies
\begin{equation}
\begin{split}
\lim_{i\rightarrow \infty}R(s^i_{\sigma^\ast_i}) = 0, 
\sup\limits_{i}{\rm Tr\,} M_{s^i_{\sigma^\ast_i}} < \infty
\end{split}
\end{equation}
and (using lemma~\ref{steingauz})
$
T_{s^i_{\sigma^\ast_i}} \rightarrow^\ast f^\ast$.

Due to the continuity of $I$ we have $
\lim_{i\rightarrow \infty} I(s^i_{\sigma^\ast_i}) = I(f^\ast)$. 
Now we set $f_i = s^i_{\sigma^\ast_i}$, $\lambda_i = \frac{1}{\sqrt{R(f_i)}}$ and we obtain the needed sequence:
\begin{equation}
\lim_{i\rightarrow \infty} I(f_i) = \lim_{i\rightarrow \infty} I(f_i) + \lambda_i R(f_i)  = I(f^\ast), \lim_{i\rightarrow \infty}\lambda_i= +\infty,
\end{equation}
where ${\rm Tr\,} M_{f_i}$ is bounded. It remains to check that our sequence regularly solves (7), i.e. $\lim_{i\rightarrow \infty} \inf\limits_{f\in {\mathcal S} ({\mathbb R}^n)} I(f) + \lambda_i R(f) =  I(f^\ast)$ (this will imply $\lim_{i\rightarrow \infty} I(f_i) + \lambda_i R(f_i) -\inf\limits_{f\in {\mathcal S} ({\mathbb R}^n)} I(f) + \lambda_i R(f)  = 0$). The inequality in one direction is obvious, 
\begin{equation}
\begin{split}
\inf\limits_{f\in {\mathcal S} ({\mathbb R}^n)} I(f) + \lambda_i R(f) \leq \inf\limits_{f\in \mathcal{G}_k}\inf_{\sigma>0} I(f_\sigma) + \lambda_i R(f_\sigma) \leq \\ \inf\limits_{f\in \mathcal{G}_k}\lim_{\sigma\rightarrow +0} I(f_\sigma) + \lambda_i R(f_\sigma) =
\inf\limits_{f\in \mathcal{G}_k}I(f)= I(f^\ast).
\end{split}
\end{equation} 
Let us prove the inverse inequality.

Since ${\rm rsol\,} (I(f), R(f))\ne \emptyset$, there exists $\{\tilde{f}_i\}\subseteq {\mathcal S} ({\mathbb R}^n)$ such that
\begin{equation}
\begin{split}
I(\tilde{f}_i)+\tilde{\lambda}_i R(\tilde{f}_i)\leq \inf\limits_{f\in {\mathcal S} ({\mathbb R}^n)}I(f)+\tilde{\lambda}_i R(f) +\epsilon_i, 
\lim_{s\rightarrow +\infty} \tilde{\lambda}_i = +\infty, \lim_{i\rightarrow +\infty}\epsilon_i = 0, {\rm Tr\,}M_{\tilde{f}_i} < \infty
\end{split}
\end{equation}
and $a = \lim_{i\rightarrow +\infty} T_{\tilde{f}_{i}}$. From theorem 5 we obtain $a\in {\rm Arg} \min \limits_{f\in \mathcal{G}'_k} I(f)$. 


One can always find a subset $\{\tilde{\lambda}_{d_i}\}\subseteq \{\tilde{\lambda}_{i}\}$ such that $\tilde{\lambda}_{d_i} < \lambda_{i}$, $\tilde{\lambda}_{d_i}\rightarrow \infty$ and obtain
\begin{equation}
\begin{split}
\inf\limits_{f\in {\mathcal S} ({\mathbb R}^n)} I(f)+\lambda_i R(f)\geq \inf\limits_{f\in {\mathcal S} ({\mathbb R}^n)} I(f)+\tilde{\lambda}_{d_i} R(f)\geq 
I(\tilde{f}_{d_i}) + \tilde{\lambda}_{d_i} R(\tilde{f}_{d_i})-\epsilon_{d_i} \geq I(\tilde{f}_{d_i}) -\epsilon_{d_i} .
\end{split}
\end{equation}
Therefore,
\begin{equation}
\begin{split}
\lim_{i\rightarrow \infty} \inf\limits_{f\in {\mathcal S} ({\mathbb R}^n)} I(f)+\lambda_i R(f) \geq \lim_{i\rightarrow \infty} I(\tilde{f}_{d_i}) -\epsilon_{d_i}  = I(a) = \inf\limits_{f\in \mathcal{G}'_k} I(f)  = I(f^\ast).
\end{split}
\end{equation}
This proves that $\{{f_i}\}$ regularly solves (7) and $\lim_{i\rightarrow \infty} f_i = f^\ast$ i.e.  $f^\ast\in {\rm rsol\,} (I(f), R(f))$.
\end{proof}

\section{The alternating scheme in the dual space for $M({\mathbf x}, {\mathbf y}) = \zeta({\mathbf x}-{\mathbf y})$}\label{mod-alt-right2} 
When $M({\mathbf x}, {\mathbf y}) = \zeta({\mathbf x}-{\mathbf y})$, the alternating scheme~\ref{alternate-mod} allows for a reformulation in the dual space. By this we mean that in Scheme~\ref{alternate-mod} we substitute $\hat{\phi}_{t}$ for the original $\phi_{t}$. If the primal Scheme~\ref{alternate-mod} deals with  operators $S_{\phi}, S_{\phi_{t-1}}$, the dual version deals with vectors of functions $\sqrt{{\hat\zeta}}\frac{\partial {\hat\phi}}{\partial {\mathbf x}}, \sqrt{{\hat\zeta}}\frac{\partial {\hat\phi}_{t-1}}{\partial {\mathbf x}}$.
The substitution is based on the following simple fact:
\begin{theorem} If $M({\mathbf x}, {\mathbf y}) = \zeta({\mathbf x}-{\mathbf y}), \zeta, \hat{\zeta}\in C ({\mathbb R}^n)$ and $\forall {\mathbf x}\,\, \hat{\zeta}({\mathbf x})> 0$, then there exist constants $c_1$ and $c_2$ such that $\Vert S_{\phi} - P_{t-1}S_{\phi_{t-1}}\Vert _\ast^2 = c_1 \Vert \,\,\Vert \frac{\partial{\hat {\phi}}}{\partial {\mathbf x}} - P_{t-1} \frac{\partial{\hat {\phi}}_{t-1}}{\partial {\mathbf x}}\Vert _2\,\,\Vert ^2_{L_{2,\hat\zeta}({\mathbb R}^{n})}$ and
$\langle x_i f \vert  M \vert  x_j f \rangle = c_2
\langle \frac{\partial{\hat f}}{\partial x_i}, \frac{\partial{\hat f}}{\partial x_j}\rangle_{L_{2,\hat\zeta}({\mathbb R}^{n})}$
\end{theorem}
\begin{proof}
Let $f: {\mathbb R}^n\rightarrow {\mathbb C}$ be such that $\Vert x_i f\Vert _{L_2({\mathbb R}^n)} < \infty$. 
\begin{equation}
\begin{split}
{\rm O}_M[\psi] = \zeta \ast \psi \Rightarrow \mathcal{F}\left\{ {\rm O}_M[\psi]\right\} \propto {\hat\zeta} {\hat\psi} \Rightarrow 
\mathcal{F}\left\{\sqrt{{\rm O}_M}[\psi]\right\} \propto \sqrt{\hat\zeta} {\hat\psi} \Rightarrow \\
S_{f}[\psi]_i = {\rm Re\,} \langle x_i f, \sqrt{{\rm O}_M}[\psi]\rangle\propto 
{\rm Re\,} \langle \mathcal{F}\left\{x_i f\right\}, \mathcal{F}\left\{\sqrt{{\rm O}_M}[\psi]\right\}\rangle \propto 
{\rm Re\,} \langle {\rm i}\frac{\partial{\hat f}}{\partial x_i}, \sqrt{\hat\zeta} {\hat\psi} \rangle = {\rm Re\,} \langle {\rm i} \sqrt{\hat\zeta} \frac{\partial{\hat f}}{\partial x_i}, {\hat\psi} \rangle
\end{split}
\end{equation}
Since $S_f[\psi]_i =  {\rm Re\,}\langle (S_f)_{i}, \psi\rangle\propto  {\rm Re\,}\langle \widehat{(S_f)_{i}}, \hat{\psi}\rangle$, we obtain 
\begin{equation}\label{kappa}
\widehat{(S_f)_{i}} = \kappa\sqrt{\hat\zeta} \frac{\partial{\hat f}}{\partial x_i}
\end{equation} 
where $\kappa$ is a constant. 

Let us now introduce a vector of functions $V_{f} = \begin{bmatrix}
(S_f)_{1}, \cdots, (S_f)_{n}
\end{bmatrix}^T\in L^n_2({\mathbb R}^{n})$. 
Using~\ref{kappa} we obtain $\widehat{(S_f)_{i}} = \kappa\sqrt{\hat\zeta} \frac{\partial{\hat f}}{\partial x_i}$, and therefore $
{\widehat V_f} = \kappa\sqrt{\hat\zeta} \frac{\partial{\hat f}}{\partial {\mathbf x}}$. 
Thus, the expression $\Vert S_{\phi} - P_{t-1}S_{\phi_{t-1}}\Vert _\ast^2$ in the alternating scheme can be rewritten as
\begin{equation}
\begin{split}
\Vert V_{\phi} - P_{t-1}V_{\phi_{t-1}}\Vert ^2_{L^n_2({\mathbb R}^{n})} \propto 
\Vert \kappa\sqrt{\hat\zeta} \frac{\partial{\hat {\phi}}}{\partial {\mathbf x}} - P_{t-1} \kappa\sqrt{\hat\zeta} \frac{\partial{\hat {\phi}}_{t-1}}{\partial {\mathbf x}} \Vert ^2_{L^n_2({\mathbb R}^{n})} \propto 
\Vert \,\,\Vert \frac{\partial{\hat {\phi}}}{\partial {\mathbf x}} - P_{t-1} \frac{\partial{\hat {\phi}}_{t-1}}{\partial {\mathbf x}}\Vert _2\,\,\Vert ^2_{L_{2,\hat\zeta}({\mathbb R}^{n})}
\end{split}
\end{equation}
The matrix $M_f$ can also be calculated from $\hat{f}$ using the following identity:
\begin{equation}
\begin{split}
\langle x_i f, M[x_j f]\rangle = \langle x_i f, \zeta\ast (x_j f)\rangle \propto 
\langle \frac{\partial{\hat f}}{\partial x_i}, \hat{\zeta}\frac{\partial{\hat f}}{\partial x_j}\rangle = 
\langle \frac{\partial{\hat f}}{\partial x_i}, \frac{\partial{\hat f}}{\partial x_j}\rangle_{L_{2,\hat\zeta}({\mathbb R}^{n})}
\end{split}
\end{equation}
\end{proof}

Let us introduce a function $\tilde{I}$ such that $\tilde{I}(f) = I(\hat{f})$.
Then, we see that all steps in Scheme~\ref{alternate-mod} can be performed with ${\hat \phi_t}$ rather than with $\phi_t$, using the algorithm~\ref{alternate-mod2}.

Informally, the dual algorithm works as follows: at each iteration $t$ we compute a function ${\hat \phi}_{t}$ adapting it to data (the term $\tilde{I}({\hat \phi})$) and adapting its gradient field to the rank reduced gradient field of the previous ${\hat \phi}_{t-1}$. For a sufficiently large $T$, it will converge and ${\hat \phi}_{T}\approx {\hat \phi}_{T-1}$. Then, the second term in the last step  will be approximately equal to $\lambda  \Vert \,\,\Vert \frac{\partial{\hat \phi}_T}{\partial {\mathbf x}} - P_{T-1}\frac{\partial{\hat \phi}_{T}}{\partial {\mathbf x}}\Vert _2\,\,\Vert ^2_{L_{2,\hat\zeta}({\mathbb R}^{n})}$, enforcing $\frac{\partial{\hat \phi}_T}{\partial {\mathbf x}} \approx P_{T-1}\frac{\partial{\hat \phi}_{T}}{\partial {\mathbf x}}$ for random ${\mathbf x}\sim \frac{\hat\zeta}{\Vert \hat\zeta\Vert _{L_1}}$.
Thus, gradients $\frac{\partial{\hat \phi}_T}{\partial {\mathbf x}}$ lie in a $k$-dimensional subspace ${\rm col\,} P_{T-1}$. This last property is a characteristic property of functions from ${\mathcal F}_k$.
\begin{algorithm}
\begin{algorithmic}
\caption{The alternating scheme in the dual space.}\label{alternate-mod2}
\State $P_0 \longleftarrow  {\mathbf 0}, {\hat \phi}_0 \longleftarrow  {\mathbf 0}$
\For{$t = 1, \cdots, T$}
\State ${\hat \phi}_{t}\leftarrow  \arg\min\limits_{{\hat \phi}} \tilde{I}({\hat \phi})+\tilde{\lambda}  \Vert \,\,\Vert \frac{\partial{\hat \phi}}{\partial {\mathbf x}} - P_{t-1}\frac{\partial{\hat \phi}_{t-1}}{\partial {\mathbf x}}\Vert _2\,\,\Vert ^2_{L_{2,\hat\zeta}({\mathbb R}^{n})} $ 
\State Calculate $M_{t} = \begin{bmatrix} {\rm Re\,}
\langle \frac{\partial{\hat \phi_t}}{\partial x_i}, \frac{\partial{\hat \phi_t}}{\partial x_j}\rangle_{L_{2,\hat\zeta}({\mathbb R}^{n})}
\end{bmatrix}$
\State Find $\{{\mathbf v}_i\}^n_1$ s.t. $M_{t}{\mathbf v}_i=\lambda_i{\mathbf v}_i$, $\lambda_1\geq \cdots\geq \lambda_n$
\State $P_{t} \longleftarrow  \sum_{i=1}^k {\mathbf v}_i {\mathbf v}_i^T$
\EndFor
\State \textbf{Output:} ${\mathbf v}_1, \cdots, {\mathbf v}_k$
\end{algorithmic}
\end{algorithm}

Absolutely analogously to the Algorithm~\ref{alternate-mod2}, one can construct a dual algorithm that deals with inverse Fourier transforms of functions, i.e. with $\mathcal{F}^{-1}[\phi]$, $\mathcal{F}^{-1}[\phi_t]$, $M_{t} = \begin{bmatrix} {\rm Re\,}
\langle \frac{\partial{\mathcal{F}^{-1}[\phi_t]}}{\partial x_i}, \frac{\partial{\mathcal{F}^{-1}[\phi_t]}}{\partial x_j}\rangle_{L_{2,\mathcal{F}^{-1}[\zeta]}({\mathbb R}^{n})}
\end{bmatrix}$ etc. This version of the dual alternating scheme will be used for designing numerical algorithms for the Gaussian MMD-PCA and HM-MMD-PCA.

\section{Proofs for Section~\ref{Approximate-MMD-PCA}}
\begin{proof}[Proof of Theorem~\ref{polynomial}.] Let $X = [{\mathbf x}_1, \cdots, {\mathbf x}_N]$. Note that $H_f = X S X^T$ where $S = [P({\mathbf x}_i \cdot {\mathbf x}_j)]_{i,j\in [N]}$. For any $U\in \mathcal{O}(n)$ and we have 
\begin{equation}
H_{f_U} = (U^T X)  [P(U^T {\mathbf x}_i \cdot U^T {\mathbf x}_j)]_{i,j\in [N]} (U^T X)^T  = U^T H_f U.
\end{equation}
Let $U = [{\mathbf u}_1, \cdots, {\mathbf u}_n]$ where $\{{\mathbf u}_i\}_{i=1}^n$ are eigenvectors  such that $H_f{\mathbf u}_i=\lambda_i{\mathbf u}_i$. Then, the rotated distribution $f_U$ is such that $H_{f_U} $ is diagonal. Note that $\inf_{\nu\in \mathcal{P}_k}\|f_U-T_\nu\|_K = \inf_{\nu\in \mathcal{P}_k}\|f-T_\nu\|_K$.

Therefore, w.l.o.g. we can assume that principal components of $H_f$ are ${\mathbf e}_1, \cdots, {\mathbf e}_n$ and $H_f{\mathbf e}_i=\lambda_i {\mathbf e}_i, i\in [n]$, where $\{{\mathbf e}_i\}_{i=1}^n$ is a canonical basis in ${\mathbb R}^n$ and $\lambda_1\geq \lambda_2\geq \cdots\geq \lambda_n$ are eigenvalues of $H_f$. From the latter we conclude that $\langle x_i f| H|x_j f\rangle = \lambda_i\delta_{ij}$. Using $P({\mathbf x}\cdot {\mathbf y}) = \sum_{j=0}^{l-1}c_j\sum_{\alpha\in ({\mathbb N}\cup\{0\})^n: |\alpha|=j}\frac{j!}{\alpha!}{\mathbf x}^\alpha{\mathbf y}^\alpha$, we obtain
\begin{equation}
\langle x_i f| H|x_i f\rangle =\sum_{j=0}^{l-1}c_j \sum_{\alpha\in ({\mathbb N}\cup\{0\})^n: |\alpha|=j}\frac{j!}{\alpha!} ({\mathbb E}_{{\mathbf x}\sim f}x_i {\mathbf x}^\alpha)^2=\lambda_i.
\end{equation}

For an input distribution $f({\mathbf x}) = \frac{1}{N}\sum_{i=1}^N \delta^n({\mathbf x}-{\mathbf x}_i)$, let us denote $f'({\mathbf x}) = \frac{1}{N}\sum_{i=1}^N \delta^k({\mathbf x}_{1:k}-({\mathbf x}_i)_{1:k})\otimes \delta^{n-k}({\mathbf x}_{k+1:n})$, where ${\mathbf x} = [{\mathbf x}_{1:k}, {\mathbf x}_{k+1:n}]$ and ${\mathbf z}_{1:k}\in {\mathbb R}^k$ equals the first $k$ components of ${\mathbf z}\in {\mathbb R}^n$. By construction,
\begin{equation}
{\mathbb E}_{{\mathbf x}\sim f}x_{i_1}\cdots x_{i_s} = {\mathbb E}_{{\mathbf x}\sim f'}x_{i_1}\cdots x_{i_s}
\end{equation}
for $i_1,\cdots, i_s \in [k]$ and
\begin{equation}
 {\mathbb E}_{{\mathbf x}\sim f'}x_{i_1}\cdots x_{i_s}=0
\end{equation}
whenever $i_j\in [n]\setminus [k]$ for at least one $j\in [s]$. Therefore,
\begin{equation}
\begin{split}
\|f-f'\|^2_K = \sum_{i=1}^n   \sum_{j=0}^{l-1}c_j \sum_{\alpha\in ({\mathbb N}\cup\{0\})^n: |\alpha|=j}\frac{j!}{\alpha!}({\mathbb E}_{{\mathbf x}\sim f}x_i{\mathbf x}^\alpha-{\mathbb E}_{{\mathbf x}\sim f'}x_i{\mathbf x}^\alpha)^2 = \\
\sum_{i=1}^k \sum_{j=0}^{l-1}c_j \sum_{\alpha\in ({\mathbb N}\cup\{0\})^n: |\alpha|=j, |\alpha_{k+1:n}|\ne 0}\frac{j!}{\alpha!}({\mathbb E}_{{\mathbf x}\sim f}x_i{\mathbf x}^\alpha)^2 +\\
\sum_{i=k+1}^n \sum_{j=0}^{l-1}c_j \sum_{\alpha\in ({\mathbb N}\cup\{0\})^n: |\alpha|=j}\frac{j!}{\alpha!}({\mathbb E}_{{\mathbf x}\sim f}x_i{\mathbf x}^\alpha)^2 = F_1+F_2
\end{split}
\end{equation}
The second sum $F_2$ equals $\sum_{i=k+1}^n \lambda_i$. 
Let us compare the first sum, $F_1$, with the second, $F_2$.
$F_1$ is a sum of positive factors of $ ({\mathbb E}_{{\mathbf x}\sim f}{\mathbf x}^\alpha)^2$ where $\alpha_{1:k}\ne 0, \alpha_{k+1:n}\ne 0$. Let $e_i\in ({\mathbb N}\cup\{0\})^n$ be an $i$th canonical unit vector and $\alpha_i$ denote an $i$th component of $\alpha$. The coefficient in front of $({\mathbb E}_{{\mathbf x}\sim f}{\mathbf x}^\alpha)^2$ in $F_1$ equals $A_\alpha = c_{|\alpha|-1}(|\alpha|-1)!\sum_{i\in [k]:\alpha_i>0}\frac{1}{(\alpha-e_i)!}=\frac{c_{|\alpha|-1}(|\alpha|-1)!}{\alpha!}|\alpha_{1:k}|$ and the coefficient in front of $({\mathbb E}_{{\mathbf x}\sim f}{\mathbf x}^\alpha)^2$ in $F_2$ equals $B_\alpha = c_{|\alpha|-1}(|\alpha|-1)!\sum_{i\in [n]\setminus [k]: \alpha_i>0}\frac{1}{(\alpha-e_i)!} = \frac{c_{|\alpha|-1}(|\alpha|-1)!}{\alpha!}|\alpha_{k+1:n}|$. Since $|\alpha_{k+1:n}|\geq 1$ and $|\alpha_{1:k}|\leq l-1$, we conclude that $A_\alpha\leq (l-1)B_\alpha$. Therefore, $F_1\leq (l-1)F_2$.

Thus, overall we have
\begin{equation}
\|f-f'\|^2_K \leq l \sum_{i=k+1}^n \lambda_i.
\end{equation}
From $\|T_\mu -T_\nu\|^2_K = d_{\rm MMD}(\mu, \nu)^2$ the statement of theorem directly follows. 
\end{proof}

In our proof of Theorem~\ref{Gaussian-MMD}. we will need the following classical theorem.
\begin{theorem}[The Gaussian Poincaré  inequality] Let $g:{\mathbb R}^n\to {\mathbb R}$ be a smooth function, then
\begin{equation}
{\rm Var}_{{\mathbf x}\sim \mathcal{N}({\mathbf 0}, \sigma^2 I_n)} [g({\mathbf x})]\leq \sigma^2 {\mathbb E}_{{\mathbf x}\sim \mathcal{N}({\mathbf 0}, \sigma^2 I_n)} \|\nabla g({\mathbf x})\|^2.
\end{equation}
\end{theorem}
\begin{proof}[Proof of Theorem~\ref{Gaussian-MMD}.]
Let us assume w.l.o.g. that principal components of $H_f$ are ${\mathbf e}_1, \cdots, {\mathbf e}_n$ and $H_f{\mathbf e}_i=\lambda_i {\mathbf e}_i, i\in [n]$, where $\{{\mathbf e}_i\}_{i=1}^n$ is a canonical basis in ${\mathbb R}^n$ and $\lambda_1\geq \lambda_2\geq \cdots\geq \lambda_n$ are eigenvalues of $H_f$. From the latter we conclude that $\langle x_i f| H |x_j f\rangle = \lambda_i\delta_{ij}$.

 Note that $H({\mathbf x}, {\mathbf y}) = \mathcal{F}^{-1}[G_\sigma]({\mathbf x}-{\mathbf y})$.
Let $\hat{f} = \mathcal{F}[f]$, $$\hat{f'}({\mathbf x}) = {\mathbb E}_{{\mathbf y}_{1:n-k}\sim \mathcal{N}({\mathbf 0},\sigma^2 I_{n-k})}   \hat{f}({\mathbf x}_{1:k}, {\mathbf y}_{1:n-k})$$ and $f' = \mathcal{F}^{-1}[\hat{f'}]$ where ${\mathbf x} = [{\mathbf x}_{1:k}, {\mathbf x}_{k+1:n}]$. From the isometry of the Fourier transform, we have
\begin{equation}
\begin{split}
\|f-f'\|_K^2 = C_n {\mathbb E}_{{\mathbf x}\sim \mathcal{N}({\mathbf 0},\sigma^2 I_n)} \|\nabla_{{\mathbf x}} \hat{f}({\mathbf x}) -\nabla_{{\mathbf x}} \hat{f'}({\mathbf x})\|^2.
\end{split}
\end{equation}
The latter expression decomposes into two terms. The first is
\begin{equation}
\begin{split}
{\mathbb E}_{{\mathbf x}\sim \mathcal{N}({\mathbf 0},\sigma^2 I_n)} \|\nabla_{{\mathbf x}_{k+1:n}} \hat{f}({\mathbf x}) -\nabla_{{\mathbf x}_{k+1:n}} \hat{f'}({\mathbf x})\|^2 = {\mathbb E}_{{\mathbf x}\sim \mathcal{N}({\mathbf 0},\sigma^2 I_n)} \|\nabla_{{\mathbf x}_{k+1:n}} \hat{f}({\mathbf x})\|^2 = \\
{\mathbb E}_{{\mathbf x}\sim \mathcal{N}({\mathbf 0},\sigma^2 I_n)} \sum_{i=k+1}^n (\frac{\partial \hat{f}({\mathbf x})}{\partial x_i})^2 = \frac{1}{C_n}\sum_{i=k+1}^n \lambda_i\\
\end{split}
\end{equation}
The second is
\begin{equation}
\begin{split}
{\mathbb E}_{{\mathbf x}\sim \mathcal{N}({\mathbf 0},\sigma^2 I_{n})}  \|\nabla_{{\mathbf x}_{1:k}} \hat{f}({\mathbf x}_{1:k}, {\mathbf x}_{k+1:n}) -
 {\mathbb E}_{{\mathbf y}_{1:n-k}\sim \mathcal{N}({\mathbf 0},I_{n-k})}  \nabla_{{\mathbf x}_{1:k}} \hat{f}({\mathbf x}_{1:k}, {\mathbf y}_{1:n-k})\|^2 = \\
 {\mathbb E}_{{\mathbf x}_{1:k}\sim \mathcal{N}({\mathbf 0},\sigma^2 I_{k})} \sum_{i=1}^k {\rm Var}_{{\mathbf x}_{k+1:n}\sim \mathcal{N}({\mathbf 0},\sigma^2 I_{n-k})} [\frac{\partial \hat{f}({\mathbf x}_{1:k}, {\mathbf x}_{k+1:n})}{\partial x_i} ].
\end{split}
\end{equation}
The latter can be bounded using the Gaussian Poincaré inequality by
\begin{equation}
\begin{split}
\sigma^2 {\mathbb E}_{{\mathbf x}\sim \mathcal{N}({\mathbf 0},\sigma^2 I_{n})}   \sum_{i=1}^k\sum_{j=k+1}^n |\frac{\partial^2 \hat{f}({\mathbf x}_{1:k}, {\mathbf x}_{k+1:n})}{\partial x_i\partial x_j}|^2 .
\end{split}
\end{equation}
After changing an order of summations one can bound the internal sum using integration by parts, i.e.
\begin{equation}
\begin{split}
\sum_{i=1}^k\int_{{\mathbb R}^n} |\frac{\partial^2 \hat{f}({\mathbf x})}{\partial x_i\partial x_j}|^2 G_\sigma({\mathbf x}) d{\mathbf x} = \sum_{i=1}^k -\int_{{\mathbb R}^n} \frac{\partial \hat{f}({\mathbf x})^\ast}{\partial x_j} \frac{\partial}{\partial x_i}(\frac{\partial^2 \hat{f}({\mathbf x})}{\partial x_i\partial x_j} G_\sigma({\mathbf x})) d{\mathbf x} =\\ 
\sum_{i=1}^k-\int_{{\mathbb R}^n} \frac{\partial \hat{f}({\mathbf x})^\ast}{\partial x_j} (\frac{\partial^3 \hat{f}({\mathbf x})}{\partial x^2_i\partial x_j} - \frac{\partial^2 \hat{f}({\mathbf x})}{\partial x_i\partial x_j}\frac{x_i}{\sigma^2}) G_\sigma({\mathbf x}) d{\mathbf x} = \\
\int_{{\mathbb R}^n}\frac{\partial \hat{f}({\mathbf x})^\ast}{\partial x_j} (\frac{\partial \Delta_{1:k} \hat{f}({\mathbf x})}{\partial x_j} -\frac{1}{\sigma^2} \frac{\partial ({\mathbf x}_{1:k}\cdot \nabla_{1:k} \hat{f}({\mathbf x}))}{\partial x_j}) G_\sigma({\mathbf x}) d{\mathbf x}
\end{split}
\end{equation}
Then, using the Cauchy–Schwarz inequality we bound the latter by
\begin{equation}
\begin{split}
\big(\int_{{\mathbb R}^n}|\frac{\partial \hat{f}({\mathbf x})}{\partial x_j}|^2 G_\sigma({\mathbf x}) d{\mathbf x}\big)^{1/2} \big(\int_{{\mathbb R}^n} |\frac{\partial \Delta_{1:k} \hat{f}({\mathbf x})}{\partial x_j} -\frac{1}{\sigma^2} \frac{\partial ({\mathbf x}_{1:k}\cdot \nabla_{1:k} \hat{f}({\mathbf x}))}{\partial x_j}|^2 G_\sigma({\mathbf x}) d{\mathbf x}\big)^{1/2}
\end{split}
\end{equation}
The first term equals $(\frac{1}{C_n}\lambda_j)^{1/2}$ and the second term, after making the inverse Fourier transform, is bounded by
\begin{equation}
\begin{split}
C_n^{-1/2}  (\langle x_j\|{\mathbf x}_{1:k}\|^2 f | H |x_j\|{\mathbf x}_{1:k}\|^2 f\rangle)^{1/2}+ 
C_n^{-1/2}\sum_{i=1}^k   (\langle x_i x_j f | T_i ({\mathbf x}- {\mathbf y}) | y_i y_j f\rangle)^{1/2}
\end{split}
\end{equation}
where $T_i=\mathcal{F}^{-1}[\frac{x_i^2 G_\sigma}{\sigma^2}] = e^{-\sigma^2\|{\mathbf x}\|^2/2}(1-\sigma^2 x_i^2)$. Since $T_i({\mathbf x}- {\mathbf y}) = H({\mathbf x},{\mathbf y})(1-\sigma^2 x_i^2-\sigma^2 y_i^2+2\sigma^2 x_iy_i)$, we have $\langle x_i x_j f | T_i ({\mathbf x}- {\mathbf y}) | y_i y_j f\rangle = \langle x_i x_j f | H | x_i x_j f\rangle-2\sigma^2 \langle x^3_i x_j f | H | x_i x_j f\rangle+2\sigma^2 \langle x^2_i x_j f | H | x^2_i x_j f\rangle$. After noting that 
\begin{equation}
|\langle {\mathbf x}^\alpha f | H | {\mathbf x}^\beta f\rangle|\leq \int|{\mathbf x}^{\alpha}f({\mathbf x})|d{\mathbf x}\cdot \int|{\mathbf x}^{\beta}f({\mathbf x})|d{\mathbf x},
\end{equation}
 we finally obtain
\begin{equation}
\begin{split}
\|f-f'\|_K^2 \leq 
\sum_{j=k+1}^n \lambda_j^{1/2} \sum_{i=1}^k\big( \int|x_ix_jf({\mathbf x})|d{\mathbf x}+\\
\sqrt{2}\sigma\sqrt{\int |x^3_ix_jf({\mathbf x})|d{\mathbf x} \int |x_ix_jf({\mathbf x})|d{\mathbf x}}+
\sqrt{2}\sigma \int|x^2_ix_jf({\mathbf x})|d{\mathbf x}\big)\leq 
M\sum_{j=k+1}^n \lambda_j^{1/2}
\end{split}
\end{equation}
where $M = \mathcal{O}(\int_{{\mathbb R}^n} \|{\mathbf x}\|^2|f({\mathbf x})|d{\mathbf x}+\sqrt{2}\sigma \sqrt{\int_{{\mathbb R}^n} \|{\mathbf x}\|^4|f({\mathbf x})|d{\mathbf x} \int_{{\mathbb R}^n} \|{\mathbf x}\|^2|f({\mathbf x})|d{\mathbf x}}+\sqrt{2}\sigma \int_{{\mathbb R}^n} \|{\mathbf x}\|^3|f({\mathbf x})|d{\mathbf x})$.

By construction, $f'\in\mathcal{G}_k'$ and it can be approached by elements of $\mathcal{G}_k$ w.r.t. norm $\|\cdot\|_K$. The statement of theorem directly follows from this observation.
\end{proof}

\section{A numerical alternating scheme for the Gaussian MMD-PCA}\label{numerical-mmd}

\subsection{Structure of $\mathcal{F}^{-1}[\mathcal{P}_k]$}
From theorems~\ref{easy} and~\ref{g-dual-f}, $\mathcal{F}^{-1}[\mathcal{P}_k]\subseteq \overline{\mathcal{F}_k}^\ast$.
In fact, Bochner's theorem~\cite{bochner1932vorlesungen} gives us that the inverse Fourier transform of any positive finite Borel measure is a continuous positive definite function. That is, if $f\in \mathcal{F}^{-1}[\mathcal{P}]$, then for any distinct ${\mathbf y}_1, \cdots, {\mathbf y}_s\in {\mathbb R}^n$ the matrix
$[f({\mathbf y}_i-{\mathbf y}_j)]_{i,j=\overline{1,n}}$ is positive semidefinite. Since $\mu({\mathbb R}^n)=1$, we additionally have $f({\mathbf 0})=1$.
Let ${\rm PDF}$ denote the set of all continuous positive definite functions on ${\mathbb R}^n$ and
\begin{equation}\label{bochner-dual}
\begin{split}
\mathcal{M}_k = \{f\in {\rm PDF}\vert  \exists{\mathbf v}_1, ..., {\mathbf v}_k\in {\mathbb R}^n, g: {\mathbb R}^k\rightarrow {\mathbb C} \,\,{\rm s.t.}\,\,
f({\mathbf x})  = g({\mathbf v}^T_1{\mathbf x}, ..., {\mathbf v}^T_k{\mathbf x}), f({\mathbf 0})=1\}.
\end{split}
\end{equation}
Thus, the following characterization of $\mathcal{F}^{-1}[\mathcal{P}_k]$ becomes evident.
\begin{theorem}\label{inclusion} $\mathcal{F}^{-1}[\mathcal{P}_k]=\mathcal{M}_k$.
\end{theorem}
\subsection{The dual form of the Gaussian MMD-PCA} 
Recall that $k({\mathbf x})= G_h^n({\mathbf x})$. Let us define another Gaussian kernel $\gamma({\mathbf x}) =  e^{-\frac{h^2\vert {\mathbf x}\vert ^2}{2}} = \mathcal{F}^{-1}[k]$. 
Let $p_{\rm{data}}({\mathbf x})$ denote the characteristic function of the random vector ${\mathbf X}_{\rm{data}}\sim\mu_{\rm{data}}$. By definition, $p_{\rm{data}}({\mathbf x}) = {\mathbb E}[e^{{\rm i}{\mathbf X}_{\rm{data}}^T{\mathbf x}}] = \frac{1}{N}\sum_{i=1}^N e^{{\rm i}{\mathbf x}^T_i{\mathbf x}}$.
Thus, $p_{\rm{data}} = \mathcal{F}^{-1}[\mu_{\rm{data}}]$ and $\mu_{\rm{data}} = \mathcal{F}[p_{\rm{data}}]$.

Using the isometry property of the inverse Fourier transform for $L_2({\mathbb R}^n)$ and the convolution theorem, we see that
\begin{equation}
\begin{split}
d_{\textsc{MMD}}(\mu, \nu) = \Vert k \ast \mu -k\ast \nu\Vert _{L_2({\mathbb R}^n)}
\propto
 \Vert \gamma({\mathbf x}) (\mathcal{F}^{-1}[\mu]({\mathbf x}) - \mathcal{F}^{-1}[\nu]({\mathbf x}))\Vert _{L_{2}({\mathbb R}^n)}.
\end{split}
\end{equation}
Thus, from Theorem~\ref{inclusion} we obtain that the task~\ref{MMD-task} is equivalent to
\begin{equation}\label{dual}
\Vert p_{\rm{data}} - q\Vert _{L_{2, \gamma^2}({\mathbb R}^n)}  \rightarrow \min_{q\in \mathcal{M}_k}
\end{equation}
where $L_{2, \gamma^2}({\mathbb R}^n)\eqdef L_{2, \nu}({\mathbb R}^n)$ with $d\nu=\gamma^2d{\mathbf x}$. 
\subsection{Algorithms for the Gaussian MMD-PCA}
Let $\Pi_k: {\mathcal{G}_k}\rightarrow \{1,+\infty\}$ and ${\text M}_k: {\mathcal{F}_k}\rightarrow \{1,+\infty\}$ be simple penalty functions:
$$
\Pi_k (\phi) = 1, \text{if}\,\,\phi\in \mathcal{P}_k\,\,\text{and}\,\,\Pi_k (\phi) = \infty, \text{otherwise}
$$
$$
{\text M}_k (\phi) = 1, \text{if}\,\,\phi\in \mathcal{M}_k\,\,\text{and}\,\,{\text M}_k (\phi) = \infty, \text{otherwise}
$$
Then, the task~\ref{MMD-task} is equivalent to
\begin{equation}
I(\phi) =  d^2_{\rm MMD}(\mu_{\rm data}, \phi) \Pi_k (\phi)\rightarrow\inf_{\phi\in {\mathcal{G}_k}}.
\end{equation}
From the result of the previous section we see that if $I(\phi) = \tilde{I}({\hat \phi})$, then
\begin{equation}
\tilde{I}({\hat \phi}) = \Vert p_{\rm{data}} - {\hat \phi}\Vert ^2_{L_{2, \gamma^2}({\mathbb R}^n)}{\text M}_k ({\hat \phi}).
\end{equation}
Thus, the Algorithm~\ref{alternate-mod3} is an adaptation of Algorithm~\ref{alternate-mod2} to MMD-PCA.

\begin{algorithm}
\begin{algorithmic}
\caption{The alternating scheme in the dual space for the Gaussian MMD-PCA}\label{alternate-mod3}
\State $P_0 \longleftarrow  {\mathbf 0}, q_0 \longleftarrow  {\mathbf 0}$
\For{$t = 1, \cdots, T$}
\State{1} $q_{t}\longleftarrow  \arg\min\limits_{q\in \mathcal{M}_k} \int_{{\mathbb R}^n}\gamma({\mathbf x})^2 \vert p_{\rm{data}}({\mathbf x})-q({\mathbf x})\vert ^2 d {\mathbf x} +\lambda  \int_{{\mathbb R}^{n}} {\hat \zeta}({\mathbf x}) \Vert \frac{\partial q}{\partial {\mathbf x}} - P_{t-1}\frac{\partial q_{t-1}}{\partial {\mathbf x}}\Vert _2^2 d {\mathbf x}$
\State{2} Calculate $M_{t} = \begin{bmatrix}
\langle \frac{\partial q_t}{\partial x_i}, \frac{\partial q_t}{\partial x_j}\rangle_{L_{2,{\hat \zeta}}({\mathbb R}^{n})}
\end{bmatrix}$
\State{3} Find $\{{\mathbf v}_i\}^n_1$ s.t. $M_{t}{\mathbf v}_i=\lambda_i{\mathbf v}_i$, $\lambda_1\geq \cdots\geq \lambda_n$
\State{4} $P_{t} \longleftarrow  \sum_{i=1}^k {\mathbf v}_i {\mathbf v}_i^T$
\EndFor
\State \textbf{Output:} $\mathcal{L} = {\rm span}({\mathbf v}_1, \cdots, {\mathbf v}_k)$
\end{algorithmic}
\end{algorithm}

If the function  $p_{\rm{data}}$ is real-valued, then only real-valued functions can appear in the Algorithm~\ref{alternate-mod3}. This assumption can be satisfied by adding reflections of initial points to the dataset (after it was centered).

At step 1, we search over $q$ given in the following parameterized form:
\begin{equation}\label{Bochner}
q_\theta({\mathbf x}) = \sum_{i=1}^{{\rm nn}} \alpha_i cos(\omega_i^T {\mathbf x})
\end{equation}
where $\alpha_i>0$ and $\sum_{i=1}^{{\rm nn}} \alpha_i=1$. In our implementation, we set $[\alpha_i]_{i=\overline{1, {\rm nn}}} = {\rm softmax} ([u_i]_{i=\overline{1, {\rm nn}}})$ and $u_i$'s are unconstrained. The number of neurons in a single layer neural network with a cosine activation function, ${\rm nn}$, is a hyperparameter. Let us denote parameters $\{\omega_i, u_i\}_{i=1}^{\rm nn}$ by $\theta$. It is easy to see the function $q_\theta$ is positive definite. Moreover, using Theorem 2 from~\cite{Barron}, it can be shown that a set of all such functions, i.e. the convex hull of $\{cos(\omega^T {\mathbf x})\vert  \omega\in {\mathbb R}^{n}\}$, is dense in a set of real-valued functions from $\mathcal{M}_k$. Though this parameterization is quite natural, finding architectures with more expressive power in a  space of real-valued positive definite functions is an open problem.

Now, to minimize 
\begin{equation}
\begin{split}
\Psi(\theta) = \int_{{\mathbb R}^n}\gamma({\mathbf x})^2 \vert p_{\rm{data}}({\mathbf x})-q_\theta({\mathbf x})\vert ^2 d {\mathbf x} + 
\lambda  \int_{{\mathbb R}^{n}} {\hat \zeta}({\mathbf x}) \Vert \frac{\partial q_\theta}{\partial {\mathbf x}} - P_{t-1}\frac{\partial q_{\theta_{t-1}}}{\partial {\mathbf x}}\Vert _2^2 d {\mathbf x}
\end{split}
\end{equation}
with stochastic gradient descent methods (in our case, the Adam optimizer) we need to have an unbiased estimator of
\begin{equation}
\begin{split}
\nabla_\theta \Psi(\theta) \propto {\mathbb E}_{{\mathbf z}\sim\gamma^2} \nabla_\theta \vert p_{\rm{data}}({\mathbf z})-q_\theta({\mathbf z})\vert ^2 + 
{\tilde\lambda}  {\mathbb E}_{{\mathbf z}'\sim{\hat \zeta}} \nabla_\theta \Vert \frac{\partial q_\theta}{\partial {\mathbf x}}({\mathbf z}') - P_{t-1}\frac{\partial q_{\theta_{t-1}}}{\partial {\mathbf x}}({\mathbf z}')\Vert _2^2
\end{split}
\end{equation}
where ${\mathbf z} \sim f$ denotes that the random vector ${\mathbf z}$ is sampled according to the probability density function $\frac{f ({\mathbf x})}{\int_{{\mathbb R}^{n}} f ({\mathbf x}) d {\mathbf x}}$.
Thus, a natural estimator of the gradient is
\begin{equation}
\begin{split}
\frac{1}{m}\sum_{i=1}^m \nabla_\theta \vert p_{\rm{data}} ({\mathbf z_i}) - q_\theta ({\mathbf z_i})\vert ^2 + 
\frac{{\tilde\lambda}}{m}\sum_{i=1}^{m} \nabla_\theta\Vert \frac{\partial q_\theta (\text{\boldmath$\xi$}_i))}{\partial {\mathbf x}} - P_{t-1} \frac{\partial q_{\theta_{t-1}} (\text{\boldmath$\xi$}_i))}{\partial {\mathbf x}}\Vert _2^2
\end{split}
\end{equation}
where $\{{\mathbf z}_i\}_{i=1}^m\sim^{iid} \gamma^2$ and $\{\text{\boldmath$\xi$}_i\}_{i=1}^{m}\sim^{iid} {\hat \zeta}$.

The last important issue with the practical numerical algorithm is the calculation of $M_t$ at step 2. By construction,
\begin{equation}
M_t = {\mathbb E}_{{\boldsymbol \chi} \sim {\hat \zeta}}\frac{\partial q_t}{\partial {\mathbf x}} ({\boldsymbol \chi}) \frac{\partial q_t}{\partial {\mathbf x}} ({\boldsymbol \chi})^T.
\end{equation}
In practice we sample ${\boldsymbol \chi}_1, \cdots, {\boldsymbol \chi}_l \sim {\hat \zeta}$ and estimate $M_t$ as 
\begin{equation}
M_t \approx \frac{1}{l}\sum_{i=1}^l \frac{\partial q_t}{\partial {\mathbf x}} ({\boldsymbol \chi}_i) \frac{\partial q_t}{\partial {\mathbf x}} ({\boldsymbol \chi}_i)^T.
\end{equation}

The details of the numerical algorithm~\ref{numerical-dual-mmd} are given below. In all our experiments with MMD-PCA we set ${\hat \zeta} = \gamma^2$.

\begin{algorithm}
\begin{algorithmic}
\caption{The numerical algorithm for the Gaussian MMD-PCA. Hyperparameters:  ${\tilde\lambda}, h, \sigma, m, l, \alpha, \beta_1, \beta_2, {\rm nn}$.}\label{numerical-dual-mmd}
\State $P_0 \longleftarrow  {\mathbf 0}, \theta_0 \longleftarrow  {\mathbf 0}$

\For{$t = 1, \cdots, T$}

\While{$\theta$ has not converged}
\State Sample $\{{\mathbf z}_i\}_{i=1}^m \sim^{iid} \gamma^2$

\State Sample $\{\text{\boldmath$\xi$}_i\}_{i=1}^{m} \sim^{iid} {\hat \zeta}$

\State $L\longleftarrow \frac{1}{m}\sum_{i=1}^m \vert p_{\rm{data}} ({\mathbf z_i}) - q_\theta ({\mathbf z_i})\vert ^2 + \frac{{\tilde\lambda}}{m}\sum_{i=1}^{m} \Vert \frac{\partial q_\theta (\text{\boldmath$\xi$}_i))}{\partial {\mathbf x}} - P_{t-1} \frac{\partial q_{\theta_{t-1}} (\text{\boldmath$\xi$}_i))}{\partial {\mathbf x}}\Vert _2^2$

\State $\theta \longleftarrow {\rm Adam} (\nabla_{\theta} L, \theta, \alpha, \beta_1, \beta_2)$
\EndWhile
\State $\theta_t \longleftarrow \theta$

\State Sample $\{\text{\boldmath$\chi$}_i\}_{i=1}^{l} \sim^{iid} {\hat \zeta}$

\State Calculate $M_{t} = \frac{1}{l}\sum_{i=1}^l \frac{\partial q_{\theta_{t}} (\text{\boldmath$\chi$}_i))}{\partial {\mathbf x}} \frac{\partial q_{\theta_{t}} (\text{\boldmath$\chi$}_i))}{\partial {\mathbf x}}^T$

\State Find $\{{\mathbf v}_i\}^n_1$ s.t. $M_{t}{\mathbf v}_i=\lambda_i{\mathbf v}_i$, $\lambda_1\geq \cdots\geq \lambda_n$

\State $P_{t} \longleftarrow  \sum_{i=1}^k {\mathbf v}_i {\mathbf v}_i^T$
\EndFor

\State \textbf{Output:} ${\mathbf v}_1, \cdots, {\mathbf v}_k$
\end{algorithmic}
\end{algorithm}

\section{A numerical alternating scheme for HM-MMD-PCA}\label{numerical-hm}
\subsection{The dual form of HM-MMD-PCA} 
Due to a well-known relationship between moments of the probability measure $\mu$ and its characteristic function $p$, i.e. ${\rm i}^s m_{i_1 \cdots i_s} = \frac{\partial^s p ({\mathbf 0})}{\partial x_{i_1}\cdots \partial x_{i_s}}$, the task~\eqref{HM-task} is equivalent to
\begin{equation}\label{HM-task-dual}
\sum_{s=1}^4\frac{\lambda_s}{n^s}\sum_{1\leq i_1, \cdots, i_s\leq n}\vert \frac{\partial^s p_{\rm data}({\mathbf 0})}{\partial x_{i_1}\cdots \partial x_{i_s}}-\frac{\partial^s q({\mathbf 0})}{\partial x_{i_1}\cdots \partial x_{i_s}}\vert ^2 \rightarrow \min_{q\in \mathcal{M}_k}.
\end{equation}
Note that the maximum mean discrepancy distance for the Gaussian kernel and the distance based on higher moments are substantially different. Indeed, even if we set $h$ as a large value (which makes $\frac{1}{h}\approx 0$), the MMD distance, unlike the HM distance, neglects higher order derivatives of the characteristic functions in the neigbourhood of the origin. Moreover, from the dual form~\eqref{HM-task-dual} it is clear that $d_{\rm HM}(\mu_{\rm data}, \nu)$ is a degenerate case of a weighted Sobolev norm between characteristic functions of $\mu_{\rm data}$ and $\nu$. 

\subsection{Algorithms for HM-MMD-PCA}
Analogously to the case of MMD-PCA we see that the task~\eqref{HM-task} is equivalent to:
\begin{equation}
I(\phi) = d_{\rm HM}(\mu_{\rm data}, \phi)^2\Pi_k (\phi) \rightarrow\inf_{\phi\in {\mathcal{G}_k}}
\end{equation}
and
\begin{equation}
\begin{split}
\tilde{I}({\hat \phi}) = 
\sum_{s=1}^4\frac{\lambda_s}{n^s}\sum_{1\leq i_1, \cdots, i_s\leq n}\vert \frac{\partial^s p_{\rm data}({\mathbf 0})}{\partial x_{i_1}\cdots \partial x_{i_s}}-\frac{\partial^s {\hat \phi}({\mathbf 0})}{\partial x_{i_1}\cdots \partial x_{i_s}}\vert ^2 {\text M}_k ({\hat \phi}). 
\end{split}
\end{equation}
Thus, the Algorithm~\ref{hm-alternate-mod} is an adaptation of Algorithm~\ref{alternate-mod2} to HM-MMD-PCA.

\begin{algorithm}
\begin{algorithmic}
\caption{The alternating scheme in the dual space for HM-MMD-PCA}\label{hm-alternate-mod}
\State $P_0 \longleftarrow  {\mathbf 0}, q_0 \longleftarrow  {\mathbf 0}$
\For{$t = 1, \cdots, T$}
\State{1} $q_{t}\longleftarrow  \arg\min\limits_{q\in \mathcal{M}_k} \sum_{s=1}^4\frac{\lambda_s}{n^s}\sum_{1\leq i_1, \cdots, i_s\leq n}\vert \frac{\partial^s p_{\rm data}({\mathbf 0})}{\partial x_{i_1}\cdots \partial x_{i_s}}-\frac{\partial^s q({\mathbf 0})}{\partial x_{i_1}\cdots \partial x_{i_s}}\vert ^2 +\lambda  \int_{{\mathbb R}^{n}} {\hat \zeta}({\mathbf x}) \Vert \frac{\partial q}{\partial {\mathbf x}} - P_{t-1}\frac{\partial q_{t-1}}{\partial {\mathbf x}}\Vert _2^2 d {\mathbf x}$
\State{2} Calculate $M_{t} = \begin{bmatrix}
\langle \frac{\partial q_t}{\partial x_i}, \frac{\partial q_t}{\partial x_j}\rangle_{L_{2,{\hat \zeta}}({\mathbb R}^{n})}
\end{bmatrix}$
\State{3} Find $\{{\mathbf v}_i\}^n_1$ s.t. $M_{t}{\mathbf v}_i=\lambda_i{\mathbf v}_i$, $\lambda_1\geq \cdots\geq \lambda_n$
\State{4} $P_{t} \longleftarrow  \sum_{i=1}^k {\mathbf v}_i {\mathbf v}_i^T$
\EndFor
\State \textbf{Output:} $\mathcal{L} = {\rm span}({\mathbf v}_1, \cdots, {\mathbf v}_k)$
\end{algorithmic}
\end{algorithm}

Again, as in a numerical algorithm for MMD-PCA, at step 1, we search over $q$ given in the form~\eqref{Bochner}. The objective of step 1 can be represented as
\begin{equation}
\begin{split}
\Phi(\theta) = \sum_{s=1}^4\lambda_s {\mathbb E}_{i_1, \cdots, i_s\sim^{iid} {\mathcal U}(1,n) }\vert \frac{\partial^s (p_{\rm data}-q_\theta)({\mathbf 0})}{\partial x_{i_1}\cdots \partial x_{i_s}}\vert ^2 + 
{\tilde\lambda}  {\mathbb E}_{{\mathbf z}'\sim{\hat \zeta}} \Vert \frac{\partial q_\theta}{\partial {\mathbf x}}({\mathbf z}') - P_{t-1}\frac{\partial q_{\theta_{t-1}}}{\partial {\mathbf x}}({\mathbf z}')\Vert _2^2
\end{split}
\end{equation}
where ${\mathcal U}(1,n)$ is the discrete uniform distribution over $\{1, \cdots, n\}$.
To apply the stochastic gradient descent methods we need to have an unbiased estimator of $\nabla_\theta \Phi(\theta)$ which is equal to
\begin{equation}
\begin{split}
\sum_{s=1}^4\lambda_s {\mathbb E}_{i_1, \cdots, i_s\sim^{iid} {\mathcal U}(1,n) } \nabla_\theta \vert \frac{\partial^s (p_{\rm data}-q_\theta)({\mathbf 0})}{\partial x_{i_1}\cdots \partial x_{i_s}}\vert ^2 +
{\tilde\lambda}  {\mathbb E}_{{\mathbf z}'\sim{\hat \zeta}} \nabla_\theta \Vert \frac{\partial q_\theta}{\partial {\mathbf x}}({\mathbf z}') - P_{t-1}\frac{\partial q_{\theta_{t-1}}}{\partial {\mathbf x}}({\mathbf z}')\Vert _2^2.
\end{split}
\end{equation}
Thus, a natural estimator of the gradient is:
\begin{equation}
\begin{split}
\sum_{s=1}^4\frac{\lambda_s}{m_1}\sum_{i=1}^{m_1} \nabla_\theta \vert \frac{\partial^s (p_{\rm data}-q_\theta)({\mathbf 0})}{\partial x_{a[s,i,1]}\partial x_{a[s,i,2]}\cdots \partial x_{a[s,i,s]}}\vert ^2 + 
\frac{{\tilde\lambda}}{m_2}\sum_{i=1}^{m_2} \nabla_\theta\Vert \frac{\partial q_\theta (\text{\boldmath$\xi$}_i))}{\partial {\mathbf x}} - P_{t-1} \frac{\partial q_{\theta_{t-1}} (\text{\boldmath$\xi$}_i))}{\partial {\mathbf x}}\Vert _2^2
\end{split}
\end{equation}
where $\{a[s,i,j]\}_{s=\overline{1,4}, i=\overline{1,m_1}, j=\overline{1,s}}\sim^{iid} {\mathcal U}(1,n)$ and $\{\text{\boldmath$\xi$}_i\}_{i=1}^{m_2}\sim^{iid} {\hat \zeta}$. Overall, we obtain the following Algorithm~\ref{hm-alternate-dual}.

\begin{algorithm}
\begin{algorithmic}
\caption{The numerical algorithm for HM-MMD-PCA. Hyperparameters:  ${\tilde\lambda}, \{\lambda_s\}_{s=\overline{1,4}}, m_1, m_2, l, \alpha, \beta_1, \beta_2, {\rm nn}$.}\label{hm-alternate-dual}
\State $P_0 \longleftarrow  {\mathbf 0}, \theta_0 \longleftarrow  {\mathbf 0}$

\For{$t = 1, \cdots, T$}

\While{$\theta$ has not converged}
\State Sample $\{a[s,i,j]\}_{s=\overline{1,4}, i=\overline{1,m_1}, j=\overline{1,s}}\sim^{iid} {\mathcal U}(1,n)$

\State Sample $\{\text{\boldmath$\xi$}_i\}_{i=1}^{m_2} \sim^{iid} {\hat \zeta}$

\State $L\longleftarrow \sum\limits_{s=1}^4\frac{\lambda_s}{m_1}\sum\limits_{i=1}^{m_1} \nabla_\theta \vert \frac{\partial^s (p_{\rm data}-q_\theta)({\mathbf 0})}{\partial x_{a[s,i,1]}\partial x_{a[s,i,2]}\cdots \partial x_{a[s,i,s]}}\vert ^2 + 
\frac{{\tilde\lambda}}{m_2}\sum\limits_{i=1}^{m_2} \nabla_\theta\Vert \frac{\partial q_\theta (\text{\boldmath$\xi$}_i))}{\partial {\mathbf x}} - P_{t-1} \frac{\partial q_{\theta_{t-1}} (\text{\boldmath$\xi$}_i))}{\partial {\mathbf x}}\Vert _2^2$

\State $\theta \longleftarrow {\rm Adam} (\nabla_{\theta} L, \theta, \alpha, \beta_1, \beta_2)$
\EndWhile
\State $\theta_t \longleftarrow \theta$

\State Sample $\{\text{\boldmath$\chi$}_i\}_{i=1}^{l} \sim^{iid} {\hat \zeta}$

\State Calculate $M_{t} = \frac{1}{l}\sum_{i=1}^l \frac{\partial q_{\theta_{t}} (\text{\boldmath$\chi$}_i))}{\partial {\mathbf x}} \frac{\partial q_{\theta_{t}} (\text{\boldmath$\chi$}_i))}{\partial {\mathbf x}}^T$

\State Find $\{{\mathbf v}_i\}^n_1$ s.t. $M_{t}{\mathbf v}_i=\lambda_i{\mathbf v}_i$, $\lambda_1\geq \cdots\geq \lambda_n$

\State $P_{t} \longleftarrow  \sum_{i=1}^k {\mathbf v}_i {\mathbf v}_i^T$
\EndFor

\State \textbf{Output:} ${\mathbf v}_1, \cdots, {\mathbf v}_k$
\end{algorithmic}
\end{algorithm}
\section{A numerical alternating scheme for WD-PCA}\label{numerical-wd}
Let us consider the case $p=1$ and denote $W(\mu, \nu) = W_1(\mu, \nu)$. By Theorem~\ref{transport}, the task~\ref{robust-pca} is equivalent to $\min_{\mu\in {\mathcal P}_k} W (\mu, \mu_{\rm{data}})$, or to the following task:
\begin{equation}
I(\phi)\rightarrow\inf_{\phi\in {\mathcal{G}_k}}
\end{equation}
where $I(T_\mu) = W (\mu, \mu_{\rm{data}})$ if $\mu\in {\mathcal{P}_k}$ and $I(\phi)=\infty$, if otherwise.
The alternating scheme~\ref{alternate-mod} is designed to solve the penalty form of the problem, i.e.
\begin{equation}
I(\phi) + \lambda R(\phi) \rightarrow \min_{\phi\in {\mathcal S}({\mathbb R}^n)},
\end{equation}
which is equivalent to
\begin{equation}
W(\phi, \mu_{\rm{data}}) + \lambda R(\phi) \rightarrow \min_{\phi\in {\mathcal S}_p({\mathbb R}^n)},
\end{equation}
where ${\mathcal S}_p({\mathbb R}^n)\subseteq {\mathcal S}({\mathbb R}^n)$ is a set of Schwartz functions that can serve as pdf: $\phi({\mathbf x})\geq 0$, $\int_{{\mathbb R}^n} \phi({\mathbf x}) d{\mathbf x} = 1$.
A numerical version of the alternating scheme requires additional specifications on a) how to minimize over $\phi$ at step 1, and b) how to estimate $M_{\phi_t}$.

\subsection{How to minimize over $\phi$?}
In the case of WD-PCA, the minimization step of the alternating scheme makes the following:
\begin{equation}\label{star}
\phi_{t}\longleftarrow  \arg\min\limits_{\phi\in {\mathcal S}_p({\mathbb R}^n)} W (\phi, \mu_{\rm{data}})+\lambda \Vert S_\phi - P_{t-1}S_{\phi_{t-1}}\Vert ^2
\end{equation}
where $S_f = \sqrt{{\rm O}_M}[{\mathbf x} f({\mathbf x})]$.

For a numerical implementation of that step we need to choose some family of functions that is dense in ${\mathcal S}_p({\mathbb R}^n)$ (or, rich enough to approach the solution $\mu^\ast$). Following the tradition of GAN research let us assume that the family is given in the following form\footnote{If $\mathcal{H}\subseteq {\mathcal S}({\mathbb R}^n)$ is not satisfied, then we can choose $\mathcal{H}_\epsilon = \left\{\phi_\theta\ast G_\epsilon^n\vert  \theta\in \Theta\right\}$ for a very small $\epsilon$.}:
\begin{equation}
\begin{split}
\mathcal{H} = \big\{\phi_\theta\vert  \phi_\theta ({\mathbf x}) {\rm \,\, is\,\,pdf\,\,of\,\,random\,\,vector\,\,}g_\theta ({\mathbf z}),  {\mathbf z}\sim p({\mathbf z}), \theta\in \Theta\big\}
\end{split}
\end{equation}
where $\{g_\theta\vert  \theta\in \Theta\}$ is a parameterized family of smooth functions (usually, a neural network) and $p({\mathbf z})$ is some fixed distribution (usually, the Gaussian distribution). Following~\cite{pmlr-v70-arjovsky17a}, we make the assumption~\ref{assumption1}. In a numerical algorithm we need an access to a procedure that samples according to $\phi_{\theta} ({\mathbf x})$, not the function itself.

\begin{assumption}\label{assumption1}
$\Vert g_{\theta'}({\mathbf z}')-g_\theta({\mathbf z})\Vert \leq L(\theta, {\mathbf z})(\Vert \theta'-\theta\Vert +\Vert {\mathbf z}'-{\mathbf z}\Vert )$ where 
\begin{equation}
{\mathbb E}_{{\mathbf z}\sim p({\mathbf z})} L(\theta, {\mathbf z}) < +\infty.
\end{equation}
\end{assumption} 

Thus, instead of solving~\ref{star} we solve:
\begin{equation}
\phi_{t}\longleftarrow  \arg\min\limits_{\phi\in {\mathcal H}} W (\phi, \mu_{\rm{data}})+\lambda \Vert S_\phi - P_{t-1}S_{\phi_{t-1}}\Vert ^2,
\end{equation}
taking into account that $\phi_{t-1}\in {\mathcal H}$.

The Kantorovich-Rubinstein duality theorem gives us that:
\begin{equation}
W (\phi_\theta, \mu_{\rm{data}}) = \max_{f: \Vert f_{{\mathbf x}}\Vert \leq 1} {\mathbb E}_{{\mathbf x}\sim \mu_{\rm{data}}} [f ({\mathbf x})] - {\mathbb E}_{{\mathbf z}\sim p({\mathbf z})} [f (g_\theta ({\mathbf z}))],
\end{equation}
which turns~\ref{star} into the following minimax task:
\begin{equation}\label{minimax-numerical}
\begin{split}
\phi_{t}\longleftarrow  \arg\min\limits_{\phi\in {\mathcal H}} \max\limits_{f: \Vert f_{{\mathbf x}}\Vert \leq 1} {\mathbb E}_{{\mathbf x}\sim \mu_{\rm{data}}} [f ({\mathbf x})] -  {\mathbb E}_{{\mathbf z}\sim p({\mathbf z})} [f (g_\theta ({\mathbf z}))]+ \lambda \Vert S_\phi - P_{t-1}S_{\phi_{t-1}}\Vert ^2.
\end{split}
\end{equation}
In practice, we choose a family of functions $\mathcal{L} = \{f_w\vert  w\in \mathcal{W}\}$ and internal maximization is made over $w\in \mathcal{W}$ with an additional penalty term that penalizes a violation of the Lipschitz condition: $\forall {\mathbf x}: \Vert f_{{\mathbf x}}\Vert \leq 1$.

A family of minimax algorithms for the minimization of $W(\phi_{\theta}, \mu_{{\rm emp}})$ was developed in a series of papers~\cite{pmlr-v70-arjovsky17a,NIPS2017_7159,wei2018improving}. 
The standard minimax scheme that gained popularity in GAN literature iterates two steps: a) $n_{{\rm iter}}$ times make a gradient ascent over $w\in \mathcal{W}$, b) make a gradient descent over $\theta$.  
The task~\ref{minimax-numerical} can be viewed as a Wasserstein GAN with an additional regularization term $\lambda T(\theta)$ where $T(\theta) = \Vert S_{\phi_{\theta}} - P_{t-1}S_{\phi_{\theta_{t-1}}}\Vert ^2$.
To adapt these algorithms to the minimization of our function, we only need to have an unbiased estimator of the gradient
$\frac{\partial T}{\partial \theta}$. This estimator is needed for the generator to make its gradient descent step.
The discriminator's part of the algorithm (in which we maximize over Lipschitz functions $f_w$) can be set in a standard fashion --- we choose~\cite{petzka2018on}'s version, in which the term $\max\{0, \Vert \frac{\partial f_w}{\partial {\mathbf x}}(\xi {\mathbf x} + (1-\xi)g_{\theta} ({\mathbf z}))\Vert -1\}^2$ enforces Lipschitz condition (see step (*) of the Algorithm~\ref{minimax}).


\begin{algorithm}
\begin{algorithmic}
\caption{Numerical algorithm for WD-PCA. We use $M ({\mathbf x}, {\mathbf y}) = e^{-\frac{\Vert {\mathbf x} - {\mathbf y}\Vert ^2}{n}}$ and default values of $\lambda = 10, \Lambda = 100, n_{\rm{critic}} = 5, m = 40, l = 10000n, \alpha = 0.00001, \beta_1 = 0.5, \beta_2 = 0.9$}\label{minimax}

\State $P_0 \longleftarrow  {\mathbf 0}, \theta_0 \longleftarrow  {\mathbf 0}$

\For{$t = 1, \cdots, T$}

\State{Minimax realization of $\min\limits_{\theta} W(\phi_{\theta}, \mu_{{\rm emp}}) + \lambda T(\theta)$ {\bf (*)}:}

\While{$\theta$ has not converged}

\For{$s = 1, ..., n_{\rm{critic}}$}







\State Discriminator updates $w$

\EndFor

\State Sample $\{{\mathbf z}_i\}^m_{i=1}, \{{\mathbf z}'_i\}^m_{i=1}\sim p({\mathbf z})$


\State $L\longleftarrow  -\frac{1}{m} \sum_{i=1}^m f_w (g_{\theta} ({\mathbf z}_i))+\lambda \frac{\sum_{i,j}\Xi (\theta, {\mathbf z}_i, {\mathbf z}'_j)}{m^2}$ ($\Xi$ is defined in~\eqref{bigxi})

\State $\theta \leftarrow {\rm Adam} (\nabla_{\theta} L, \theta, \alpha, \beta_1, \beta_2)$

\EndWhile

\State $\theta_{t}\longleftarrow  \theta$

\State{Realization of step {\bf (**)}:}

\State Sample $\{{\mathbf z}_i\}^l_{i=1}$, $\{{\mathbf z}'_i\}^l_{i=1} \sim p({\mathbf z})$

\State $M_{t} \longleftarrow \sum_{ij} g_{\theta_{t}} ({\mathbf z}_i)g_{\theta_{t}} ({\mathbf z}'_j)^T M( g_{\theta_{t}} ({\mathbf z}_i), g_{\theta_{t}} ({\mathbf z}'_j))$

\State Find $\{{\mathbf v}_i\}^n_1$ s.t. $M_{t}{\mathbf v}_i=\lambda_i{\mathbf v}_i$, $\lambda_1\geq \cdots\geq \lambda_n$

\State $P_{t} \longleftarrow  \sum_{i=1}^k {\mathbf v}_i {\mathbf v}_i^T$

\EndFor

\State \textbf{Output:} ${\mathbf v}_1, \cdots, {\mathbf v}_k$
\end{algorithmic}
\end{algorithm}

\subsection{How to estimate $\frac{\partial T}{\partial \theta}$ and $M_{\phi_{\theta_t}}$?}
Another important aspect of the numerical algorithm is the complexity of estimating the matrix $M_{\phi_{\theta_t}}$ at step (**).  The following theorem shows that we only need to sample ${\mathbf z}\sim p$ a sufficient number of times to estimate $\frac{\partial T}{\partial \theta}$ and $M_{\phi_{\theta_t}}$.

\begin{theorem}\label{addition} 
If $\phi_\theta$ is pdf of the random vector $g_{\theta}({\mathbf z})$, ${\mathbf z}\sim p({\mathbf z})$, then
\begin{equation}
\begin{split}
\frac{\partial T}{\partial \theta} = {\mathbb E}_{{\mathbf z},{\mathbf z}'\sim p} \frac{\partial \Xi(\theta, {\mathbf z}, {\mathbf z}')}{\partial \theta}, \\
M_{\phi_{\theta}} = {\mathbb E}_{{\mathbf z},{\mathbf z}'\sim p} g_{\theta}({\mathbf z}) g_{\theta}({\mathbf z}')^T M(g_{\theta}({\mathbf z}), g_{\theta}({\mathbf z}'))
\end{split}
\end{equation} 
where 
\begin{equation}\label{bigxi}
\begin{split}
\Xi(\theta, {\mathbf z}, {\mathbf z}') = (g_{\theta}({\mathbf z})\cdot g_{\theta}({\mathbf z}')) M(g_{\theta}({\mathbf z}), g_{\theta}({\mathbf z}')) - 
2 (g_{\theta}({\mathbf z})\cdot P_{t-1} g_{\theta_{t-1}}({\mathbf z}')) M(g_{\theta}({\mathbf z}), g_{\theta_{t-1}}({\mathbf z}'))
\end{split}
\end{equation} 
and RHS is well-defined.
\end{theorem}

To prove the theorem we need the following lemma first.
\begin{lemma}\label{addition2}
$\Vert S_\phi - PS_{\psi}\Vert ^2 = {\mathbb E}_{{\mathbf x},{\mathbf y}\sim \phi} ({\mathbf x}\cdot {\mathbf y}) M({\mathbf x}, {\mathbf y}) + {\mathbb E}_{{\mathbf x},{\mathbf y}\sim \psi} ({\mathbf x}\cdot P{\mathbf y}) M({\mathbf x}, {\mathbf y}) - 2{\mathbb E}_{{\mathbf x}\sim \phi,{\mathbf y}\sim \psi} ({\mathbf x}\cdot P{\mathbf y}) M({\mathbf x}, {\mathbf y})$.
\end{lemma}

\begin{proof}[Proof of lemma]
\begin{equation}
\begin{split}
\Vert S_\phi - PS_{\psi}\Vert ^2 =\Vert \sqrt{{\rm O}_M}[{\mathbf x} \phi({\mathbf x})] - P\sqrt{{\rm O}_M}[{\mathbf x} \psi({\mathbf x})]\Vert ^2 = \\
\Vert \sqrt{{\rm O}_M}[{\mathbf x} \phi({\mathbf x}) - P {\mathbf x} \psi({\mathbf x})]\Vert ^2 =  \sum_{i=1}^n \Vert \sqrt{{\rm O}_M}[x_i \phi({\mathbf x}) - (P {\mathbf x})_i \psi({\mathbf x})]\Vert ^2 = \\
\sum_{i=1}^n \langle x_i \phi({\mathbf x})\vert  {\rm O}_M[x_i \phi({\mathbf x})]\rangle +  \langle (P {\mathbf x})_i \psi({\mathbf x})\vert  {\rm O}_M[(P {\mathbf x})_i \psi({\mathbf x})]\rangle - 
2 \langle (P {\mathbf x})_i \psi({\mathbf x})\vert  {\rm O}_M[x_i \phi({\mathbf x})]\rangle 
= \\ {\mathbb E}_{{\mathbf x},{\mathbf y}\sim \phi} ({\mathbf x}\cdot {\mathbf y}) M({\mathbf x}, {\mathbf y}) +   {\mathbb E}_{{\mathbf x},{\mathbf y}\sim \psi} ({\mathbf x}\cdot P{\mathbf y}) M({\mathbf x}, {\mathbf y}) -   2{\mathbb E}_{{\mathbf x}\sim \phi,{\mathbf y}\sim \psi} ({\mathbf x}\cdot P{\mathbf y}) M({\mathbf x}, {\mathbf y}).
\end{split}
\end{equation}
\end{proof}

\begin{proof}[Proof of Theorem~\ref{addition}]
Using lemma~\ref{addition2} we have
\begin{equation}
\begin{split}
T(\theta) = {\mathbb E}_{{\mathbf x},{\mathbf y}\sim \phi_\theta} ({\mathbf x}\cdot {\mathbf y}) M({\mathbf x}, {\mathbf y}) +  {\mathbb E}_{{\mathbf x},{\mathbf y}\sim \phi_{\theta_{t-1}}} ({\mathbf x}\cdot P_{t-1}{\mathbf y}) M({\mathbf x}, {\mathbf y}) - \\
2{\mathbb E}_{{\mathbf x}\sim \phi_{\theta},{\mathbf y}\sim \phi_{\theta_{t-1}}} ({\mathbf x}\cdot P_{t-1}{\mathbf y}) M({\mathbf x}, {\mathbf y}) =  {\mathbb E}_{{\mathbf z},{\mathbf z}'\sim p} (g_\theta ({\mathbf z})\cdot g_\theta ({\mathbf z}')) M(g_\theta ({\mathbf z}), g_\theta ({\mathbf z}')) + \\
{\mathbb E}_{{\mathbf z},{\mathbf z}'\sim p} (g_{\theta_{t-1}}({\mathbf z})\cdot P_{t-1}g_{\theta_{t-1}}({\mathbf z}')) M(g_{\theta_{t-1}}({\mathbf z}), g_{\theta_{t-1}}({\mathbf z}')) - \\ 2{\mathbb E}_{{\mathbf z},{\mathbf z}'\sim p} (g_{\theta}({\mathbf z})\cdot P_{t-1}g_{\theta_{t-1}}({\mathbf z}')) M(g_{\theta}({\mathbf z}), g_{\theta_{t-1}}({\mathbf z}')).
\end{split}
\end{equation}
The second term does not depend on $\theta$. Therefore,
\begin{equation}
\frac{\partial T}{\partial \theta} = \frac{\partial}{\partial \theta} {\mathbb E}_{{\mathbf z},{\mathbf z}'\sim p}  \Xi(\theta, {\mathbf z}, {\mathbf z}'),
\end{equation}
where
\begin{equation}
\begin{split}
\Xi (\theta, {\mathbf z}, {\mathbf z}') = (g_\theta ({\mathbf z})\cdot g_\theta ({\mathbf z}')) M(g_\theta ({\mathbf z}), g_\theta ({\mathbf z}')) - 2 (g_{\theta}({\mathbf z})\cdot P_{t-1}g_{\theta_{t-1}}({\mathbf z}')) M(g_{\theta}({\mathbf z}), g_{\theta_{t-1}}({\mathbf z}')).
\end{split}
\end{equation}

If ${\mathbb E}_{{\mathbf z},{\mathbf z}'\sim p}  \frac{\partial \Xi(\theta, {\mathbf z}, {\mathbf z}')}{\partial \theta}$ is well-defined (the proof of sufficiency of that condition is similar to the proof of Theorem 3 from~\cite{pmlr-v70-arjovsky17a}), then, using Leibniz integral rule, we obtain
\begin{equation}
\frac{\partial}{\partial \theta} {\mathbb E}_{{\mathbf z},{\mathbf z}'\sim p}  \Xi (\theta, {\mathbf z}, {\mathbf z}') =  {\mathbb E}_{{\mathbf z},{\mathbf z}'\sim p}  \frac{\partial \Xi(\theta, {\mathbf z}, {\mathbf z}')}{\partial \theta}.
\end{equation}

The fact that
\begin{equation}
M_{\phi_{\theta}} = {\mathbb E}_{{\mathbf z},{\mathbf z}'\sim p} g_{\theta}({\mathbf z}) g_{\theta}({\mathbf z}')^T M(g_{\theta}({\mathbf z}), g_{\theta}({\mathbf z}'))
\end{equation}
is obvious from the definition $M_{\phi_{\theta}} = {\mathbb E}_{{\mathbf x},{\mathbf y}\sim \phi_\theta } {\mathbf x} {\mathbf y}^T M({\mathbf x},{\mathbf y})$.
\end{proof}


\subsubsection{Definition of ${\mathcal H}$}
Specifically, for robust PCA/outlier pursuit applications, we define $\phi_{\theta} ({\mathbf x})$ as a probability density function   of the random vector ${\mathbf a} + {\mathbf b}$, where ${\mathbf a}$, ${\mathbf b}$ are independent and ${\mathbf a}$ is the $i$-th column of matrix $\theta_1\in {\mathbb R}^{n\times N}$ (where $i\sim {\mathcal U}(1,N)$ is sampled uniformly from $\{1, \cdots, N\}$), 
${\mathbf b} = g_{\theta_2} ({\mathbf c})$, ${\mathbf c}\sim {\mathcal N} ({\mathbf 0}, I_n)$ and $g_{\theta_2}: {\mathbb R}^n \rightarrow {\mathbb R}^n$ is a neural network with weights $\theta_2$. Thus, $\theta = (\theta_1, \theta_2)$. 
It can be checked that ${\mathcal H}$, defined in this way, satisfies the Assumption~\ref{assumption1}. We specifically introduce the random vector ${\mathbf a}$ here because, according to Theorem~\ref{transport}, the ultimate solution of the problem corresponds to $\theta_1 = Y$ and ${\mathbf b} = {\mathbf 0}$. This guarantees that the solution is approachable from set ${\mathcal H}$.

\section{A numerical alternating scheme for SDR-ORF}\label{numerical-sdr}
For a binary classification case, given a labeled dataset $\{({\mathbf x}_i, y_i)\}_{i=1}^N$, ${\mathbf x}_i\in {\mathbb R}^n, y_i \in {\mathcal C}$, ${\mathcal C} = \{0, 1\}$ we formulate the sufficient dimension reduction problem as the minimization task 
\begin{equation}
J(f) = {\mathbb E}_{({\mathbf z}, c)\sim \mu_{\rm{data}}, \text{\boldmath$\epsilon$} \sim N({\mathbf 0}, \upsilon^2 I_n)} L(c, f({\mathbf z}+\text{\boldmath$\epsilon$})) \rightarrow \min\limits_{f\in \mathcal{F}_k},
\end{equation}
where $L(c, y) = -c \log(y) - (1-c)\log (1-y)$. 

We apply the alternating scheme in the dual space (Algorithm~\ref{alternate-mod2}) to this task. We set $M({\mathbf x},{\mathbf y}) = \zeta ({\mathbf x} - {\mathbf y})$, where $\hat\zeta$ is a strictly positive probability density function. A numerical version of the scheme is given below (Algorithm~\ref{below}). 

At every iteration $t=1, \cdots, T$ of the Algorithm~\ref{alternate-mod2} we solve the task (in our case $\tilde{I} = J$)
\begin{equation}{\hat \phi}_{t}\leftarrow  \arg\min\limits_{{\hat \phi}} \tilde{I}({\hat \phi})+\tilde{\lambda}  \Vert \,\,\Vert \frac{\partial{\hat \phi}}{\partial {\mathbf x}} - P_{t-1}\frac{\partial{\hat \phi}_{t-1}}{\partial {\mathbf x}}\Vert_2 \,\,\Vert ^2_{L_{2,\hat\zeta}({\mathbb R}^{n})}.
\end{equation}
In a numerical version of the algorithm we assume that ${\hat \phi}$ is given as a neural network $f_\theta$, i.e. our task becomes
\begin{equation}
\theta_{t}\leftarrow  \arg\min\limits_{\theta} J (f_\theta)+\tilde{\lambda}  {\mathbb E}_{\text{\boldmath$\xi$}\sim \hat\zeta}\Vert \frac{\partial f_\theta}{\partial {\mathbf x}} (\text{\boldmath$\xi$}) - P_{t-1}\frac{\partial f_{\theta_{t-1}}}{\partial {\mathbf x}} (\text{\boldmath$\xi$})\Vert ^2.
\end{equation}
The gradient of the function $\Phi(\theta) = J(f_\theta)+\tilde{\lambda}  {\mathbb E}_{\text{\boldmath$\xi$}\sim \hat\zeta}\Vert \frac{\partial f_\theta}{\partial {\mathbf x}} (\text{\boldmath$\xi$}) - P_{t-1}\frac{\partial f_{\theta_{t-1}}}{\partial {\mathbf x}} (\text{\boldmath$\xi$})\Vert ^2$ equals
\begin{equation}
\begin{split}
\frac{\partial \Phi(\theta)}{\partial \theta} = {\mathbb E}_{({\mathbf z}, c)\sim P_{\rm{data}}, \text{\boldmath$\epsilon$} \sim N({\mathbf 0}, \upsilon^2 I_n)} \frac{\partial }{\partial \theta} L(c, f_\theta({\mathbf z}+\text{\boldmath$\epsilon$}))   +
\tilde{\lambda}  {\mathbb E}_{\text{\boldmath$\xi$}\sim \hat\zeta}\frac{\partial }{\partial \theta} \Vert \frac{\partial f_\theta}{\partial {\mathbf x}} (\text{\boldmath$\xi$}) - P_{t-1}\frac{\partial f_{\theta_{t-1}}}{\partial {\mathbf x}} (\text{\boldmath$\xi$})\Vert ^2.
\end{split}
\end{equation}
That is why $\nabla_{\theta} L$ (given to Adam optimizer in the gradient descent loop) in the Algorithm~\ref{below} is an unbiased estimator of $\frac{\partial \Phi(\theta)}{\partial \theta}$. Thus, in the ``while loop'' we find optimal ${\hat \phi}_{t} = 
f_{\theta_t}$.

According to Algorithm~\ref{alternate-mod2}, the next goal is to estimate $$M_t = \begin{bmatrix} {\rm Re\,}
\langle \frac{\partial{\hat \phi_t}}{\partial x_i}, \frac{\partial{\hat \phi_t}}{\partial x_j}\rangle_{L_{2,\hat\zeta}({\mathbb R}^{n})}
\end{bmatrix}.$$ It is easy to see that
\begin{equation}
M_t  = {\mathbb E}_{\text{\boldmath$\chi$}\sim \hat\zeta}
\frac{\partial{\hat \phi_t}}{\partial {\mathbf x}} (\text{\boldmath$\chi$}) \frac{\partial{\hat \phi_t}}{\partial {\mathbf x}} (\text{\boldmath$\chi$})^T = {\mathbb E}_{\text{\boldmath$\chi$}\sim \hat\zeta}
\frac{\partial f_{\theta_t}}{\partial {\mathbf x}} (\text{\boldmath$\chi$}) \frac{\partial f_{\theta_t}}{\partial {\mathbf x}} (\text{\boldmath$\chi$})^T.
\end{equation}
From the last we see that the matrix $M_t$ can be estimated by sampling $\text{\boldmath$\chi$} \sim \hat{\zeta}$ a sufficient number of times (the parameter $l$ in our algorithm). All the rest is identical to Algorithm~\ref{alternate-mod2}.

\begin{algorithm}
\begin{algorithmic}
\caption{The numerical alternating scheme for SDR-ORF. We use $\upsilon=1.0$, $\hat{\zeta} ({\mathbf x} ) = G_{0.8}^n ({\mathbf x} )$ and default values of $\tilde{\lambda} = 10, m \approx 50, m' = 100, l = 30000, \alpha = 0.0001, \beta_1 = 0.5, \beta_2 = 0.9$}\label{below}

\State $P_0 \longleftarrow  {\mathbf 0}, \theta_0 \longleftarrow  {\mathbf 0}$

\For{$t = 1, \cdots, T$}

\While{$\theta$ has not converged}

\State Sample $\{({\mathbf z}_i, c_i)\}_{i=1}^m \sim P_{\rm{data}}$

\State Sample $\{\text{\boldmath$\epsilon$}_i\}_{i=1}^m \sim N({\mathbf 0}, \upsilon^2 I_n)$

\State Sample $\{\text{\boldmath$\xi$}_i\}_{i=1}^{m'} \sim \hat{\zeta}$

\State $L\longleftarrow \frac{1}{m}\sum_{i=1}^m L(c_i, f_\theta ({\mathbf z_i}+ \text{\boldmath$\epsilon$}_i)) + \frac{\tilde{\lambda}}{m'}\sum_{i=1}^{m'} \Vert \frac{\partial f_\theta (\text{\boldmath$\xi$}_i))}{\partial {\mathbf x}} - P_{t-1} \frac{\partial f_{\theta_{t-1}} (\text{\boldmath$\xi$}_i))}{\partial {\mathbf x}}\Vert ^2$

\State $\theta \longleftarrow {\rm Adam} (\nabla_{\theta} L, \theta, \alpha, \beta_1, \beta_2)$

\EndWhile

\State $\theta_t \longleftarrow \theta$

\State Sample $\{\text{\boldmath$\chi$}_i\}_{i=1}^{l} \sim \hat{\zeta}$

\State Calculate $M_{t} = \frac{1}{l}\sum_{i=1}^l \frac{\partial f_{\theta_{t}} (\text{\boldmath$\chi$}_i))}{\partial {\mathbf x}} \frac{\partial f_{\theta_{t}} (\text{\boldmath$\chi$}_i))}{\partial {\mathbf x}}^T$

\State Find $\{{\mathbf v}_i\}^n_1$ s.t. $M_{t}{\mathbf v}_i=\lambda_i{\mathbf v}_i$, $\lambda_1\geq \cdots\geq \lambda_n$

\State $P_{t} \longleftarrow  \sum_{i=1}^k {\mathbf v}_i {\mathbf v}_i^T$

\EndFor

\State \textbf{Output:} ${\mathbf v}_1, \cdots, {\mathbf v}_k$

\end{algorithmic}
\end{algorithm}

The regression version of the algorithm can be obtained by setting $L(c,c') = (c-c')^2$.
Implementations for different databases can be found at \href{https://github.com/k-nic/Alternating-Scheme}{github}.

\end{document}